\documentclass[a4paper]{amsart}
\usepackage{amsmath,amsthm,amssymb,latexsym,epic,bbm,comment,color}
\usepackage{graphicx,enumerate,stmaryrd}
\usepackage[all,2cell]{xy}
\xyoption{2cell}

\usepackage{mathtools}

\usepackage{color}
\definecolor{purple}{RGB}{128,0,128}

\def\H{{\mathbb H}}
\def\ov{\overline}

\def\ii{\textbf{\itshape i}}
\def\jj{\textbf{\itshape j}}
\def\kk{\textbf{\itshape k}}

\def\Stab{{\rm Stab}}
\newcommand{\ps}{{\Psi^{\rm st}_{-1}}}

\newcommand{\g}{{\mathfrak g}}
\newcommand{\Lie}{{\rm Lie}}
\newcommand{\PiG}{{\Pi\!\Gr}}
\newcommand{\id}{{\rm id}}

\usepackage{dsfont}
\renewcommand{\mathbb}{\mathds}

\newcommand{\Z}{\mathbb Z}
\newcommand{\C}{\mathbb C}
\newcommand{\R}{{\mathbb R}}

\newcommand{\E}{\mathbb E}
\newcommand{\gr}{\mathrm{gr}}

\newcommand{\Gr}{\operatorname{Gr}}
\newcommand{\GL}{\operatorname{GL}}
\newcommand{\Q}{\operatorname{Q}}

\newcommand{\PGL}{\operatorname{PGL}}

\newcommand{\ord}{\textsf{ord}}

\newtheorem{theorem}{Theorem}

\newtheorem{corollary}[theorem]{Corollary}
\newtheorem{proposition}[theorem]{Proposition}

\newtheorem*{theorem*}{Theorem}

\theoremstyle{definition}

\newtheorem{remark}[theorem]{Remark}

\usepackage[all]{xy}
\usepackage[active]{srcltx}
\usepackage[parfill]{parskip}

\newcommand{\mcJ}{\mathcal J}
\newcommand{\mcM}{\mathcal M}
\newcommand{\mcN}{\mathcal N}
\newcommand{\mcO}{\mathcal O}
\newcommand{\mcE}{\mathcal E}
\newcommand{\mcH}{\mathcal H}
\newcommand{\al}{\alpha}

\newcommand{\cG}{\sc\mbox{G}\hspace{1.0pt}}

\theoremstyle{plain}
\newtheorem{prop}{Proposition}[section]
\newtheorem{lem}[prop]{Lemma}
\newtheorem{thm}[prop]{Theorem}
\newtheorem{cor}[prop]{Corollary}
\theoremstyle{definition}
\newtheorem{subsec}[prop]{}
\newtheorem{rem}[prop]{Remark}

\newcommand{\M}{{\mathcal M}}

\newcommand{\into}{\hookrightarrow}
\newcommand{\isoto}{\overset{\sim}{\to}}

\newcommand{\labelto}[1]{\xrightarrow{\makebox[1.5em]{\scriptsize ${#1}$}}}

\newcommand{\hs}{\kern 0.8pt}
\newcommand{\hssh}{\kern 1.2pt}
\newcommand{\hshs}{\kern 1.6pt}
\newcommand{\hssss}{\kern 2.0pt}

\newcommand{\hm}{\kern -0.8pt}
\newcommand{\hmm}{\kern -1.2pt}

\newcommand{\mO}{{\mathcal O}}
\newcommand{\uprho}{\hs^\rho\hm}

\newcommand{\Aut}{{\rm Aut}}
\newcommand{\G}{{\Gamma}}

\newcommand{\SmallMatrix}[1]{\text{\tiny\arraycolsep=0.4\arraycolsep\ensuremath
		{\begin{pmatrix}#1\end{pmatrix}}}}

\newcommand{\Mat}[1]{\text{\SMALL\arraycolsep=0.4\arraycolsep\ensuremath
		{\begin{pmatrix}#1\end{pmatrix}}}}

\def\H{{\mathbb H}}
\def\ov{\overline}

\def\ii{\textbf{\itshape i}}
\def\jj{\textbf{\itshape j}}
\def\kk{\textbf{\itshape k}}

\def\Stab{{\rm Stab}}

\begin{document}
\title[$\Pi$-symmetric super-Grassmannian]
{Automorphisms and real structures for\\ a $\Pi$-symmetric super-Grassmannian}

\author{Elizaveta Vishnyakova\\
{\Tiny appendix by}\\
Mikhail Borovoi}

\begin{abstract}
	Any complex-analytic vector bundle $\E$ admits naturally defined homotheties $\phi_{\al}$, $\al\in \C^*$,  i.e. $\phi_{\al}$ is the multiplication of a local section by a complex number $\al$.
We investigate the question when such  automorphisms can be lifted to a non-split supermanifold corresponding to $\E$. Further, we compute the automorphism supergroup of a $\Pi$-symmetric super-Grassmannian $\Pi\!\Gr_{n,k}$, and, using  Galois cohomology, we classify the real structures on $\Pi\!\Gr_{n,k}$ and compute the corresponding supermanifolds of real points.
\end{abstract}

\date{\today}

\maketitle

\setcounter{tocdepth}{1}
\tableofcontents

\section{Introduction}
  Let $\E$ be a complex-analytic vector bundle over a complex-analytic manifold $M$.
  There are natural homotheties $\phi_{\al}$, $\al\in \C^*$, defined on local sections as the multiplication by a complex number  $\al\ne 0$.
   Any    automorphism $\phi_{\al}: \E\to \E$ may be naturally extended to an automorphism $\wedge \phi_{\al}$ of $\bigwedge\E$.
   Let $\mcE$ be the locally free sheaf corresponding to $\E$. Then the ringed space $(M,\bigwedge\mcE)$ is a split supermanifold
   equipped with the  supermanifold automorphisms $(id,\wedge \phi_{\al})$, $\al\in \C^*$.
   Let $\mcM$ be any non-split supermanifold with retract $(M,\bigwedge\mcE)$.
   We investigate the question whether the automorphism $\wedge \phi_{\al}$ can be lifted to $\mcM$.
   We show that this question is related to the notion of the order of the supermanifold $\mcM$ introduced in \cite{Rothstein};
   see Section \ref{sec Order of a supermanifold}.

  Let $\M=\Pi\!\Gr_{n,k}$ be a $\Pi$-symmetric super-Grassmannian; see Section \ref{sec charts on Gr} for the definition.
   We use obtained results to compute the automorphism group $\operatorname{Aut} \mathcal M$ and the automorphism supergroup, given in terms of a super-Harish-Chandra pair.

  \begin{theorem*}[Theorem \ref{t:Aut}]
  	
  		{\bf (1)} If $\mathcal M = \Pi\!\Gr_{n,k}$, where $n\ne 2k$, then
  	$$
  	\operatorname{Aut} \mathcal M\simeq \PGL_n(\mathbb C) \times \{\id, \Psi^{st}_{-1} \} .
  	$$

  	The automorphism supergroup is given by the Harish-Chandra pair
  	$$
  	( \PGL_n(\mathbb C) \times \{\id, \Psi^{st}_{-1} \},  \mathfrak{q}_{n}(\mathbb C)/\langle E_{2n}\rangle).
  	$$
  	
  		{\bf (2)} If $\mathcal M = \Pi\!\Gr_{2k,k}$, where $k\geq 2$, then
  	$$
  	\operatorname{Aut} \mathcal M\simeq  \PGL_{2k}(\mathbb C) \rtimes  \{\id, \Theta, \Psi^{st}_{-1}, \Psi^{st}_{-1}\circ \Theta \},
  	$$
  	where $\Theta^2 = \Psi^{st}_{-1}$, $\Psi^{st}_{-1}$ is a central element of $\Aut\,\M$, and $\Theta \circ g\circ \Theta^{-1} = (g^t)^{-1}$ for $g\in  \PGL_{2k}(\mathbb C)$.
  	
  	The automorphism supergroup is given by the Harish-Chandra pair
  	$$
  	(\PGL_{2k}(\mathbb C) \rtimes  \{\id, \Psi^{st}_{-1}, \Theta, \Psi^{st}_{-1}\circ \Theta \},  \mathfrak{q}_{2k}(\mathbb C)/\langle E_{4k}\rangle),
  	$$
  	where $\Theta \circ C\circ \Theta^{-1} = - C^{t_i}$ for $C\in  \mathfrak{q}_{2k}(\mathbb C)/\langle E_{4k}\rangle$ and $ \Psi^{st}_{-1}\circ C \circ  (\Psi^{st}_{-1})^{-1} = (-1)^{\tilde{C}} C$,  where $\tilde C\in\Z/2\Z$ is the  parity of $C$.

  	{\bf (3)} If $\mathcal M = \Pi\!\Gr_{2,1}$, then
  	$$
  		\operatorname{Aut} \mathcal M\simeq \PGL_{2}(\mathbb C) \times \mathbb C^*.
  	$$
  	The automorphism supergroup is given by the Harish-Chandra pair
  	$$
  	( \PGL_{2}(\mathbb C) \times \mathbb C^*,  	\mathfrak g \rtimes \langle z\rangle).
  	$$
  	Here $\mathfrak g$ is a $\Z$-graded Lie superalgebra described in Theorem \ref{teor vector fields on supergrassmannians}, $z$ is the grading operator of $\mathfrak{g}$.  The action of $\PGL_{2}(\mathbb C) \times \mathbb C^*$ on $z$ is trivial, and $\phi_{\al}\in \C^*$ multiplies $X\in \mathfrak v(\Pi\!\Gr_{2,1})_k$ by $\al^k$.
\end{theorem*}
Here $\ps=(\id,\psi^{st}_{-1})\in \operatorname{Aut} \mathcal M$, where $\psi^{st}_{-1}$
is an automorphism of the structure sheaf $\mcO$ of $\mcM$
defined by $\psi^{st}_{-1}(f) = (-1)^{\tilde f} f$ for a homogeneous local section  $f$ of $\mcO$, where we denoted by $\tilde f\in\Z/2\Z$  the  parity of $f$. We denote by $C^{t_i}$ the $i$-transposition of the matrix $C$, see (\ref{eq i transposition}).  The automorphism $\Theta$ is constructed in Section \ref{sec construction of Theta}. We denoted by $g^t$ the transpose of $g$.

We classify the real structures on a $\Pi$-symmetric super-Grassmannian $\Pi\!\Gr_{n,k}$ using Galois cohomology.

\begin{theorem*}[Theorem \ref{c:Pi}]
	The number  of the equivalence classes of real structures $\mu$ on $\mcM$, and  representatives of these classes, are given in the list below:
\begin{enumerate}
	\item[\rm (i)] If $n$ is odd, then there are two equivalence classes with representatives
	$$
	\mu^o, \quad (1,\ps)\circ\mu^o.
	$$
	\item[\rm (ii)] If $n$ is even and $n\neq 2k$,
	then there are four equivalence classes with representatives
	$$
	\mu^o,\quad (1,\ps)\circ\mu^o, \quad  (c_J,1)\circ\mu^o, \quad (c_J,\ps)\circ\mu^o.
	$$
	
	\item[\rm (iii)] If $n=2k\ge 4$, then there are $k+3$ equivalence classes with representatives
	$$
	\mu^o,\quad (c_J,1)\circ\mu^o, \quad  (c_r,\Theta)\circ\mu^o, \,\, r= 0,\ldots, k.
	$$

	\item[\rm (iv)] 	If $(n,k)= (2,1)$, then there are two equivalence classes with representatives
	$$
	\mu^o,\quad  (c_J,1)\circ\mu^o.
	$$
\end{enumerate}
Here  $\mu^o$ denotes the standard real structure on $\M=\PiG_{n,k}$, see Section \ref{ss:real-structures}.
Moreover, $c_J\in\PGL_n(\C)$ and $c_r\in\PGL_{2k}(\C)$ for $r= 0,\ldots, k$ are certain elements
constructed  in Proposition \ref{p:H1} and Subsection \ref{ss:cp}, respectively.
\end{theorem*}

Further, we describe the corresponding real subsupermanifolds when they exist.
 Let $\mu$ be a real structure on $\mcM=\PiG_{n,k}$, and assume that the  set of fixed points
$ M^{\mu_0}$ is non-empty. Consider the ringed space $\M^{\mu}:= (M^{\mu_0}, \mcO^{\mu^*})$
where $\mcO^{\mu^*}$ is the sheaf of fixed points of $\mu^*$ over $M^{\mu}$.
Then $\M^{\mu}$ is a real supermanifold.
We describe this supermanifold in Theorem \ref{theor real main}.

 \textbf{Acknowledgments:}  The author was partially  supported by
 Coordena\c{c}\~{a}o de Aperfei\c{c}oamento de Pessoal de N\'{\i}vel Superior - Brasil (CAPES) -- Finance Code
 001, (Capes-Humboldt Research Fellowship),  by FAPEMIG, grant APQ-01999-18, Rede Mineira de Matemática-RMMAT-MG, Projeto RED-00133-21. We thank Peter \linebreak Littelmann for hospitality and the wonderful working atmosphere at the University of Cologne and we thank Dmitri Akhiezer for helpful comments. We also thank  Mikhail Borovoi for suggesting to write this paper and for writing the appendix.

\section{Preliminaries}

\subsection{Supermanifolds}
This paper is devoted to the study of the automorphism supergroup of a $\Pi$-symmetric super-Grassmannian $\Pi\!\Gr_{n,k}$, and to a classification of real structures on $\Pi\!\Gr_{n,k}$.  Details about the theory of supermanifolds can be found in \cite{Bern,Leites,BLMS}. As usual, the superspace $\mathbb C^{n|m}:= \mathbb C^{n}\oplus \mathbb C^{m}$ is a $\Z_2$-graded vector space over $\mathbb C$ of dimension $n|m$.   A {\it superdomain in $\mathbb{C}^{n|m}$} is a ringed space $\mathcal U:=(U,\mathcal F_U\otimes \bigwedge (\mathbb C^{m})^*)$, where $U\subset \mathbb C^{n}$ is an open set and $\mathcal F_U$ is the sheaf of holomorphic functions on $U$. If $(x_a)$ is a system of coordinates in $U$ and $(\xi_b)$ is a basis in $(\mathbb C^{m})^*$ we call $(x_a,\xi_b)$ a system of coordinates in $\mathcal U$. Here $(x_a)$ are called even coordinates of $\mathcal U$, while $(\xi_b)$ are called odd ones.
A {\it supermanifold} $\mcM = (M,\mathcal{O})$ of dimension $n|m$ is a $\mathbb{Z}_2$-graded
ringed space that is locally isomorphic to a super\-domain in
$\mathbb{C}^{n|m}$. Here the underlying space $M$ is a complex-analytic manifold.  A
{\it morphism} $F:(M,\mcO_{\mcM}) \to (N,\mcO_{\mcN})$ of two
supermanifolds is, by definition, a morphism of the corresponding $\mathbb{Z}_2$-graded locally ringed spaces.
In more details, $F = (F_{0},F^*)$ is a pair, where $F_{0}:M\to N$ is a holomorphic map and
$F^*: \mathcal{O}_{\mathcal N}\to (F_{0})_*(\mathcal{O}_{\mathcal M})$
is a homomorphism of sheaves of $\mathbb{Z}_2$-graded local superalgebras.
We see that the morphism $F$ is even, that is, $F$ preserves the $\mathbb{Z}_2$-gradings of the sheaves.
 A morphism $F: \mcM\to \mcM$ is called an {\it automorphism of $\mcM$} if $F$ is an automorphism of the corresponding $\mathbb{Z}_2$-graded ringed spaces. The automorphisms of $\mcM$ form a group, which we denote by  $\operatorname{Aut} \mathcal M$. Note that in this paper we also consider the automorphism supergroup, see a definition below.

A supermanifold $\mcM=(M,\mcO)$  is called {\it split}, if its structure sheaf is isomorphic to $\bigwedge \mathcal E$, where $\mathcal E$ is a sheaf of sections of a holomorphic  vector bundle $\mathbb E$ over $M$.  In this case the structure sheaf of $\mcM$ is $\mathbb Z$-graded, not only $\Z_2$-graded. There is  a functor assigning to any supermanifold a split supermanifold.  Let us briefly remind this construction. Let $\mcM=(M,\mathcal O)$ be a supermanifold. Consider the following filtration in $\mathcal O$
$$
\mathcal O = \mathcal J^0 \supset \mathcal J \supset \mathcal J^2 \supset\cdots \supset \mathcal J^p \supset\cdots,
$$
where $\mathcal J$ is the subsheaf of ideals in $\mcO$ locally generated by odd elements of $\mcO$. We define
$$
\mathrm{gr} \mathcal M := (M,\mathrm{gr}\mathcal O),\quad \text{where} \quad
\mathrm{gr}\mathcal O: = \bigoplus_{p \geq 0} \mathcal J^p/\mathcal J^{p+1}.
$$
The supermanifold $\mathrm{gr} \mathcal M$ is split and it is called the {\it retract} of $\mcM$. The underlying space of $\mathrm{gr} \mathcal M$ is the complex-analytic manifold $(M,\mathcal O/\mathcal J)$, which coincides with $M$. The structure sheaf $\mathrm{gr}\mathcal O$ is isomorphic to $\bigwedge \mathcal E$, where $\mathcal E= \mathcal J/\mathcal J^{2}$ is a locally free sheaf of $\mathcal O/\mathcal J$-modules on $M$.

Further let $\mcM =(M,\mcO_{\mcM})$ and $\mathcal{N}= (N,\mcO_{\mcN})$ be two supermanifolds, $\mathcal J_{\mcM}$ and  $\mathcal J_{\mcN}$ be the subsheaves of ideals in $\mcO_{\mcM}$ and $\mcO_{\mcN}$, which are locally generated by odd elements in $\mcO_{\mcM}$ and in $\mcO_{\mcN}$, respectively.   Any morphism $F:\mcM \to \mathcal{N}$ preserves these shaves of ideals, that is $F^*(\mcJ_{\mcN}) \subset (F_{0})_*(\mathcal{J}_{\mathcal M})$, and more generally $F^*(\mcJ^p_{\mcN}) \subset (F_{0})_*(\mathcal{J}^p_{\mathcal M})$ for any $p$. Therefore $F$ induces naturally a morphism $\mathrm{gr}(F): \mathrm{gr} \mathcal M\to \mathrm{gr} \mathcal N$. Summing up, the functor $\gr$ is defined.

\subsection{A classification theorem for supermanifolds}\label{sec A classification theorem}

Let $\mathcal M=(M,\mathcal O)$ be a (non-split) supermanifold. Recall that we denoted by $\operatorname{Aut} \mathcal M$ the group of all (even) automorphisms of $\mathcal M$. Denote by $\mathcal{A}ut \mathcal O$ the sheaf of automorphisms of $\mcO$. Consider the following subsheaf of $\mathcal{A}ut \mathcal O$
\begin{align*}
\mathcal{A}ut_{(2)} \mathcal O := \{F\in \mathcal{A}ut \mathcal O\,\,|\,\,\, \gr (F) =id\}.
\end{align*}
This sheaf plays an important role in the classification of supermanifolds, see below.  The sheaf
$\mathcal{A}ut\mathcal{O}$ has the following filtration
\begin{equation*}
\mathcal{A}ut \mathcal{O}=\mathcal{A}ut_{(0)}\mathcal{O} \supset
\mathcal{A}ut_{(2)}\mathcal{O}\supset \ldots \supset
\mathcal{A}ut_{(2p)}\mathcal{O} \supset \ldots ,
\end{equation*}
where
$$
\mathcal{A}ut_{(2p)}\mathcal{O} = \{a\in\mathcal{A}ut\mathcal{O}\mid a(u)\equiv u\mod \mathcal{J}^{2p} \,\, \text{for any}\,\,\, u\in \mcO\}.
$$
Recall that $\mathcal J$ is the subsheaf of ideals generated by odd elements in $\mathcal O$. Let $\E$ be the bundle corresponding to the locally free sheaf $\mcE=\mcJ/\mcJ^2$ and let $\operatorname{Aut} \E$ be the group of all automorphisms of  $\E$.  Clearly, any automorphism of $\E$ gives rise to an automorphism of $\gr \mcM$, and thus we get a natural action of the group $\operatorname{Aut} \E$ on the sheaf $\mathcal{A}ut (\gr\mathcal{O})$ by $Int:
(a,\delta)\mapsto a\circ \delta\circ a^{-1}$, where
$\delta\in\mathcal{A}ut (\gr\mathcal{O})$ and $a\in
\operatorname{Aut} \E$. Clearly, the group
$\operatorname{Aut} \E$ leaves invariant the
subsheaves $\mathcal{A}ut_{(2p)} \gr\mathcal{O}$. Hence
$\operatorname{Aut} \E$ acts on the cohomology sets
$H^1(M,\mathcal{A}ut_{(2p)} \gr\mathcal{O})$.  The unit element
$\epsilon\in H^1(M,\mathcal{A}ut_{(2p)} \gr\mathcal{O})$ is fixed under this action. We denote by $H^1(M,\mathcal{A}ut_{(2p)}\gr\mathcal{O})/
\operatorname{Aut} \E$ the set of orbits of the action in $H^1(M,\mathcal{A}ut_{(2p)}\gr\mathcal{O})$ induced by $Int$.

Denote by $[\mcM]$ the class of supermanifolds which are isomorphic to $\mcM= (M,\mcO)$.
(Here we consider complex-analytic supermanifolds up to isomorphisms inducing the identical isomorphism of the base spaces.) The following theorem was proved in \cite{Green}.

\begin{theorem}[{\bf Green}]\label{Theor_Green}
	Let $(M,\bigwedge \mcE)$ be a fixed split supermanifold. Then
	$$
	\begin{array}{c}
	\{[\mcM ] \mid \gr\mathcal{O} \simeq\bigwedge \mcE\}
	\stackrel{1:1}{\longleftrightarrow}
	H^1(M,\mathcal{A}ut_{(2)}\gr\mathcal{O})/
	\operatorname{Aut} \E.
	\end{array}
	$$
	The split supermanifold $(M,\bigwedge \mcE)$ corresponds to the fixed point $\epsilon$.
\end{theorem}

\subsection{Tangent sheaf of $\mcM$ and $\gr \mcM$}
Let again $\mathcal M=(M,\mathcal O)$ be a (non-split) supermanifold. The {\it tangent sheaf} of a supermanifold $\mcM$ is by definition the sheaf $\mathcal T = \mathcal{D}er\mcO$ of
derivations of the structure sheaf $\mcO$. Sections of the sheaf $\mathcal T$ are called {\it holomorphic vector fields} on $\mcM$.
The vector superspace $\mathfrak v(\mcM) = H^0(M, \mathcal T)$ of all holomorphic vector fields is a complex Lie superalgebra
with the bracket
$$
[X,Y]= X\circ Y- (-1)^{\tilde X\tilde Y} Y\circ X,\quad X,Y\in \mathfrak v(\mcM),
$$
where $\tilde Z$ is the parity of an element $Z\in \mathfrak v(\mcM)$.
 The Lie superalgebra $\mathfrak v(\mcM)$ is finite dimensional if $M$ is compact.

Let $\dim \mcM=n|m$.  The tangent sheaf $\mathcal T$ possesses the following filtration:
$$
\mathcal T=\mathcal T_{(-1)} \supset \mathcal T_{(0)}  \supset \mathcal T_{(1)}  \supset \cdots \supset \mathcal T_{(m)} \supset \mathcal T_{(m+1)}=0,
$$
where
$$
\mathcal T_{(p)} = \{ v\in \mathcal T \,\,|\,\, v(\mcO) \subset \mcJ^p,\,\, v(\mcJ) \subset \mcJ^{p+1}  \},\quad p\geq 0.
$$

Denote by $\mathcal T_{\gr}$ the tangent sheaf of the retract $\gr \mcM$. Since the structure sheaf $\gr \mcO$ of $\gr \mcM$ is $\Z$-graded, the sheaf  $\mathcal T_{\gr}$ has the following induced $\Z$-grading
$$
\mathcal T_{\gr} = \bigoplus_{p\geq -1} (\mathcal T_{\gr})_{p},
$$
 where
 $$
 (\mathcal T_{\gr})_{p}= \{\, v\in \mathcal T_{\gr} \,\,|\,\, v(\gr\mcO_q) \subset \gr\mcO_{q+p}\,\, \text{for any}\,\, q\in \mathbb Z  \}.
 $$

 We have the following exact sequence the sheaves of groups
\begin{equation}\label{eq exact sequence}
e \to \mathcal{A}ut_{(2p+2)}\mathcal{O} \to \mathcal{A}ut_{(2p)}\mathcal{O} \to (\mathcal T_{\gr})_{2p}\to 0
\end{equation}
for any $p\geq 1$, see \cite{Rothstein}. More details about this sequence can be also found in \cite[Proposition 3.1]{COT}

\subsection{Order of a supermanifold}\label{sec Order of a supermanifold} Let again $\mathcal M=(M,\mathcal O)$ be a (non-split) supermanifold. According to Theorem \ref{Theor_Green} a supermanifold corresponds to an element $[\gamma]\in H^1(M,\mathcal{A}ut_{(2)}\gr\mathcal{O})/
\operatorname{Aut} \E$. Furthermore for any $p\geq 1$ we have the following natural embedding of sheaves
$$
\mathcal{A}ut_{(2p)}\mathcal{O} \hookrightarrow \mathcal{A}ut_{(2)} \mathcal{O},
$$
that induces the map of $1$-cohomology sets
$$
H^1(M,\mathcal{A}ut_{(2p)}\mathcal{O}) \to H^1(M, \mathcal{A}ut_{(2)} \mathcal{O}).
$$
(Note that our sheaves are not abelian.) Denote by $H_{2p}$ the image of $H^1(M,\mathcal{A}ut_{(2p)}\mathcal{O})$ in $H^1(M, \mathcal{A}ut_{(2)} \mathcal{O})$. We get the following $\operatorname{Aut} \E$-invariant filtration
\begin{align*}
H^1(M, \mathcal{A}ut_{(2)} \mathcal{O})= H_{2} \supset H_{4}  \supset H_{6}  \supset \cdots .
\end{align*}

Let $\gamma \in [\gamma]$ be any representative.
As in \cite{Rothstein} we define the order $o(\gamma)$ of the cohomology class $\gamma\in H^1(M, \mathcal{A}ut_{(2)} \mathcal{O})$ to be equal to the
maximal number between the numbers $2p$ such that $\gamma\in H_{2p}$. The order of the supermanifold $\mcM$ is by definition the order of the corresponding cohomology class $\gamma$. We put $o(\mcM):=\infty$, if $\mcM$ is a split supermanifold.

\subsection{The automorphism supergroup of a complex-analytic compact supermanifold}

Let us remind a description of a Lie supergroup in terms of a super-Harish-Chandra pair. A {\it Lie supergroup} $\mathcal G$ is a group object in the category of supermanifolds, see for example \cite{Vish_funk,V} for details. Any Lie supergroup can be described using a super-Harish-Chandra pair, see
\cite{Bern} and also \cite{BCC,V}, due to the following theorem, see \cite{V} for the complex-analytic case.

\begin{theorem}\label{theor Harish-Chandra}
	The category of complex Lie supergroups is
	equivalent to the category of complex super Harish-Chandra pairs. 	
\end{theorem}

 A {\it complex super Harish-Chandra pair} is a pair
$(G,\mathfrak{g})$ that consists of a complex-analytic Lie group $G$ and a Lie
superalgebra $\mathfrak{g}=\mathfrak{g}_{\bar
	0}\oplus\mathfrak{g}_{\bar 1}$ over $\mathbb C$, where $\mathfrak{g}_{\bar 0}=\Lie (G)$, endowed with a representation $\operatorname{Ad}: G\to \operatorname{Aut} \mathfrak{g}$ of $G$ in $\mathfrak{g}$ such that
\begin{itemize}
	\item $\operatorname{Ad}$ preserves the parity and induces the adjoint representation 	of $G$ in $\mathfrak{g}_{\bar 0}$,
	
	\item the differential $(\operatorname{d} \operatorname{Ad})_e$ at the identity $e\in G$ coincides with
	the adjoint representation $\operatorname{ad}$ of $\mathfrak g_{\bar 0}$ in $\mathfrak g$.
\end{itemize}

Super Harish-Chandra pairs form a category. (A definition of a morphism is natural, see in \cite{Bern} or in \cite{V}.)

A supermanifold $\mcM=(M,\mcO)$ is called compact if its base space $M$ is compact. If $\mcM$ is a compact complex-analytic supermanifold, the Lie superalgebra of vector fields $\mathfrak o(\mcM)$ is finite dimensional. For a compact complex-analytic supermanifold $\mcM$ we define the {\it automorphism supergroup} as the super-Harish-Chandra pair
\begin{equation}\label{eq def of automorphism supergroup}
(\operatorname{Aut} \mcM, \mathfrak o(\mcM)).
\end{equation}

\section{Super-Grass\-mannians and $\Pi$-symmetric super-Grassmannians}\label{sec charts on Gr}

\subsection{Complex-analytic super-Grass\-mannians and complex-analytic\\ $\Pi$-symmetric super-Grassmannians}\label{sec def of a supergrassmannian}
A super-Grassmannian $\Gr_{m|n,k|l}$ is the supermanifold  that parameterizes all $k|l$-dimen\-sional linear subsuperspaces in $\mathbb C^{m|n}$. Here $k\leq m$, $l\leq n$ and $k+l< m+n$. The underlying space of $\Gr_{m|n,k|l}$ is the product of two usual Grassmannians $\Gr_{m,k}\times \Gr_{n,l}$.
The structure of a supermanifold on $\Gr_{m|n,k|l}$ can be defined in the following way. Consider the following $(m+n)\times (k+l)$-matrix
$$
\mathcal L=\left(
\begin{array}{cc}
A & B\\
C&D\\
\end{array}
\right).
$$
Here $A=(a_{ij})$ is a $(m\times k)$-matrix, whose entries $a_{ij}$ can be regarded as (even) coordinates in the domain of all complex $(m\times k)$-matrices of rank $k$. Similarly $D=(d_{sr})$ is a $(n\times l)$-matrix, whose entries $d_{sr}$ can be regarded as (even) coordinates in the domain of all complex $(n\times l)$-matrices of rank $l$. Further, $B=(b_{pq})$ and $C=(c_{uv})$ are $(m\times l)$ and $(n\times k)$-matrices, respectively, whose entries $b_{pq}$ and $c_{uv}$ can be regarded as generators of a Grassmann algebra. The matrix $\mathcal L$ determines the following open subsuperdomain in $\mathbb C^{mk+nl|ml+nk}$
$$
\mathcal V =(V,\mathcal F_V\otimes \bigwedge (b_{pq},c_{uv})),
$$
where $V$ is the product  of  the domain of complex $(m\times k)$-matrices of rank $k$ and  the domain of complex $(n\times l)$-matrices of rank $l$, $\mathcal F_V$ is the sheaf of holomorphic functions on $V$ and  $\bigwedge (b_{pq},c_{uv})$ is the Grassmann algebra with generators $(b_{pq},c_{uv})$.
 Let us define an action $\mu:\mathcal V\times \GL_{k|l}(\mathbb C) \to \mathcal V$  of the Lie supergroup $\GL_{k|l}(\mathbb C)$ on $\mathcal V$ on the right in the natural way, that is by matrix  multiplication.
 The quotient space under this action is called the {\it super-Grassmannian} $\Gr_{m|n,k|l}$.
Now consider the case $m=n$. A {\it $\Pi$-symmetric super-Grassmannian} $\Pi\!\Gr_{n,k}$ is a subsupermanifold in   $\Gr_{n|n,k|k}$, which is invariant under odd involution $\Pi: \mathbb C^{n|n}\to \mathbb C^{n|n}$, see below.

Let us describe $\Gr_{m|n,k|l}$ and $\Pi\!\Gr_{n,k}$ using charts and local coordinates \cite{Manin}. First of all let as recall a construction of an atlas for the usual Grassmannian $\Gr_{m,k}$. Let $e_1,\ldots e_m$ be the standard basis in $\mathbb C^m$.  Consider a complex $(m\times k)$-matrix $C=(c_{ij})$, where $i=1,\ldots, m$ and $j=1,\ldots, k$, of rank $k$.
Such a matrix determines a $k$-dimensional subspace $W$ in $\mathbb C^m$ with basis $\sum\limits_{i=1}^mc_{i1}e_i,\ldots, \sum\limits_{i=1}^mc_{ik}e_i$.  Let $I\subset\{1,\ldots,m\}$ be a subset of cardinality $k$ such that the square submatrix  $L=(c_{ij})$, $i\in I$ and $j=1,\ldots, k$, of $C$ is non-degenerate. (There exists such a subset since $C$ is of rank $k$.) Then the matrix $C':= C\cdot L^{-1}$ determines the same subspace $W$ and contains the identity submatrix $E_k$ in the lines with numbers $i\in I$. Let $U_I$ denote the set of all $(m\times k)$-complex matrices $C'$ with the identity submatrix $E_k$ in the lines with numbers $i\in I$. Any point $x\in U_I$ determines a $k$-dimensional subspace $W_x$ in $\mathbb C^n$ as above, moreover if $x_1,x_2\in U_I$, $x_1\ne x_2$, then $W_{x_1}\ne W_{x_2}$. Therefore, the set $U_I$ is a subset in $\Gr_{m,k}$. We can verify that $U_I$ is open in a natural topology in $\Gr_{m,k}$ and it is homeomorphic to $\mathbb C^{(m-k)k}$. Therefore $U_I$ can be regarded as a chart on $\Gr_{m,k}$. Further any $k$-dimensional vector subspace in $\mathbb C^n$ is contained in some $U_J$ for a subset $J\subset\{1,\ldots,m\}$ of cardinality $|J|=k$. Hence the collection $\{U_I\}_{|I| =k}$ is an atlas on $\Gr_{m,k}$.

Now we are ready to describe an atlas $\mathcal A$ on $\Gr_{m|n,k|l}$. Let $I=(I_{\bar 0},I_{\bar 1})$ be a pair of sets, where
$$
I_{\bar 0}\subset\{1,\ldots,m\}\quad \text{and} \quad I_{\bar 1}\subset\{1,\ldots,n\},
$$
with $|I_{\bar 0}| = k,$ and $|I_{\bar 1}| = l$. As above to such an $I$ we can assign a chart $U_{I_{\bar 0}} \times U_{I_{\bar 1}}$ on $\Gr_{m,k}\times \Gr_{n,l}$. Let $\mathcal A = \{\mathcal U_{I}\}$ be a family of superdomains parametrized by $I=(I_{\bar 0},I_{\bar 1})$, where
$$
\mathcal U_I:= (U_{I_{\bar 0}}\times U_{I_{\bar 1}}, \mathcal F_{U_{I_{\bar 0}}\times U_{I_{\bar 1}}}\otimes \bigwedge ((m-k)l+ (n-l)k)).
$$
Here $\bigwedge (r)$ is a Grassmann algebra with $r$ generators and $\mathcal F_{U_{I_{\bar 0}}\times U_{I_{\bar 1}}}$ is the sheaf of holomorphic function on $U_{I_{\bar 0}}\times U_{I_{\bar 1}}$. Let us describe the superdomain $\mathcal U_I$ in a  different way. First of all assume for simplicity that $I_{\bar 0}=\{m-k+1,\ldots, m\}$,
$I_{\bar 1}=\{n-l+1,\ldots, n\}$. Consider the following matrix
$$
\mathcal Z_{I} =\left(
\begin{array}{cc}
X&\Xi\\
E_{k}&0\\
H&Y\\0&E_{l}\end{array} \right),
$$
where $E_{s}$ is the identity matrix of size $s$. We assume that the entries of $X=(x_{ij})$ and $Y=(y_{rs})$ are coordinates in the domain $U_{I_{\bar 0}}$ and the domain $U_{I_{\bar 1}}$, respectively. We also assume that the entries of $\Xi=(\xi_{ab})$ and of $H=(\eta_{cd})$ are generators of the Grassmann algebra $\bigwedge ((m-k)l+ (n-l)k)$. We see that  the matrix $\mathcal Z_I$ determines a superdomain
$$
\mathcal U_I:= (U_{I_{\bar 0}}\times U_{I_{\bar 1}}, \mathcal F_{U_{I_{\bar 0}}\times U_{I_{\bar 1}}}\otimes \bigwedge (\xi_{ab},\eta_{cd}))
$$
with even coordinates $x_{ij}$ and $y_{rs}$, and odd coordinates $\xi_{ab}$ and $\eta_{cd}$.

Let us describe $\mathcal U_I$ for any $I=(I_{\bar 0},I_{\bar 1})$. Consider the following $(m+n)\times (k+l)$-matrix
$$
\mathcal Z_{I} =\left(
\begin{array}{cc}
X'&\Xi'\\
H'&Y'\\
\end{array} \right).
$$
Here the blokes $X'$, $Y'$, $\Xi'$ and $H'$ are of size $m\times k$, $n\times l$, $m\times l$ and $n\times k$, respectively.
We assume that this matrix contains the identity submatrix in the lines with numbers $i\in I_{\bar 0}$ and $i\in \{m+j\,\, |\,\, j\in I_{\bar 1}\} $.  Further, non-trivial entries of $X'$ and $Y'$ can be regarded as coordinates in $U_{I_{\bar 0}}$ and $U_{I_{\bar 1}}$, respectively, and non-trivial entries of $\Xi'$ and $H'$ are identified with generators of the Grassmann algebra $\bigwedge ((m-k)l+ (n-l)k)$, see definition of $\mathcal U_I$. Summing up, we have obtained another description of $\mathcal U_I$.

The last step is to define the transition functions in $\mathcal U_I\cap \mathcal U_J$. To do this we need the matrices $\mathcal Z_I$ and $\mathcal Z_J$. We put $\mathcal Z_{J} =\mathcal Z_{I}C_{IJ}^{-1}$,
where $C_{IJ}$ is an invertible submatrix in $\mathcal Z_{I}$ that consists of the lines with numbers $i\in J_{\bar 0}$ and $m + i,$ where $i\in J_{\bar 1}$. This equation gives us a relation between coordinates of $\mathcal U_I$ and $\mathcal U_I$, in other words the transition functions in $\mathcal U_I\cap \mathcal U_J$. The supermanifold obtained by gluing these charts together is called the super-Grassmannian $\Gr_{m|n,k|l}$. The supermanifold $\Pi\!\Gr_{n,k}$ is defined as a subsupermanifold in $\Gr_{n|n,k|k}$ defined in $\mathcal Z_I$ by the equations $X'=Y'$ and $\Xi'=H'$. We can define the  $\Pi\!\Gr_{n,k}$ as all fixed points of an automorphism of $\Gr_{n|n,k|k}$ induced by an odd linear involution $\Pi:\mathbb C^{n|n}\to \mathbb C^{n|n}$, given by
$$
\left(
\begin{array}{cc}
0&E_n\\
E_n&0\\
\end{array} \right)
\left(
\begin{array}{c}
V\\
W\\
\end{array} \right) =
\left(
\begin{array}{c}
W\\
V\\
\end{array} \right),
$$
where $\left(
\begin{array}{c}
V\\
W\\
\end{array} \right)$ is the column of right coordinates of a vector in $\mathbb C^{n|n}$. In our charts $\Pi\!\Gr_{n,k}$ is defined by the following equation
$$
\left(
\begin{array}{cc}
0&E_n\\
E_n&0\\
\end{array} \right)
\left(
\begin{array}{cc}
X&\Xi\\
H&Y\\
\end{array} \right)
\left(
\begin{array}{cc}
0&E_k\\
E_k&0\\
\end{array} \right) = \left(
\begin{array}{cc}
X&\Xi\\
H&Y\\
\end{array} \right),
$$
or equivalently,
$$
X= Y,\quad H=\Xi.
$$

An atlas $\mathcal A^{\Pi}$ on  $\Pi\!\Gr_{n,k}$ contains local charts $\mathcal U_I^{\Pi}$ parameterized by $I\subset \{ 1,\ldots, n\}$ with $|I|=k$. The retract $\gr\Pi\!\Gr_{n,k}$ of  $\Pi\!\Gr_{n,k}$ is isomorphic to $(\Gr_{n,k}, \bigwedge \Omega)$, where  $\Omega$ is the sheaf of $1$-forms on $\Gr_{n,k}$.
More information about super-Grassmannians and $\Pi$-symmetric super-Grassmannians can be found in \cite{Manin}, see also \cite{COT,Vish_Pi sym}.

\subsection{$\Pi$-symmetric super-Grassmannians over $\mathbb R$ and $\mathbb H$}\label{symmetric super-Grassmannians over R and H}

 We will also consider $\Pi$-symmetric super-Grassmannians $\Pi\!\Gr_{n,k}(\mathbb R)$ and $\Pi\!\Gr_{n,k}(\mathbb H)$ over $\mathbb R$ and $\mathbb H$. These supermanifolds are defined in a similar way as $\Pi\!\Gr_{n,k}$ assuming that all coordinates are real or quaternion. In more details, to define $\Pi\!\Gr_{n,k}(\mathbb R)$ we just repeat the construction of local charts and transition functions above assuming that we work over $\mathbb R$. The case of $\Pi\!\Gr_{n,k}(\mathbb H)$ is slightly more complicated. Indeed, we consider charts $\mathcal Z_I$ as above with even and odd coordinates $X=(x_{ij})$ and $\Xi= (\xi_{ij})$, respectively, where by definition
 $$
 x_{ij}:=
 \left(\begin{array}{cc}
 x^{ij}_{11}&  x^{ij}_{12}\\
 -\bar x^{ij}_{12}&  \bar x^{ij}_{11}
 \end{array}
 \right),\quad \xi_{ij}:=\left(\begin{array}{cc}
 \xi_{11}^{ij}&  \xi^{ij}_{12}\\
 -\bar \xi^{ij}_{12}&  \bar \xi^{ij}_{11}
 \end{array}
 \right).
 $$
 Here $x^{ij}_{ab}$ are even complex variables and $\bar x^{ij}_{ab}$ is the complex conjugation of $x^{ij}_{ab}$.  Further, any $\xi_{ab}^{ij}$ is an odd complex variable and $\bar\xi_{ab}^{ij}$ is its complex conjugation. (Recall that a complex conjugation of a complex odd variable $\eta=\eta_1+i\eta_2$  is $\bar \eta :=\eta_1-i\eta_2$, where $\eta_i$ is a real odd variable.) To obtain $\Pi\!\Gr_{n,k}(\mathbb H)$ we repeat step by step the construction above.

\subsection{The order of a $\Pi$-symmetric super-Grassmannian}

We start this subsection with the following theorem proved in \cite[Theorem 5.1]{COT}.

\begin{theorem}\label{theor PiGr is splitt iff}
	A $\Pi$-symmetric super-Grassmannian $\Pi\!\Gr_{n,k}$ is split if and only if $(n,k) = (2,1)$.
\end{theorem}

From  \cite[Theorem 4.4]{COT} it follows that for the $\Pi$-symmetric super-Grassmannian $\mcM= \Pi\!\Gr_{n,k}$ we have
$
H^1(M, (\mathcal T_{\gr})_{p})=\{0\},$ $p\geq 3.$
This implies the following statement.

\begin{proposition}\label{prop o(PiGR)}
	A $\Pi$-symmetric super-Grassmannian $\Pi\!\Gr_{n,k}$ is a supermanifold of   order $2$ for $(n,k)\ne (2,1)$. The order of $\Pi\!\Gr_{2,1}$ is $\infty$, since this supermanifold is split.
\end{proposition}

\begin{proof}
	To show the statement consider the exact sequence (\ref{eq exact sequence}) for $\mcM = \Pi\!\Gr_{n,k}$ and the corresponding exact sequence of cohomology sets
	\begin{align*}
\to	H^1(M,\mathcal{A}ut_{(2p+2)}\mathcal{O} )\to H^1(M,\mathcal{A}ut_{(2p)}\mathcal{O}) \to H^1(M, (\mathcal T_{\gr})_{2p}) \to .
	\end{align*}
	Since $H^1(M, (\mathcal T_{\gr})_{p})=\{0\}$ for $p\geq 3$, see \cite[Theorem 4.4]{COT}, and $\mathcal{A}ut_{(2q)}\mathcal{O} = \id$ for sufficiently large $p$, we have by induction
	$H^1(M, \mathcal{A}ut_{(2q)}\mathcal{O}) =\{\epsilon\}, \,\,\, q\geq 2.$
Therefore $H_{2p}=\{ \epsilon\}$ for $p\geq 2$. Since by Theorem \ref{theor PiGr is splitt iff}, the $\Pi$-symmetric super-Grassmannian $\Pi\!\Gr_{n,k}$ is not split for $(n,k)\ne (2,1)$, the corresponding to $\Pi\!\Gr_{n,k}$, where $(n,k)\ne (2,1)$, cohomology class $\gamma$ is not trivial.  Therefore, $\gamma\in H_2\setminus H_4 = H_2\setminus \{\epsilon\} $.
	This completes the proof.
\end{proof}

\section{Lifting of homotheties on a non-split supermanifold}

\subsection{Lifting of an automorphism in terms of Green's cohomology}
On any vector bundle $\E$ over $M$ we can define a natural automorphism $\phi_{\alpha}$, where $\alpha\in \mathbb C^*=\mathbb C\setminus \{0\}$. In more details, $\phi_{\alpha}$ multiplies any local section by the complex number $\al$. 	Let $r$ be the minimum between positive integers $k$  such that $\alpha^k=1$. The number $r$ is called the {\it order} $\textsf{ord}(\phi_{\al})$ of the automorphism $\phi_{\alpha}$. If such a number does not exist we put $\textsf{ord}(\phi_{\alpha}) = \infty$. In this section we study a possibility of lifting of $\phi_{\alpha}$ on a non-split supermanifold corresponding to $\E$.

A possibility of lifting of an automorphism (or an action of a Lie group) to a non-split supermanifold was studied in \cite{Oni_lifting}, see also \cite[Proposition 3.1]{Bunegina} for a proof of a particular case. In particular the following result was obtained there. Denote by $\underline{\operatorname{Aut}} \E$ the group of automorphisms of $\E$, which are not necessary identical on $M$. Clearly, we have $\operatorname{Aut}\E\subset \underline{\operatorname{Aut}} \E$.

\begin{proposition}\label{prop lift of gamma}
	Let $\gamma\in H^1(M,\mathcal{A}ut_{(2)}\gr\mathcal{O})$ be a Green cohomology class of $\mathcal M$. Then ${\sf B}\in \underline{\operatorname{Aut}} \E$ lifts to $\mathcal M$ if and only if for the induced map in the cohomology group we have ${\sf B}(\gamma)=\gamma$.
\end{proposition}

Consider the case ${\sf B}= \phi_{\al}$ in details. Let us choose an acyclic covering $\mathcal U = \{U_{a}\}_{a\in I}$ of $M$. Then by the Leray theorem, we have an isomorphism $H^1(M,\mathcal{A}ut_{(2)}\gr\mathcal{O}) \simeq H^1(\mathcal U,\mathcal{A}ut_{(2)}\gr\mathcal{O})$, where $H^1(\mathcal U,\mathcal{A}ut_{(2)}\gr\mathcal{O})$ is the \u{C}ech 1-cohomology set corresponding to $\mathcal U$. Let $(\gamma_{ab})$ be a \u{C}ech cocycle representing $\gamma$ with respect to this isomorphism. Then
\begin{align*}
\gamma = \phi_{\al}(\gamma) \,\,\, \Longleftrightarrow \,\,\, \gamma_{ab}=  u_{a} \circ \phi_{\al}(\gamma_{ab}) \circ u_{b}^{-1} =
 u_{a} \circ\phi_{\al} \circ \gamma_{ab} \circ \phi^{-1}_{\al} \circ u_{b}^{-1},
\end{align*}
where $u_{c}\in \mathcal{A}ut_{(2)}\gr\mathcal{O} (U_c)$. In Theorem \ref{theor main} we will show that we always can find a \v{C}ech cocycle $(\gamma_{ab})$ representing the cohomology class $\gamma$ such that
\begin{equation}\label{eq cocycle exact form}
\gamma_{ab}=  \phi_{\al}(\gamma_{ab})  =
\phi_{\al} \circ \gamma_{ab} \circ \phi^{-1}_{\al}.
\end{equation}

 \subsection{Natural gradings in a superdomain}\label{sec aut theta}

 Let us consider a superdomain $\mathcal U:= (U, \mcO)$, where $\mcO= \mathcal F\otimes \bigwedge(\xi_1,\ldots \xi_m)$ and $\mathcal F$ is the sheaf of holomorphic functions on $U$, with local coordinates $(x_a, \xi_b)$. For any $\al\in \mathbb C^*$ we define an automorphism $\theta_{\al}: \mcO\to \mcO$ of order $r= \textsf{ord}(\theta_{\al})$ given by $\theta_{\al} (x_a) =x_a$ and $\theta_{\al} (\xi_b) = \al\xi_b $.  Clearly $\theta_{\al}$ defines the following $\mathbb Z_{r}$-grading (or $\Z$-grading if $r=\infty$) in $\mcO$:
 \begin{equation}\label{eq decomposition al}
 \mcO= \bigoplus_{\tilde k\in \mathbb Z_{r}} \mcO^{\tilde k}, \quad \text{where}\quad
 \mcO^{\tilde k} = \{f\in \mcO \,\,|\,\, \theta_{\al}(f)  = \al^{\tilde   k} f \}.
 \end{equation}

 If $r=2$, the decomposition (\ref{eq decomposition al}) coincides with the standard decomposition of $\mcO=\mcO_{\bar 0}\oplus \mcO_{\bar 1}$ into even and odd parts
 $$
 \mcO_{\bar 0} =  \mcO^{\tilde 0}, \quad \mcO_{\bar 1} =  \mcO^{\tilde 1}.
 $$

 \subsection{Lifting of an automorphism $\phi_{\al}$,  local picture}\label{sec Automorphism psi_al}

  Let $\E$ be a vector bundle, $\mcE$ be the sheaf of section of $\E$, $(M,\bigwedge\mcE)$ be the corresponding split supermanifold, and $\mcM=(M,\mcO)$ be a (non-split) supermanifold with the retract $\gr\mcM\simeq (M,\bigwedge\mcE)$. Recall that the automorphism  $\phi_{\alpha}$ of $\E$ multiplies any local section of $\E$ by the complex number $\al$.  We say that $\psi_{\al}\in H^0(M, \mathcal{A}ut\mathcal{O})$ is a {\it lift} of $\phi_{\al}$ if $\gr(id,\psi_{\al})= (id,\wedge\phi_{\al})$.

 Let $\mathcal B=\{\mathcal V_{a}\}$ be any atlas on $\mcM$ and let $\mathcal V_{a}\in \mathcal B$ be a chart with even and odd coordinates $(x_i,\xi_j)$, respectively. In any such $\mathcal V_{a}\in \mathcal B$ we can define an automorphism $\theta_{\al}^a = \theta_{\al}^a (\mathcal V_{a})$ as in Section \ref{sec aut theta} depending on $\mathcal V_{a}$. This is $\theta^a_{\al}(x_i)=x_i$ and $\theta^a_{\al}(\xi_j)=\xi_j$.

 \begin{proposition}\label{prop new coordinates}
 Let $\psi_{\alpha}$ be a lift of the automorphism $\phi_{\alpha}$  	of order $r= \textsf{ord}(\phi_{\alpha})$.
 \begin{enumerate}
 	\item  If $r$ is even, then there exists an atlas $\mathcal A=\{\mathcal U_{a}\}$ on $\mcM$ with local coordinates $(x^{a}_i,\xi^{a}_j)$ in $\mathcal U_{a}=(U_{a}, \mcO|_{U_a})$  such that
 	$$
 	 \theta_{\al}^a(\psi_{\alpha}	(x_i^{a}))  =  \psi_{\alpha}	(x_i^{a}), \quad   \theta_{\al}^a (\psi_{\alpha} (\xi_k^{a})) = \alpha \psi_{\alpha} (\xi_k^{a}).
 	$$

 	\item  If $r>1$ is odd or if $r=\infty$, then there exists an atlas $\mathcal A=\{\mathcal U_{a}\}$ on $\mcM$ with local coordinates $(x^{a}_i,\xi^{a}_j)$ in $\mathcal U_{a}=(U_{a}, \mcO|_{U_a})$  such that
 	$$
 	\psi_{\alpha}	(x_i^{a}) = x_i^{a} ,\quad  \psi_{\alpha} (\xi_j^{a}) = \al \xi_j^{a}.
 	$$
 \end{enumerate}
 	
 \end{proposition}

 \begin{proof}
 	Let $\mathcal A$ be any atlas on $\mcM$ and let us fix a chart $\mathcal U\in \mathcal A$ with coordinates $(x_i,\xi_j)$.
In local coordinates any lift $\psi_{\alpha}$ of $\phi_{\alpha}$ can be written in the following form
 	\begin{align*}
 	\psi_{\alpha}(x_i) = x_i + F_{2}+F_4+\cdots;\quad
 	 \psi_{\alpha}(\xi_j) = \alpha (\xi_j + G_3+ G_5\cdots),
 	\end{align*}
 	where $F_s=F_s(x_i,\xi_j)$ is a homogeneous polynomial in variables $\{\xi_j\}$ of degree $s$, and the same for $G_q=G_q(x_i,\xi_j)$ for odd $q$.
 We note that
$$
\psi_{\alpha}(F_{s})=\alpha^s F_{s}+\mcJ^{s+1},  \quad \psi_{\alpha}(G_{q})=\alpha^q G_{q}+\mcJ^{q+1}
$$
for any even $s$ and odd $q$.
The idea of the proof is to use successively the following coordinate change	
\begin{equation}\label{eq change x'= x+, xi'=xi}
\begin{split}
&(I)\quad	x'_i= x_i+ \frac{1}{1-\alpha^{2p}} F_{2p}(x_i,\xi_j),\quad \xi'_j = \xi_j;\\
&(II)\quad
x''_i= x'_i,\quad \xi''_j = \xi'_j+ \frac{1}{1-\alpha^{2p}} G_{2p+1}(x'_i,\xi'_j),
\end{split}
\end{equation}
where $p=1,2,3\ldots$ in the following way.

If $r=2$ there is nothing to check. If $r>2$, first of all we  apply (\ref{eq change x'= x+, xi'=xi})(I) and (\ref{eq change x'= x+, xi'=xi})(II) successively for $p=1$.  After coordinate changes (\ref{eq change x'= x+, xi'=xi})(I) we have
 \begin{align*}
 &\psi_{\alpha} (x'_i) = \psi_{\alpha} (x_i+ \frac{1}{1-\alpha^2} F_2) = x_i + F_2 + \frac{\alpha^2}{1-\alpha^2} F_2 +\cdots=\\
  &x_i +  \frac{1}{1-\alpha^2} F_2 +\cdots = x'_i +\cdots \in x'_i + \mathcal J^3;\quad
  \psi_{\alpha} (\xi'_j) \in  \al \xi'_j + \mathcal J^3.
 \end{align*}
 After coordinate changes (\ref{eq change x'= x+, xi'=xi})(II) similarly we will have
 \begin{equation}\label{eq after change p=1}
  \psi_{\alpha} (x''_i) \in x''_i + \mathcal J^4,\quad
 \psi_{\alpha} (\xi''_j) \in \al\xi''_j + \mathcal J^4.
 \end{equation}
	Now we change notations $x_i:=x''_i$ and $\xi_j:=\xi''_j$. Further, since (\ref{eq after change p=1}) holds, we have
	\begin{align*}
	\psi_{\alpha}(x_i) = x_i + F_4+F_6+\cdots;\quad
	\psi_{\alpha}(\xi_j) = \alpha (\xi_j + G_5 + G_7+\cdots).
	\end{align*}
Here we used the same notations for monomials $F_s$ and $G_q$ as above, however after the first step these functions may change.
Now we continue to change coordinates consequentially in this way.
 If $\al^{2p}\ne 1$ for any $p\in \mathbb N$, that is the order $r= \textsf{ord}(\phi_{\alpha})$ is odd or infinite, we can continue this procedure and obtain the required coordinates. This proves the second statement.

 If $r$ is even we  continue our procedure for $p<r/2$. Now in our new coordinates $\psi_{\al}$ has the following form
 \begin{align*}
 &\psi_{\alpha}(x_i) = x_i + F_{r}+F_{r+2}\cdots ;\quad
 &\psi_{\alpha}(\xi_j) = \alpha \xi_j  + \al G_{r+1} + \al G_{r+3} +\cdots.
 \end{align*}

 For any $p$ such that $\al^{2p}\ne 1$, the changes of variables inverse to (\ref{eq change x'= x+, xi'=xi})(I) and (\ref{eq change x'= x+, xi'=xi})(II) have the following form
 \begin{equation}\label{eq inverse of coordinate change}
 \begin{split}
 &(I)\quad	x_a= x'_a+ F'(x'_i,\xi'_j)_{(2p)}, \quad  \xi_b= \xi'_b ;\\
 &(II)\quad	x'_a= x''_a, \quad  \xi'_b= \xi''_b + G'(x''_i,\xi''_j)_{(2p+1)},
 \end{split}
 \end{equation}
 where $F'(x'_i,\xi'_j)_{(2p)}\in \mcJ^{2p}$ and $G'(x''_i,\xi''_j)_{(2p+1)} \in \mcJ^{2p+1}$.

Now we use again the coordinate change (\ref{eq change x'= x+, xi'=xi})(I) and (\ref{eq change x'= x+, xi'=xi})(II) for $p= r+2$, successively. Explicitly after coordinate changes (\ref{eq change x'= x+, xi'=xi})(I) using (\ref{eq inverse of coordinate change})  for $p= r+2$ we have
\begin{align*}
\psi_{\alpha} (x'_i) = \psi_{\alpha} (x_i+ \frac{1}{1-\alpha^{r+2}} F_{r+2}(x_i,\xi_j)) = x_i + F_r(x_i,\xi_j)+
F_{r+2}(x_i,\xi_j) +\\
\frac{\alpha^{r+2}}{1-\alpha^{r+2}} F_{r+2}(x_i,\xi_j) +\cdots=
x_i +  \frac{1}{1-\alpha^{r+2}} F_{r+2}(x_i,\xi_j)  + F_r(x_i,\xi_j) +\cdots =\\
 x'_i +  F_r(x_i,\xi_j)+\cdots \in x'_i +F_r(x'_i,\xi'_j) +\mathcal J^{r+3};\\
\psi_{\alpha} (\xi'_j) \in  \al \xi'_j +\al G_{r+1}(x'_i,\xi'_j) + \mathcal J^{r+3}.
\end{align*}
After the coordinate change (\ref{eq change x'= x+, xi'=xi})(II), we will have
 \begin{align*}
\psi_{\alpha} (x''_i) \in x''_i + F_r(x''_i,\xi''_j)+ \mathcal J^{r+4},\quad
\psi_{\alpha} (\xi''_j) \in \al\xi''_j + \al G_{r+1}(x''_i,\xi''_j) + \mathcal J^{r+4}.
\end{align*}

 Repeating this procedure for $p= r+4, \ldots, 2r-2$ and so on for $p\ne kr$, $k\in \mathbb N$ we obtain the result.
\end{proof}

\subsection{Lifting of an automorphism $\phi_{\al}$,  global picture}

 Now we will show that a supermanifold with an automorphism $\psi_{\al}$ has very special transition functions in an atlas $\mathcal A=\{\mathcal U_{a}\}$ from in Proposition \ref{prop new coordinates}. Recall that in any $\mathcal U_{a}\in \mathcal A$ with coordinates $(x_i,\xi_j)$ we can define an automorphism $\theta_{\al}^a = \theta_{\al}^a (\mathcal U_{a})$ as in Section \ref{sec aut theta} by $\theta^a_{\al}(x_i)=x_i$ and $\theta^a_{\al}(\xi_j)=\xi_j$.

 \begin{theorem}\label{theor main}
 	Let  $\mathcal A=\{\mathcal U_{a}\}$ be an atlas as in Proposition \ref{prop new coordinates} and	let there exists a lift $\psi_{\al}$ of the automorphism $\phi_{\al}$ of order $r= \textsf{ord}(\phi_{\alpha})$. Let us take two charts $\mathcal U_{a},\, \mathcal U_{b}\in \mathcal A $ such that  $U_{a}\cap U_{b}\ne \emptyset$ with coordinates $(x^{a}_s, \xi^{a}_t)$ and $(x^{b}_i, \xi^{b}_j)$, respectively, with  the transition functions $\Psi_{a b}: \mathcal U_{b}\to \mathcal U_{a}$.
 	\begin{enumerate}
 		\item[(I)] If $r$ is even, then we have
 		\begin{equation}\label{eq transition functions}
 		\theta_{\al}^b(\Psi_{a b}^* (x^{a}_s))  = \Psi_{a b}^* (x^{a}_s);\quad \theta_{\al}^b (\Psi_{a b}^* (\xi^{a}_t)) = \alpha \Psi_{a b}^* (\xi^{a}_t).
 		\end{equation}
 	Or more generally,
 	\begin{equation}\label{eq transition functions new}
 	\theta_{\al}^b \circ \Psi_{a b}^*   = \Psi_{a b}^* \circ  \theta_{\al}^a;\quad \theta_{\al}^b \circ \Psi_{a b}^*  =  \Psi_{a b}^* \circ  \theta_{\al}^a.
 \end{equation}

 		\item[(II)] If we can find an atlas $\mathcal A$ with transition functions satisfying (\ref{eq transition functions}),  the automorphism $\phi_{\al}$ possesses a lift $\psi_{\al}$.
 		
 		\item[(III)] If $r>1$ is odd or $r=\infty$, then $\mcM$ is split.
 	\end{enumerate}

 \end{theorem}
 \begin{proof}
 	{\it (III)} Let $\Psi_{a b}^* (x^{a}_s) :=L(x^{b}_i, \xi^{b}_j)= \sum\limits_{k}L_{2k}$, where $L_{2k}$ are homogeneous polynomials of degree $2k$ in variables $\{\xi^{b}_j\}$. Then if $r>1$ is odd or $r=\infty$ by Proposition \ref{prop new coordinates} we have
 	\begin{align*}
 	\psi_{\al}\circ \Psi^*_{a b}(x^{a}_s)& = \psi_{\al} (\sum_{k}L_{2k}) = L_0 + \al^2L_{2} + \al^4L_{4} + \cdots ;\\
 	\Psi^*_{a b}\circ \psi_{\al}(x^{a}_s) &= \Psi^*_{a b} ( x^{a}_s) = L_0 + L_{2} +L_4  +\cdots.
 	\end{align*}
 	Since $\psi_{\al}$ globally defined on $\mcM$, we have the following equality
 	\begin{equation}\label{eq equality for psi_al}
 	\psi_{\al}\circ \Psi^*_{a b} = \Psi^*_{a b}\circ \psi_{\al},
 	\end{equation}
 	which implies that $L_{2q} = 0$ for any $q\geq 1$. Similarly, the equality $\psi_{\al}\circ \Psi^*_{a b}(\xi^{a}_t) = \Psi^*_{a b}\circ \psi_{\al}(\xi^{a}_t)$ implies that $\Psi^*_{a b}(\xi^{a}_t)$ is linear in $\{\xi^{b}_j\}$. In other words, $\mcM$ is split.
 	
 	{\it (I)} Now assume that $r$ is even. Similarly to above we have
 	\begin{align*}
 	\psi_{\al}\circ \Psi^*_{a b}(x^{a}_s)& = \psi_{\al} (\sum_{k}L_{2k}) = L_0 + \al^2L_{2} +  \cdots + \al^{r-2}L_{r-2} + L' ;\\
 	\Psi^*_{a b}\circ \psi_{\al}(x^{a}_s) &= \Psi^*_{a b} ( x^{a}_s + F_r+F_{2r}+\cdots ) = L_0 + L_{2}  +\cdots L_{r-2} + L'',
 	\end{align*}
 	where $L',L''\in \mcJ^{r}$. Again the equality (\ref{eq equality for psi_al}) implies that $L_2=\cdots = L_{r-2}=0$. Similarly, we can show that
 	$$
 	\Psi^*_{a b} (\xi^{a}_t) =  M_1+ M_{r+1} + M_{r+3}+\cdots ,
 	$$
 	where $M_{2k+1}$ are homogeneous polynomials of degree $2k+1$ in variables $\{\xi^{b}_j\}$.

 	Now if $T=T_0+T_1+T_2+\ldots$ is a decomposition of a super-function into homogeneous polynomials in $\{\xi^{b}_j\}$, denote by $[T]_q:= T_q$ its $q$'s part.  Using that $\psi_{\al} (L_{sr})$, where $s\in \mathbb N$, is $\theta_{\al}^b$-invariant,  we have
 	\begin{align*}
 	[\psi_{\al}\circ \Psi^*_{a b}(x^{a}_s)]_{2p} = \al^{2p} L_{2p},\quad 2p=r+2,\ldots, 2r-2.
 	\end{align*}
  Further, using  $\Psi^*_{a b} (F_r) $ is $\theta_{\al}^b$-invariant $mod\, \mcJ^{2r}$, we have
 	\begin{align*}
 	 [\Psi^*_{a b}\circ \psi_{\al}(x^{a}_s)]_{2p} = L_{2p}, \quad 2p=r+2,\ldots, 2r-2.
 	\end{align*}

 	This result implies that $L_{r+2}=\cdots= L_{2r-2}= 0$.  Similarly we work with $M(x^{b}_i, \xi^{b}_j)$.  In the same way we show that $L_{p}=0$ for any $p\ne sr$, where $s=0,1,2,\ldots$.
 	
 	{\it (II)}  If $\mcM$ possesses an atlas $\mathcal A$ with transition functions satisfying (\ref{eq transition functions new}), a lift $\psi_{\al}$ can be defined in the following way for any chart $\mathcal U_{a}$
 	\begin{equation}\label{eq psi standatd any al}
 	\psi_{\al}(x^{a}_i) = x^{a}_i;\quad \psi_{\al}(\xi^{a}_j) = \al \xi^{a}_j.
 	\end{equation}
 	Formulas (\ref{eq transition functions}) shows that $\psi_{\al}$ is well-defined. The proof is complete.
 \end{proof}

 \begin{remark}
Now we can show that (\ref{eq cocycle exact form}) is equivalent  to Theorem \ref{theor main} (I).  Let again $\Psi_{a b}: \mathcal U_{b}\to \mathcal U_{a}$ be the transition function defined in $\mathcal U_a\cap \mathcal U_b$.  In \cite[Section 2]{Bunegina} it was shown that we can decompose these transition functions  in the following way
$$
\Psi^*_{ab} = \gamma_{ab} \circ \gr \Psi^*_{ab},
$$
where  $(\gamma_{ab})$ is a \v{C}ech cocycle corresponding to the covering $\mathcal A=\{\mathcal U_a\}$ representing $\mcM$, see Theorem \ref{Theor_Green}, and $\gamma_{ab}$ is written in coordinates of $\mathcal U_b$.  In other words this means that the transition functions $\Psi_{ab}$  may be obtained from the transition functions of $\gr\Psi_{a b}: \mathcal U_{b}\to \mathcal U_{a}$ of $\gr \mcM$ applying the automorphism $\gamma_{ab}$. (Here we identified $\gr \mathcal U_c$ and $\mathcal U_c$ in a  natural way.)

 In the structure sheaf of $\mathcal U_a$ (respectively $\mathcal U_b$) there is an automorphism $\theta_{\al}^a$ (respectively $\theta_{\al}^b$) defined as above. Since $\gr\mathcal U_c= \mathcal U_c$, we get $\theta_{\al}^a = \phi_{\al}|_{\mathcal U_a}$. Recall that the statement Theorem \ref{theor main} (I)  we can reformulate in the following way
  $$
 \Psi^*_{ab}\circ \phi_{\al} = \phi_{\al} \circ  \Psi^*_{ab}.
 $$
  Further, since, $\gr\Psi^*_{ab}\circ \phi_{\al} = \phi_{\al}\circ \gr \Psi^*_{ab}$, we get $\phi_{\al} \circ \gamma_{ab} = \gamma_{ab}\circ \phi_{\al}$.   Conversely, if $\gamma_{ab}$ is $\phi_{\al}$-invariant, then applying $\Psi^*_{ab} = \gamma_{ab} \circ \gr \Psi^*_{ab}$ we get Theorem \ref{theor main} (I).
\end{remark}

\begin{remark}
	In case $r=\infty$ the result of Theorem \ref{theor main} can be deduced from an observation made in \cite{Koszul} about lifting of graded operators.
\end{remark}

Now we can formulate several corollaries of Theorem \ref{theor main}.

\begin{corollary}
Let $r= \textsf{ord}(\phi_{\al})>1$ and let there exists a lift $\psi_{\al}$ of $\phi_{\al}$ on $\mcM$. Then there exist another lift, denoted by $\psi'_{\al}$, of $\phi_{\al}$ and an atlas $\mathcal A=\{\mathcal U_{a}\}$ with local coordinates $(x^{a}_i,\xi^{a}_j)$ in $\mathcal U_{a}$  such that
$$
\psi'_{\alpha}	(x_i^{a}) = x_i^{a} ,\quad  \psi'_{\alpha} (\xi_k^{a}) = \al \xi_k^{a}.
$$
Indeed, for $r>1$ is odd or $r=\infty$ we can use Proposition \ref{prop new coordinates}(2). For $r$ is even the statement follows from Formulas (\ref{eq psi standatd any al}).
 \end{corollary}

 \begin{corollary}\label{cor psi_-1 exists}
 	Any supermanifold $\mcM$ possesses a lift of an automorphism $\phi_{-1}$. Indeed, by definition $\mcM$ possesses an atlas satisfying (\ref{eq transition functions}).  Therefore
 	 in (any) local coordinates $(x_a,\xi_b)$ of $\mcM$ we can define an automorphism $\psi^{st}_{-1}$ by the following formulas
 	 $$
 	 \psi^{st}_{-1}(x_a)=x_a;\quad  \psi^{st}_{-1}(\xi_b)=-\xi_b.
 	 $$
 	 We will call this automorphism {\it standard}.
 	We also can define this automorphism in the following coordinate free way
 	$$
 	\psi^{st}_{-1}(f)=(-1)^{\tilde i}f, \quad f\in \mathcal O_{\bar i}.
 	$$
 \end{corollary}

 \begin{corollary}\label{cor phi can be lifted iff}
 Let $r= \textsf{ord}(\phi_{\al})>1$ be odd or $\infty$. Then the automorphism $\phi_{\al}$ can be lifted to a supermanifold $\mcM$ if and only if $\mcM$ is split.
 \end{corollary}

 \begin{corollary}\label{cor order of smf and order of al}
 	If the automorphism $\phi_{\al}$ can be lifted to a supermanifold $\mcM$, then $o(\mcM)\geq \textsf{ord}(\phi_{\al})$, where $o(\mcM)$ is the order of a supermanifold $\mcM$, see Section \ref{sec Order of a supermanifold}. In particular, if $o(\mcM)=2$, the automorphism $\phi_{\al}$ can be listed to $\mcM$ only for $\al=\pm 1$.
 \end{corollary}

  \subsection{Lifting of the automorphism $\phi_{1}$ and consequences}
 By definition any lift $\psi_{1}$ of  the automorphism $\phi_{1}=\id$ is a global section of the sheaf $H^0(M,\mathcal{A}ut_{(2)}\mathcal{O})$, see  Section \ref{sec A classification theorem}. The $0$-cohomology group $H^0(M,\mathcal{A}ut_{(2)}\mathcal{O})$ can be computed using the following exact sequence
 \begin{align*}
 \{e\} \to \mathcal{A}ut_{(2q+2)}\mathcal{O} \to \mathcal{A}ut_{(2q)}\mathcal{O} \to (\mathcal T_{\gr})_{2q}\to 0, \quad p\geq 1,
 \end{align*}
see (\ref{eq exact sequence}). Further let we have two lifts $\psi_{\al}$ and $\psi'_{\al}$ of $\phi_{\al}$. Then the composition $\Psi_1:=(\psi_{\al})^{-1}\circ \psi'_{\al}$ is a lift of $\phi_{1}$. Therefore any lift $\psi'_{\al}$ is equal to the composition $\psi_{\al} \circ \Psi_1$ of a fixed lift $\psi_{\al}$ and an element from $\Psi_1\in H^0(M,\mathcal{A}ut_{(2)}\mathcal{O})$. In particular, according Corollary \ref{cor psi_-1 exists} there always exists the standard lift $\psi^{st  }_{-1}$ of $\phi_{-1}$.  Therefore for any lift $\psi'_{-1}$ we have $\psi'_{-1} = \psi^{st}_{-1} \circ \Psi_1$, where  $\Psi_1\in H^0(M,\mathcal{A}ut_{(2)}\mathcal{O})$.

 \section{Automorphisms of the structure sheaf of $\Pi\!\Gr_{n,k}$
}

Let $\mathcal M=\Pi\!\Gr_{n,k}$ be a $\Pi$-symmetric super-Grassmannian. Recall that the retract $\gr\Pi\!\Gr_{n,k}$ of  $\Pi\!\Gr_{n,k}$ is isomorphic to $(\Gr_{n,k}, \bigwedge \Omega)$, where  $\Omega$ is the sheaf of $1$-forms on the usual Grassmannian $\Gr_{n,k}$. The sheaf $\Omega$ is the sheaf of sections of the cotangent bundle $\textsf{T}^*(M)$ over $M=\Gr_{n,k}$. In the next subsection we recover a well-known result  about the automorphism group $\operatorname{Aut}\textsf{T}^*(M)$ of $\textsf{T}^*(M)$.

 \subsection{Automorphisms of the cotangent bundle over a Grassmannian}

 Let $M= \Gr_{n,k}$ be the usual Grassmannian, i.e. the complex manifold that paramete\-rizes all  $k$-dimensional linear subspaces in $\mathbb C^n$ and let  $\textsf{T}^*(M)$ be its cotangent bundle. It is well-known result that $\operatorname{End} \textsf{T}^*(M) \simeq \mathbb C$. Therefore, $\operatorname{Aut}\textsf{T}^*(M) \simeq \mathbb C^*$. For completeness we will prove this fact using use the Borel-Weil-Bott Theorem, see for example \cite{ADima} for details.

 Let  $G=\GL_{n}(\mathbb C)$ be the general linear group,  $P$ be a parabolic subgroup in $G$, $R$ be the reductive part of $P$ and let $\E_{\chi}\to G/P$ be the homogeneous vector bundle corresponding to a representation  $\chi$ of
 $P$ in the fiber $E=(\E_{\chi})_{P}$. Denote by $\mathcal E_{\chi}$ the sheaf of holomorphic section of $\E_{\chi}$.
 In the Lie algebra  $\mathfrak{gl}_{n}(\mathbb C)=\operatorname {Lie}(G)$
 we fix the Cartan subalgebra $\mathfrak t=
 \{\operatorname{diag}(\mu_1,\dots,\mu_n)\}$,
 the following system of positive roots
 $$
 \Delta^+=\{\mu_i-\mu_j\,\,|\,\, \,\,1\leq i<j \leq n\},
 $$
 and the following system of simple roots
 $  \Phi= \{\alpha_1,..., \alpha_{n-1}\}, \,\,\,
 \alpha_i=\mu_i-\mu_{i+1}$, where $i=1,\ldots , n-1$.
 Denote by $\mathfrak t^*(\mathbb R)$
 a real subspace in $\mathfrak t^*$
 spanned by $\mu_j$. Consider the scalar product $( \,,\, )$ in $\mathfrak t^*(\mathbb R)$ such that the vectors  $\mu_j$ form an orthonormal basis. An element $\gamma\in \mathfrak t^*(\mathbb R)$ is called {\it dominant} if $(\gamma, \alpha)\ge 0$ for all $\alpha \in \Delta^+$. We assume that $B^-\subset P$, where $B^-$ is the Borel subgroup corresponding to $\Delta^-$.

 \begin{theorem}[Borel-Weil-Bott]
 	\label{teor borel}  Assume that the representation	$\chi: P\to \GL(E)$ is completely reducible and $\lambda_1,..., \lambda_s$ are highest weights of $\chi|R$. Then the $G$-module $H^0(G/P,\mathcal E_{\chi})$ is isomorphic to the sum of irreducible $G$-modules with highest weights $\lambda_{i_1},..., \lambda_{i_t}$, where
 	$\lambda_{i_a}$ are dominant highest weights of $\chi|R$.
 \end{theorem}

Now we apply this theorem to the case of the usual Grassmannian $\Gr_{n,k}$.  We have $\Gr_{n,k}\simeq G/P$, where $G= \GL_n(\mathbb C)$ and  $P\subset G$ is given by
$$
P= \left\{ \left(
\begin{array}{cc}
A&0\\
B&C
\end{array}
\right)
\right\},
$$
where $A$ is a complex $k\times k$-matrix. We see that $R= \GL_k(\mathbb C)\times\GL_{n-k}(\mathbb C)$. The isotropy representation $\chi$ of $P$ can be computed in a standard way, see for instance \cite[Proposition 5.2]{COT}. The representation $\chi$ is completely reducible and it is equal to $\rho_1\otimes \rho^*_2$, where $\rho_1$ and $\rho_2$ are standard representations of the Lie groups $\GL_k(\mathbb C)$ and $\GL_{n-k}(\mathbb C)$, respectively.

 \begin{proposition}\label{prop automorphisms of T^*(M)}
For usual Grassmannian $M= \Gr_{n,k}$, where $n-k,k>0$, we have
 	$$
 	\operatorname{End} \textsf{T}^*(M) \simeq \mathbb C,\quad \operatorname{Aut}\textsf{T}^*(M) \simeq \mathbb C^*.
 	$$
 \end{proposition}
 \begin{proof}
 	The cotangent bundle $\textsf{T}^*(M)$ over $M$ is homogeneous and the corresponding representation is the dual to isotropy representation $\chi$. Let us compute the representation $\omega$ of $P$ corresponding to the homogeneous bundle
 	$$
 	\operatorname{End} \textsf{T}^*(M)\simeq \textsf{T}(M) \otimes \textsf{T}^*(M).
 	$$
 	The representation $\omega$ is completely reducible and we have
 	$$
 \omega|R=	\rho_1\otimes \rho^*_2\otimes\rho_1^*\otimes \rho_2 \simeq \rho_1\otimes \rho^*_1\otimes\rho_2\otimes \rho^*_2.
 	$$
 	Therefore, we have
 	\begin{enumerate}
 		\item $\omega|R = 1+ ad_{1}+ ad_2 + ad_1\otimes ad_2$ for $k>1$ and $n-k>1$;
 		\item $1 + ad_2$ for $k=1$ and $n-k>1$;
 		\item $1 + ad_1$ for $k>1$ and $n-k=1$;
 		\item $1$ for $k=n-k=1$,
 	\end{enumerate}
  	where $1$ is the trivial one dimensional representation, $ad_1$ and $ad_2$ are adjoint representations of $\GL_k(\mathbb C)$ and $\GL_{n-k}(\mathbb C)$, respectively. Then the heights weights of the representation $\omega|R$ are
  	\begin{enumerate}
  		\item $0,$ $\mu_1-\mu_k$, $\mu_{k+1}-\mu_{n}$,   $\mu_1-\mu_k+ \mu_{k+1}-\mu_{n}$ for $k>1$ and $n-k>1$;
  		\item $0,$ $\mu_{2}-\mu_{n}$ for $k=1$ and $n-k>1$;
  		\item $0,$ $\mu_1-\mu_{n-1}$ for $k>1$ and $n-k=1$;
  		\item $0$ for $k=n-k=1$,
  	\end{enumerate}
  respectively.
 	We see that the unique dominant weight is $0$ in any case.  By Borel-Weil-Bott Theorem we obtain the result.
  \end{proof}

 \subsection{The group $H^0(M,\mathcal{A}ut \mathcal O)$}

Recall that $\mathcal M=(M,\mcO)=\Pi\!\Gr_{n,k}$ is a $\Pi$-symmetric super-Grassmannian. To compute the automorphisms of  $\mcO$ we use the following exact sequence of sheaves
\begin{equation}\label{eq exact sec sheaves 1}
e\to \mathcal{A}ut_{(2)} \mathcal O \xrightarrow[]{\iota} \mathcal{A}ut \mathcal O \xrightarrow[]{\sigma} \mathcal{A}ut (\Omega) \to e,
\end{equation}
where $\mathcal{A}ut (\Omega)$ is the sheaf of automorphisms of the sheaf of $1$-forms $\Omega$. Here the map $\iota$ is the natural inclusion and  $\sigma$ maps any $\delta:\mcO\to \mcO$ to $\sigma(\delta): \mcO/\mcJ\to \mcO/\mcJ$, where $\mcJ$  is again the sheaf of ideals generated by odd elements in $\mcO$. Consider the corresponding to (\ref{eq exact sec sheaves 1}) exact sequence of $0$-cohomology groups
\begin{equation}\label{eq exact seq automorphisms}
\{e\} \to  H^0(M, \mathcal{A}ut_{(2)} \mathcal O )\longrightarrow  H^0(M, \mathcal{A}ut \mathcal O) \longrightarrow \operatorname{Aut} \textsf{T}^*(M),
\end{equation}
and the corresponding to (\ref{eq exact sequence}) exact sequence of $0$-cohomology groups
\begin{equation}\label{eq exact seq automorphisms 3}
\{e\} \to H^0(M, \mathcal{A}ut_{(2p+2)}\mathcal{O}) \to H^0(M,\mathcal{A}ut_{(2p)}\mathcal{O}) \to H^0(M,(\mathcal T_{\gr})_{2p}),\quad p\geq 1.
\end{equation}

In \cite[Therem 4.4]{COT} it has been proven that
\begin{equation}\label{eq Oni Theorem 4.4}
H^0(M, \mathcal (\mathcal T_{\gr})_s)=\{0\}\quad \text{for}\,\,\, s\geq 2.
\end{equation}
(For $\mathcal M=\Pi\!\Gr_{2,1}$ this statement follows from dimensional reason.) Therefore,
\begin{equation}\label{eq H^0()Aut_(2)}
H^0(M, \mathcal{A}ut_{(2)} \mathcal O) =\{e\}.
\end{equation}

Recall that the automorphism $\psi^{st}_{-1}$ of the structure sheaf was defined in Corollary \ref{cor psi_-1 exists}.

 \begin{theorem}\label{theor Aut O for Pi symmetric}
 	Let $\mathcal M=\Pi\!\Gr_{n,k}$ be a $\Pi$-symmetric super-Grassmannian and $(n,k)\ne (2,1)$. Then
 	$$
 	H^0(\Gr_{n,k},\mathcal{A}ut \mathcal O) =\{id, \psi^{st}_{-1} \}.
 	$$ 	
 	For $\mathcal M=\Pi\!\Gr_{2,1}$ we have
 	$$
 	H^0(\Gr_{2,1},\mathcal{A}ut \mathcal O)\simeq \mathbb C^*.
 	$$
 \end{theorem}

 \begin{proof}
  	From (\ref{eq exact seq automorphisms}), (\ref{eq H^0()Aut_(2)}) and Proposition \ref{prop automorphisms of T^*(M)}, it follows that
 	$$
 	\{e\} \to  H^0(M, \mathcal{A}ut \mathcal O) \to  \{\phi_{\alpha}\,\,|\,\, \al\in \mathbb C^* \}\simeq \mathbb C^*.
 	$$
 	Now the statement follows from Proposition \ref{prop o(PiGR)} and Corollary \ref{cor order of smf and order of al}. In more details, for $(n,k)\ne (2,1)$, we have $o(\mcM) =2$, therefore
 	$\phi_{\alpha}$ can be lifted to $\mcM$ if and only if $\ord(\phi_{\alpha})=1$ or $2$. In other words, $\al=\pm 1$.

 	In the case $\mathcal M=\Pi\!\Gr_{2,1}$, we have $\dim \mcM = (1|1)$. Therefore, 	$\mathcal{A}ut_{(2)} \mathcal O = id$ and any $\phi_{\alpha}$ can be lifted to $\mcM$.  The proof is complete.
 \end{proof}

We finish this section with the following theorem.

 \begin{theorem}\label{theor Aut gr O for Pi symmetric}
	Let $\gr \mathcal M=(M,\gr \mcO)=\gr \Pi\!\Gr_{n,k}$, where $\Pi\!\Gr_{n,k}$ is a $\Pi$-symmetric super-Grassmannian. Then
	$$
	H^0(\Gr_{n,k},\mathcal{A}ut (\gr \mathcal O))= \operatorname{Aut} \textsf{T}^*(M) \simeq \mathbb C^*.
	$$ 	
\end{theorem}

\begin{proof}
	In Sequence \ref{eq exact seq automorphisms 3} we can replace $\mcO$ by $\gr \mcO$. (This sequence is exact for any $\mcO'$ such that $\gr\mcO'\simeq \gr \mcO$.) By (\ref{eq Oni Theorem 4.4}) as above we get
	$$
	H^0(M, \mathcal{A}ut_{(2)} (\gr\mathcal O)) =\{e\}.
	$$
	By (\ref{eq exact seq automorphisms}) we have
	$$
	\{e\} \to   H^0(M, \mathcal{A}ut (\gr \mathcal O)) \longrightarrow \operatorname{Aut} \textsf{T}^*(M)\simeq \mathbb C^*.
	$$
	Hence any automorphism from $\operatorname{Aut} \textsf{T}^*(M)$ induces an automorphism of $\gr \mathcal O$, we obtain the result.
\end{proof}

 \section{The automorphism supergroup $\operatorname{Aut}\Pi\!\Gr_{n,k}$ of a $\Pi$-symmetric super-Grassmannian}

\subsection{The automorphism group of  $\Gr_{n,k}$}\label{sec The automorphism group of  Gr}

The following theorem can be found for example in \cite[Chapter 3.3, Theorem 1, Corollary 2]{ADima}.

\begin{theorem}\label{theor autom group of usual grassmannian}
	The automorphism group $\operatorname{Aut} (\Gr_{n,k})$ is isomorphic to $\PGL_n(\mathbb C)$ if $n\ne 2k$ and if $(n,k)=(2,1)$; and $\PGL_n(\mathbb C)$  is a normal subgroup of index $2$ in $\operatorname{Aut} (\Gr_{n,k})$  for $n=2k$, $k\ne 1$.
	
		More precisely in the case $n=2k\geq 4$ we have
	$$
	\operatorname{Aut} (\Gr_{2k,k}) = \PGL_n(\mathbb C) \rtimes \{\id, \Phi \},
	$$
	where $\Phi^2 =\id$ and $\Phi\circ g\circ \Phi^{-1} = (g^t)^{-1}$ for $g\in \PGL_n(\mathbb C)$.
\end{theorem}

An additional automorphism $\Phi$ can be described geometrically.  (Note that an additional automorphism is not unique.)
 It is well-known that $\Gr_{n,k}\simeq \Gr_{n,n-k}$ and this isomorphism is given by $\Gr_{n,k}  \ni V \mapsto V^{\perp} \in \Gr_{n,n-k}$, where $V^{\perp}$ is the orthogonal complement of $V\subset \mathbb C^n$ with respect to a bilinear form $B$.   In the case $n=2k$ we clearly have $\Gr_{n,k} = \Gr_{n,n-k}$, hence the map $V \mapsto V^{\perp}$ induces an automorphism of $\Gr_{2k,k}$, which we denote by $\Phi_B$. This automorphism is not an element of $\PGL_n(\mathbb C)$ for $(n,k)\ne (2,1)$.

  Assume that $B$ is the symmetric bilinear form, given in the standard basis of  $\mathbb C^n$ by the identity matrix. Denote the corresponding automorphism by $\Phi$.
     Let us describe $\Phi$ in the standard coordinates on $\Gr_{2k,k}$, given in Section \ref{sec def of a supergrassmannian}. Recall that the chart $U_I$ on $\Gr_{2k,k}$, where $I=\{k+1, \ldots, 2k\}$, corresponds to the following matrix
 $
 \left(\begin{array}{c}
 X\\
 E\\
\end{array}
 \right),
 $
where $X$ is a $k\times k$-matrix of local coordinates  and $E$ is the identity matrix. We have
$$
\left(\begin{array}{c}
X\\
E\\
\end{array}
\right) \xrightarrow{\Phi}
\left(\begin{array}{c}
E\\
-X^t\\
\end{array}
\right),
$$
since
$$
\left(\begin{array}{c}
	E\\
	-X^t\\
\end{array}
\right)^t \cdot
\left(
\begin{array}{c}
X\\
E\\
\end{array}
\right) = \left(\begin{array}{cc}
E& -X\\
\end{array}
\right) \cdot
\left(
\begin{array}{c}
X\\
E\\
\end{array}
\right) =0.
$$
More general, let $U_I$, where $|I|= k$, be another chart on  $\Gr_{2k,k}$ with coordinates $(x_{ij})$, $i,j=1,\ldots, k$, as described in Section \ref{sec def of a supergrassmannian}. Denote $J:= \{ 1,\ldots, 2k\}\setminus I$. Then $U_J$ is again a chart on  $\Gr_{2k,k}$ with coordinates $(y_{ij})$, $i,j=1,\ldots, k$. Then the automorphism  $\Phi$ is given by $y_{ij} = -x_{ji}$.

\begin{remark}
	In case $(n,k)= (2,1)$ the automorphism $\Phi$ described above is defined as well, however it coincides with the following automorphism from $\PGL_2(\mathbb C)$
	\begin{align*}
	\left(\begin{array}{cc}
	0&1\\
	-1&0\\
	\end{array}
	\right)\cdot
	\left(\begin{array}{c}
	x\\
	1\\
	\end{array}
	\right) = \left(\begin{array}{c}
	1\\
	-x\\
	\end{array}
	\right) = \left(\begin{array}{c}
	1\\
	-x^t\\
	\end{array}
	\right).
	\end{align*}
	The same in another chart.
\end{remark}

Let us discuss properties of $\Phi$ mentioned in Theorem \ref{theor autom group of usual grassmannian}.
Clearly $\Phi^2 = \id$. Further, for $g\in \PGL_n(\mathbb C)$ we have
\begin{align*}
\left[(g^t)^{-1}\cdot \left(\begin{array}{c}
E\\
-X^t\\
\end{array}
\right)\right]^t \cdot \left[g \cdot
\left(
\begin{array}{c}
X\\
E\\
\end{array}
\right)\right] = \left(\begin{array}{cc}
E& -X\\
\end{array}
\right) \cdot g^{-1}\cdot g \cdot
\left(
\begin{array}{c}
X\\
E\\
\end{array}
\right) =0.
\end{align*}
(In other charts $U_I$ the argument is the same.) In other words, if $V\subset \mathbb C^{2k}$ is a linear subspace of dimension $k$, then $(g \cdot V)^{\perp} = (g^t)^{-1} \cdot  V^{\perp}$. Hence,
\begin{align*}
V \xmapsto[]{\Phi^{-1}}V^{\perp}  \xmapsto[]{\text{\,\,\,} g\text{\,\,\,} } g\cdot V^{\perp} \xmapsto[]{\text{\,\,}\Phi\text{\,\,}} (g^t)^{-1} \cdot  V.
\end{align*}
Therefore, $\Phi\circ g\circ \Phi^{-1} = (g^t)^{-1}$.

\subsection{About lifting of the automorphism $\Phi$}\label{sec lifting of exeptional hom}
\subsubsection{Lifting of the automorphism $\Phi$ to $\gr \Pi\!\Gr_{2k,k}$}
Recall that we have
$$
\gr \Pi\!\Gr_{n,k}\simeq (\Gr_{n,k}, \bigwedge \Omega),
$$ where  $\Omega$ is the sheaf of $1$-forms on $\Gr_{n,k}$. Therefore any automorphism of $\Gr_{n,k}$ can be naturally lifted to $\gr \mcM=\gr \Pi\!\Gr_{n,k}$. Indeed, the lift of an automorphism $F$ of $\Gr_{n,k}$ is the automorphism $(F,\wedge \operatorname{d}  (F))$ of $(\Gr_{n,k}, \bigwedge \Omega)$.
Further, by Theorem \ref{theor Aut gr O for Pi symmetric} we have
$$
\{e\} \to   H^0(M, \mathcal{A}ut (\gr \mathcal O)) \simeq \mathbb C^* \longrightarrow \operatorname{Aut}( \gr \mcM) \longrightarrow  \operatorname{Aut} (\Gr_{n,k}).
$$
Hence,
$$
\operatorname{Aut} (\gr \mcM )\simeq \mathbb C^* \rtimes \operatorname{Aut} (\Gr_{n,k}) .
$$
Now we see that $\operatorname{Aut} (\gr \mcM )$ is isomorphic to the group of all automorphisms  $\underline{\operatorname{Aut}} \textsf{T}^*(M)$ of $\textsf{T}^*(M)$.  An automorphism $\phi_{\al} \in \mathbb C^*$ commutes with any $(F,\wedge \operatorname{d}  (F))\in \operatorname{Aut} (\Gr_{n,k})$. Hence we obtain the following result.

\begin{theorem}\label{theor aut gr mcM}
If $\gr\mathcal M = \gr \Pi\!\Gr_{n,k}$, then
$$
\operatorname{Aut} (\gr \mcM )\simeq \underline{\operatorname{Aut}} \textsf{T}^*(M)\simeq  \operatorname{Aut} (\Gr_{n,k})\times \mathbb C^*.
$$	
In other words, any automorphism of  $\gr \mcM $ is induced by an automorphism of $\textsf{T}^*(M)$.
	
More precisely,

	{\bf (1)}	If $\gr\mathcal M = \gr \Pi\!\Gr_{2k,k}$, where $k\geq 2$, then
$$
	\operatorname{Aut} (\gr\mathcal M)\simeq  (\PGL_{2k}(\mathbb C) \rtimes \{\id, (\Phi, \wedge d(\Phi)) \})\times \mathbb C^*,
	$$
	where $(\Phi, \wedge d(\Phi)) \circ g\circ (\Phi, \wedge d(\Phi))^{-1} = (g^t)^{-1}$ for $g\in  \PGL_{2k}(\mathbb C)$.
	
		{\bf (2)} For other $(n,k)$, we have 
	$$
	\operatorname{Aut} (\gr\mathcal M)\simeq \PGL_n(\mathbb C) \times \mathbb C^*.
	$$
\end{theorem}

\begin{corollary}
		We see, Theorem \ref{theor aut gr mcM}, that any lift
	of the automorphism $\Phi$ to $\gr \Pi\!\Gr_{2k,k}$
	has the following form
	$$
	\phi_{\al} \circ (\Phi, \wedge d(\Phi)),\quad \al\in \mathbb C^*.
	$$
\end{corollary}

\subsubsection{An explicit construction of lifts of the automorphism $\Phi$ to $\gr \Pi\!\Gr_{2k,k}$}\label{sec explicit Phi}
 In Section \ref{sec charts on Gr}
we constructed the atlas $\mathcal A^{\Pi}=\{\mathcal U_I^{\Pi}\}$ on $\Pi\!\Gr_{n,k}$. Therefore,  $\gr\mathcal A^{\Pi}:=\{\gr\mathcal U_I^{\Pi}\}$ is an atlas on $\gr \Pi\!\Gr_{n,k}$. For the sake of completeness, we describe a lift $(\Phi, \wedge d(\Phi))$ of $\Phi$ from Section \ref{sec The automorphism group of  Gr} in our local charts.

 First  consider the following two coordinate matrices, see Section \ref{sec charts on Gr}:
\begin{equation}\label{eq two standard charts}
\mathcal Z_{1}=
\left(\begin{array}{cc}
X&\Xi \\
E&0\\
\Xi& X\\
0&E
\end{array}
\right),
\quad
\mathcal Z_{2} = \left(\begin{array}{cc}
E&0 \\
Y&H\\
0& E\\
H&Y
\end{array}
\right),
\end{equation}
where $X = (x_{ij})$, $Y= (y_{ij})$ are $k\times k$-matrices of local even coordinates and $\Xi = (\xi_{st})$, $H = (\eta_{st})$ are $k\times k$-matrices of local odd coordinates on $\Pi\!\Gr_{2k,k}$. Denote by $\mathcal V_i\in \mathcal A^{\Pi}$ the corresponding to $\mathcal Z_i$ superdomain. Then $\gr \mathcal V_1$ and $\gr \mathcal V_2$ are superdomains in $\gr\mathcal A^{\Pi}$ with coordinates $(\gr (x_{ij}), \gr (\xi_{st}))$ and $(\gr (y_{ij}), \gr (\eta_{st}))$, respectively. (Note that we can consider any superfunction $f$ as a morphism between supermanifolds, therefore $\gr f$ is defined.)

We can easily check that the coordinate $\gr (\xi_{ij})$ (or $\gr (\eta_{ij})$) can be identified with the $1$-form $d(\gr(x_{ij}))$ (or $d(\gr(y_{ij}))$, respectively) for any $(ij)$. Using this fact we can describe the automorphism $(\Phi, \wedge d(\Phi))$ on $\gr \Pi\!\Gr_{n,k}$. We get in our local charts
\begin{equation*}
\left(\begin{array}{cc}
\gr X& \gr \Xi \\
E&0\\
\gr \Xi&\gr X\\
0&E
\end{array}
\right)
 \xrightarrow{(\Phi, \wedge d(\Phi))}
 \left(\begin{array}{cc}
 E&0\\
 -\gr X^t& -\gr  \Xi^t \\
 0&E\\
 - \gr \Xi^t& - \gr X^t\\
 \end{array}
 \right).
\end{equation*}
We can describe the automorphism $(\Phi, \wedge d(\Phi))$ in any other charts of $\gr\mathcal A^{\Pi}$ in a similar way. Clearly, $(\Phi, \wedge d(\Phi))\circ (\Phi, \wedge d(\Phi)) =id$.

\subsubsection{About lifting of the automorphism $\Phi$ to
	$\Pi\!\Gr_{2k,k}$}

In this subsection we use results obtained in \cite{COT}. Recall that $\Omega$ is the sheaf of $1$-forms on $\Gr_{n,k}$ and $\mathcal T_{\gr} = \bigoplus_{p\in \Z} (\mathcal T_{\gr})_p$ is the tangent sheaf of $\gr\Pi\!\Gr_{n,k}$. We have the following isomorphism
$$
(\mathcal T_{\gr})_2 \simeq \bigwedge^3 \Omega\otimes \Omega^* \oplus \bigwedge^2 \Omega\otimes \Omega^*.
$$
see \cite[Formula 2.13]{COT}. (This isomorphism holds for any supermanifold with the retract $(M,\bigwedge \Omega)$). Therefore,
\begin{equation}\label{eq H^1-1}
H^1(\Gr_{n,k},(\mathcal T_{\gr})_2) \simeq H^1(\Gr_{n,k},\bigwedge^3 \Omega\otimes \Omega^*) \oplus H^1(\Gr_{n,k},\bigwedge^2 \Omega\otimes \Omega^*).
\end{equation}
By \cite[Proposition 4.10]{COT} we have
\begin{equation}\label{eq H^1-2}
H^1(\Gr_{n,k},\bigwedge^3 \Omega\otimes \Omega^*) =\{0\}.
\end{equation}
Further by Dolbeault-Serre theorem we have
\begin{equation}\label{eq H^1-3}
H^1(\Gr_{n,k}, \bigwedge^2 \Omega\otimes \Omega^*) \simeq H^{2,1} (\Gr_{n,k}, \Omega^*).
\end{equation}
Combining Formulas (\ref{eq H^1-1}), (\ref{eq H^1-2}) and (\ref{eq H^1-3}) we get
\begin{equation*}
H^1(\Gr_{n,k},(\mathcal T_{\gr})_2) \simeq H^{2,1} (\Gr_{n,k}, \Omega^*).
\end{equation*}

Consider the exact sequence (\ref{eq exact sequence}) for the sheaf $\gr \mcO$
$$
e \to \mathcal{A}ut_{(2p+2)}\gr\mathcal{O} \to \mathcal{A}ut_{(2p)}\gr\mathcal{O} \to (\mathcal T_{\gr})_{2p}\to 0.
$$
Since $H^1(M, (\mathcal T_{\gr})_{p})=\{0\}$ for $p\geq 3$, see \cite[Theorem 4.4]{COT}, we have $$H^1(\Gr_{n,k},\mathcal{A}ut_{(2p)} \gr\mathcal{O}) = \{\epsilon\}\quad \text{for}\,\,\, p\geq 2.
$$
Hence we have the following inclusion
\begin{equation}\label{eq inclusion}
H^1(\Gr_{n,k}, \mathcal{A}ut_{(2)}\gr\mathcal{O})  \hookrightarrow H^1(\Gr_{n,k}, (\mathcal T_{\gr})_2)\simeq H^{2,1} (\Gr_{n,k}, \Omega^*).
\end{equation}

Let $\gamma\in H^1(\Gr_{2k,k}, \mathcal{A}ut_{(2)}\gr\mathcal{O})$ be the Green cohomology class of the supermanifold $\Pi\!\Gr_{2k,k}$, see Theorem \ref{Theor_Green}.
Denote by $\eta$ the image of $\gamma$ in $H^{2,1} (\Gr_{2k,k}, \Omega^*)$. (The notation $\eta$ we borrow in \cite{COT}.)

\begin{theorem}\label{theor lift of Phi}
 The automorphism $\phi_{\al} \circ (\Phi, \wedge d(\Phi))$, where $\al\in \mathbb C^*$, can be lifted to
  $\Pi\!\Gr_{2k,k}$, where $k\geq 2$, if and only if $\al = \pm i$.

  The $\Pi$-symmetric super-Grassmannian $\Pi\!\Gr_{2,1}$ is split, in other words,  $\Pi\!\Gr_{2,1}\simeq \gr  \Pi\!\Gr_{2,1}$. Therefore any $\phi_{\al} \circ (\Phi, \wedge d(\Phi))$ is an automorphism of $\Pi\!\Gr_{2,1}$.
\end{theorem}

\begin{proof}
	This statement can be deduced from results of \cite{COT}. Indeed, by \cite[Theorem 5.2 (1)]{COT},  $\Pi\!\Gr_{2k,k}$, where  $k\geq 2$, corresponds to the $(2; 1)$-form $\eta \neq0$ defined by \cite[Formula (4.19)]{COT}. Further, the inclusion (\ref{eq inclusion}) is $\underline {\operatorname{Aut}} \textsf{T}^*(M)$-invariant. In \cite[Lemma 3.2]{COT} it was shown that $\phi_{\al}  (\eta) = \al^2 \eta$ and in the proof of \cite[Theorem 4.6 (2)]{COT} it was shown that $(\Phi, \wedge d(\Phi))(\eta) = -\eta \in H^{2,1} (\Gr_{n,k}, \Omega^*)$.
		This implies that
		$$
		[\phi_{\al} \circ (\Phi, \wedge d(\Phi))] (\eta) =\eta
		$$
		if and only if $\al = \pm i$.
	Hence,	
		$$
	[\phi_{\al} \circ (\Phi, \wedge d(\Phi))](\gamma) =\gamma\in H^1(M, \mathcal{A}ut_{(2)}\gr\mathcal{O})
		$$
			if and only if $\al = \pm i$.
		Hence by Proposition \ref{prop lift of gamma} we obtain the result.
\end{proof}

\subsubsection{A geometric construction of a lift of $\Phi$ to 	$\Pi\!\Gr_{2k,k}$}\label{sec construction of Theta}

Together with the supermanifold $\Pi\!\Gr_{n,k}$ we can consider a $\Pi^L$-symmetric super-Grassmannian $\Pi\!\Gr^L_{n,k}$. A construction of $\Pi\!\Gr^L_{n,k}$ is similar to the construction of $\Pi\!\Gr_{n,k}$ given in Section \ref{sec def of a supergrassmannian}. The difference is that we write even and odd coordinates in rows, not columns. For example a coordinate matrix $\mathcal Z_{I}$, where $I=\{1,\ldots, k\}$, of  $\Pi\!\Gr^L_{n,k}$ has the following form
\begin{align*}
\left(\begin{array}{cccc}
E& X & 0 & \Xi\\
0&\Xi & E& X\\
\end{array}
\right)
\end{align*}
The supermanifold  $\Pi\!\Gr^L_{n,k}$ is a super-Grassmannian of $\Pi^L$-symmetric $k$-dimensional superspaces in $\C^{n|n}$, where $\Pi^L$ is an odd involution, which is linear with respect to left coordinates in $\C^{n|n}$. (If we write our involution $\Pi$ in left coordinates, it will be superlinear, not linear, see \cite{Leites}.)

In $\C^{2k|2k}$ we can consider  a bilinear form given  in the standard basis by the identity matrix of size $4k$. This form is not super-symmetric. The form induces a map $\Pi\!\Gr_{2k,k} \to \Pi\!\Gr^L_{2k,k}$ given in charts by
\begin{equation}\label{eq Z to Z perp}
\left(\begin{array}{cc}
 X& \Xi \\
E&0\\
\Xi&X\\
0&E
\end{array}
\right)  \longrightarrow
\left(\begin{array}{cccc}
E& -X & 0 & -\Xi\\
0&-\Xi & E& -X\\
\end{array}
\right)
\end{equation}
For other charts, let $\mathcal U_I$, where $|I|= k$, be another chart on  $\Pi\Gr_{2k,k}$ with coordinates $(x_{ij}, \xi_{st})$, $i,j,s,t=1,\ldots, k$, as described in Section \ref{sec def of a supergrassmannian}. Denote $J:= \{ 1,\ldots, 2k\}\setminus I$. Then $\mathcal U^L_J$ is a chart on  $\Pi\Gr^L_{2k,k}$ with coordinates $(y_{ij},\eta_{st})$, $i,j,s,t=1,\ldots, k$. Then our map is given by $y_{ij} = -x_{ij}$, $\eta_{st} = -\xi_{st}$. Since,
\begin{align*}
\left(\begin{array}{cccc}
E& -X & 0 & -\Xi\\
0&-\Xi & E& -X\\
\end{array}
\right) \cdot \left(\begin{array}{cc}
X& \Xi \\
E&0\\
\Xi&X\\
0&E
\end{array}
\right) = 0.
\end{align*}
and the same for coordinate matrices of charts $\mathcal U_I$ and $\mathcal U^L_J$, $J:= \{ 1,\ldots, 2k\}\setminus I$, we see that the map is independent on the choice of a chart.

Further, for $\Pi$-symmetric even matrices an $i$-transposition is defined, see (\ref{eq i transposition}),
$$
\left(\begin{array}{cc}
A & B\\
B & A\\
\end{array}
\right)^{t_i} = \left(\begin{array}{cc}
A^t & iB^t\\
iB^t & A^t\\
\end{array}
\right),
$$
which satisfies $(C_1\cdot C_2)^{t_i} = C_2^{t_i} \cdot C_1^{t_i}$. (Here $i B^t$ is the matrix $B^t$, transpose to the matrix $B$, multiplied by the complex number $i=\sqrt{-1}$.) Now we can define an automorphism of  $\Pi\!\Gr_{n,k}$ by
\begin{align*}
\left(\begin{array}{cc}
X& \Xi \\
E&0\\
\Xi&X\\
0&E
\end{array}
\right)  \longrightarrow
\left(\begin{array}{cccc}
E& -X & 0 & -\Xi\\
0&-\Xi & E& -X\\
\end{array}
\right) \xrightarrow{{t_i}}
\left(\begin{array}{cc}
E&0\\
-X^t& -i\Xi^t \\
0&E\\
-i\Xi^t&-X^t\\
\end{array}
\right).
\end{align*}
In charts $\mathcal Z_I \to \mathcal Z_J$ on $\Pi\!\Gr_{n,k}$, where $J:= \{ 1,\ldots, 2k\}\setminus I$, this map is given by $y_{ij} = -x_{ji}$, $\eta_{st} = -i\xi_{ts}$. In notations of Theorem \ref{theor lift of Phi}, the obtained automorphism is the (unique) lift of $\phi_{i} \circ (\Phi, \wedge d(\Phi))$ for $(n,k)\ne (2,1)$ and obtained automorphism is equal to $\phi_{i} \circ (\Phi, \wedge d(\Phi))$ for $(n,k)= (2,1)$.

We denote the automorphism described in this section by $\Theta$.

\subsection{The automorphism supergroup of  $\Pi\!\Gr_{n,k}$}

As above we denote by $\mathfrak v(\mathcal M) := H^0(M, \mathcal T)$ the Lie superalgebra of vector fields on a supermanifold $\mcM$. In this section $\mcM= \Pi\!\Gr_{n,k}$. Recall that $\mathfrak{q}_{n}(\mathbb C)$ is a strange Lie superalgebra, see \cite{Kac} for details. This Lie superalgebra contains all  matrices over $\C$ in the following form
$$
\left(\begin{array}{cc}
A&B \\
B&A\\
\end{array}
\right),
$$
where $A,B$ are complex square matrix of size $n$. For instance the identity matrix $E_{2n}$ of size $2n$ is an element of $\mathfrak{q}_{n}(\mathbb C)$. The corresponding Lie supergroup we denote by $\Q_n(\C)$. This is a subsupergroup in $\GL_{n|n}(\C)$ which is invariant with respect to the odd involution $\Pi$. This Lie supergroup can be seen as a subseperdomain of the following form in $\GL_{n|n}(\C)$.
\begin{align*}
\left(\begin{array}{cc}
L&M \\
M&L\\
\end{array}
\right),
\end{align*}
where $L$ is an $n\times n$-matrix of even coordinates, while $M$ is an $n\times n$-matrix of odd coordinates. The corresponding to $\Q_n(\C)$  Harish-Chandra pair is $(\GL_{n}(\C),\mathfrak{q}_{n}(\mathbb C))$.

In \cite{COT} the following theorem was proved, see  \cite[Theorem 5.2]{COT}, \cite[Theorem 4.2]{COT} and \cite[Section 5]{Vish_Pi sym} for an explicit description of $\mathfrak v(\Pi\!\Gr_{2,1})$.

\begin{theorem}\label{teor vector fields on supergrassmannians}  {\bf (1)} If $\mathcal M = \Pi\!\Gr_{n,k}$, where $(n,k)\ne (2,1)$, then
	$$
	\mathfrak v(\mathcal M)\simeq \mathfrak{q}_{n}(\mathbb C)/\langle E_{2n}\rangle.
	$$
	
{\bf (2)}	If $\mathcal M = \Pi\!\Gr_{2,1}$, then
	$$
	\mathfrak v(\mathcal M)\simeq
	\mathfrak g \rtimes \langle z\rangle  \simeq \mathfrak{q}_{2}(\mathbb C)/\langle E_{4}\rangle \rtimes \langle z\rangle.
	$$
	Here $\mathfrak g = \mathfrak g_{-1}\oplus \mathfrak g_0\oplus \mathfrak g_{1}$ is a $\mathbb Z$-graded Lie superalgebra defined in the following way.
	$$
	\mathfrak g_{-1}= V, \quad \mathfrak g_{0}= \mathfrak{sl}_2(\mathbb C), \quad  \mathfrak g_{1}= \langle d \rangle,
	$$
	where $V= \mathfrak{sl}_2(\mathbb C)$ is the adjoint $\mathfrak{sl}_2(\mathbb C)$-module, $[\mathfrak g_{0}, \mathfrak g_{1}] = \{0\}$, $[d , -]$ maps identically $\mathfrak g_{-1}$ to $\mathfrak g_{0}$, and $z$ is the grading operator of the $\mathbb Z$-graded Lie superalgebra $\mathfrak{g}$, the element $d$ corresponds to the matrix $\left(\begin{array}{cc}
	0&E \\
	E&0\\
	\end{array}
	\right) \in \mathfrak{q}_{2}(\mathbb C)$.
	\end{theorem}

Recall that the automorphism $\psi^{st}_{-1}$ of the structure sheaf was defined in Corollary \ref{cor psi_-1 exists}. Denote by $\Psi^{st}_{-1}$ the corresponding automorphism of the supermanifold $\mcM$. Note that for any supermanifold $(M,\mcO)$ the action of $\ps$ of $H^0(M,\mathcal T)$ is given by $v \mapsto \psi^{st}_{-1}\circ v\circ (\psi^{st}_{-1})^* = (-1)^{\tilde v} v$. If our supermanifold is split and $v$ is a vector field of degree $k$, we have $\phi_{\al}\circ v\circ \phi_{\al^{-1}} = \al^k v$. In particular for the graded operator $z$ from  Theorem \ref{teor vector fields on supergrassmannians} we have $\phi_{\al}\circ z\circ \phi_{\al^{-1}} = z$.

 Now everything is ready to prove the following theorem. Recall that the automorphism supergroup for a compact complex supermanifold is defined in terms of super-Harish-Chandra pairs by formula (\ref{eq def of automorphism supergroup}).

\renewcommand{\AA}{{\sf A}}

\begin{theorem}\label{t:Aut}
	{\bf (1)} If $\mathcal M = \Pi\!\Gr_{n,k}$, where $n\ne 2k$, then
	$$
	\AA\coloneqq\operatorname{Aut} \mathcal M\simeq  \PGL_n(\mathbb C) \times \{\id, \Psi^{st}_{-1} \}.
	$$

	The automorphism supergroup is given by the Harish-Chandra pair
	$$
	( \PGL_n(\mathbb C) \times \{\id, \Psi^{st}_{-1} \},  \mathfrak{q}_{n}(\mathbb C)/\langle E_{2n}\rangle).
	$$
	
	{\bf (2)} If $\mathcal M = \Pi\!\Gr_{2k,k}$, where $k\geq 2$, then
		$$
	\AA\coloneqq\operatorname{Aut} \mathcal M\simeq  \PGL_{2k}(\mathbb C) \rtimes  \{\id, \Theta, \Psi^{st}_{-1}, \Psi^{st}_{-1}\circ \Theta \},
	$$
	where $\Theta^2 = \Psi^{st}_{-1}$, $\Psi^{st}_{-1}$ is a central element of $\AA$,
    and $\Theta \circ g\circ \Theta^{-1} = (g^t)^{-1}$ for $g\in  \PGL_{2k}(\mathbb C)$.
	
     The automorphism supergroup is given by the Harish-Chandra pair
	$$
	(\PGL_{2k}(\mathbb C) \rtimes  \{\id, \Psi^{st}_{-1}, \Theta, \Psi^{st}_{-1}\circ \Theta \},  \mathfrak{q}_{2k}(\mathbb C)/\langle E_{4k}\rangle),
	$$
	where $\Theta \circ C\circ \Theta^{-1} = - C^{t_i}$ for $C\in  \mathfrak{q}_{2k}(\mathbb C)/\langle E_{4k}\rangle$ and $\Psi^{st}_{-1} \circ C\circ (\Psi^{st}_{-1})^{-1} = (-1)^{\tilde C} C$.
	
	{\bf (3)}	If $\mathcal M = \Pi\!\Gr_{2,1}$, then
		$$
	\AA\coloneqq	\operatorname{Aut} \mathcal M\simeq  \PGL_{2}(\mathbb C)\times \mathbb C^*.
		$$

		The automorphism supergroup is given by the Harish-Chandra pair
		$$
		( \PGL_{2}(\mathbb C)\times \mathbb C^*,  	\mathfrak g \rtimes \langle z\rangle).
		$$
		Here $\mathfrak g$ is a $\Z$-graded Lie superalgebra described in Theorem \ref{teor vector fields on supergrassmannians},  the action of $ \PGL_{2}(\mathbb C)\times \mathbb C^*$ on $z$ is trivial, and $\phi_{\al}\in \C^*$ multiplies $X\in \mathfrak v(\Pi\!\Gr_{2,1})_k$ by $\al^k$.
\end{theorem}

\begin{proof}
	We use the following exact sequence of groups:
	$$
	e \to H^0(M,\mathcal{A}ut_{(2)}\mathcal{O}) \to \operatorname{Aut} \mathcal M \to \operatorname{Aut} (\gr \mathcal M).
	$$
By (\ref{eq H^0()Aut_(2)}) we have $H^0(M,\mathcal{A}ut_{(2)}\mathcal{O})=e$. Therefore, $\operatorname{Aut} \mathcal M$ is a subgroup in $\operatorname{Aut} (\gr \mathcal M)$ and the group $\operatorname{Aut} (\gr \mathcal M)$ was computed in Theorem \ref{theor aut gr mcM}. If $(n,k) =(2,1)$, we have $ \Pi\!\Gr_{2,1}\simeq \gr  \Pi\!\Gr_{2,1}$. Hence, the result about the automorphism group follows from Theorem \ref{theor aut gr mcM}.

 In \cite{Manin} it was proven that  $\Pi\!\Gr_{n,k}$ possesses an effective action of $\PGL_n(\mathbb C)$, which is compatible with the natural action of $\PGL_n(\mathbb C)$ on $\gr\mcM$ and $M$ for any $(n,k)$. Assume that $(n,k) \ne (2,1)$. By  Theorem \ref{theor Aut O for Pi symmetric} we see that $H^0(M,\mathcal{A}ut \mathcal O) =\{id, \psi^{st}_{-1} \}$ (automorphisms which are identical on the base space $M$). In other words,  $\phi_{\al} $ can be lifted to $\Pi\!\Gr_{n,k}$ if and only if $\al =\pm 1$. Note that the automorphism $\Psi^{st}_{-1}=(\id, \psi^{st}_{-1})$ is defined on any supermanifold and it always commutes with any other automorphism. This implies the result for $n\ne 2k$.

For $n=2k$, $k\geq 2$, by Theorem \ref{theor lift of Phi}, the automorphism $\phi_{\al} \circ (\Phi, \wedge d(\Phi))$ can be lifted to  $\Pi\!\Gr_{n,k}$ if and only if $\al =\pm i$. Above we denoted the lift of  $\phi_{i} \circ (\Phi, \wedge d(\Phi))$  by $\Theta$. We check that  $ \{\id, \Psi^{st}_{-1}, \Theta, \Psi^{st}_{-1}\circ \Theta \}$  is a subgroup and $\PGL_{2k}(\mathbb C)$ is a normal subgroup in
 $\operatorname{Aut} \mathcal M$. Indeed, $\Theta^2 = \Psi^{st}_{-1}$ since
 \begin{align*}
 \gr (\Theta^2) = \gr (\Theta)^2 = \phi_{i} \circ (\Phi, \wedge d(\Phi)) \circ \phi_{i} \circ (\Phi, \wedge d(\Phi)) = \phi_{-1} = \gr (\Psi^{st}_{-1}).
 \end{align*}

 Further, $(\Psi^{st}_{-1})^2 =\id$ and  $\Psi^{st}_{-1}$ is central.
 Moreover, $\PGL_{2k}(\mathbb C)$ is normal in $\operatorname{Aut}(\gr \mathcal M)$, hence it is normal in $\operatorname{Aut}\mathcal M$ as well. We also have $\Theta \circ g\circ \Theta^{-1} = (g^t)^{-1}$ for $g\in  \PGL_{2k}(\mathbb C)$, since by Theorem \ref{theor aut gr mcM} we have
 \begin{align*}
 \gr (\Theta \circ g\circ \Theta^{-1}) = (\Phi, \wedge d(\Phi)) \circ g\circ (\Phi, \wedge d(\Phi))^{-1} = (g^t)^{-1}.
 \end{align*}

 Let us define the action of $\AA$ on the Lie superalgebra $\mathfrak v(\Pi\!\Gr_{n,k})$. The actions of $\Psi^{st}_{-1}$
and $(\id,\phi_{\al})$ were defined above. Further, since the automorphism $(\Phi, \wedge d(\Phi))$ preserves the $\Z$-grading, its action on $z$ is trivial. Let us compute the action of $\Theta$ on $\mathfrak{q}_{2k}(\mathbb C)/\langle E_{4k}\rangle$. Let $C$ be a homogeneous element in $\mathfrak{q}_{2k}(\mathbb C)/\langle E_{4k}\rangle$, then $E_{4k}+ t C$ is a one-parameter subgroup in $\Q_{2k}(\C)/\langle E_{4k}\rangle$. Here the parity of $t$ is the same as the parity of $C$. We need to compute $\Theta \circ (E_{4k}+ t C)\circ \Theta^{-1}$. We have
 \begin{align*}
 \left(\begin{array}{cc}
 X& \Xi \\
 E&0\\
 \Xi&X\\
 0&E
 \end{array}
 \right) \xrightarrow{\Theta^{-1}}
 \left(\begin{array}{cc}
 E&0\\
 -X^t& i\Xi^t \\
 0&E\\
 i\Xi^t&-X^t\\
 \end{array}
 \right)\xrightarrow{(E_{4k}+ t C)} (E_{4k}+ t C) \cdot
 \left(\begin{array}{cc}
 E&0\\
 -X^t& i\Xi^t \\
 0&E\\
 i\Xi^t&-X^t\\
 \end{array}
 \right)\\
 \xrightarrow{\Theta}
  ((E_{4k} + t C)^{-1})^{t_i} \cdot
 \left(\begin{array}{cc}
 X& \Xi \\
 E&0\\
 \Xi&X\\
 0&E
 \end{array}
 \right).
 \end{align*}
 Therefore,
 \begin{align*}
 \Theta \circ C\circ \Theta^{-1}  = \frac{d}{dt}\Big|_{t=0} (\Theta \circ (E_{4k}+ t C)\circ \Theta^{-1}) = \frac{d}{dt}\Big|_{t=0} ((E_{4k} + t C)^{-1})^{t_i} = - C^{t_i}.
 \end{align*}

  The proof is complete.
\end{proof}


\parskip=5pt
\parindent=12pt

\section{Real structures on a supermanifold}

\subsection{Real structures on commutative superalgebras}

Let us fix $\epsilon_i\in \{\pm 1\}$ for $i=1,2,3$. Following Manin \cite[Definition 3.6.2]{Manin} a
{\em real structure of type $(\epsilon_1,\epsilon_2,\epsilon_3)$} on a $\C$-superalgebra $A$
is an $\mathbb R$-linear (even) automorphism $\rho$ of $A$ such that
\[ \uprho(\uprho(a))=\epsilon_1^{\tilde a} a,
\qquad\uprho(ab)= \epsilon_3\epsilon_2^{\tilde a\hs\tilde b}\hs\uprho b\,\uprho\hm a, \qquad \uprho(\lambda a)=\bar\lambda\hs\uprho\hm a
,\]
for $\lambda\in\C$ and for homogeneous elements  $a,b\in  A$, where $\tilde a$ and $\tilde b$ denote the parities of $a$ and $b$, respectively.  To any real structure $\rho$ of type $(\epsilon_1,\epsilon_2,\epsilon_3)$ on $A$, we can assign a real structure $\rho'$ of type $(\epsilon_1,\epsilon_2,-\epsilon_3)$  as follows: we put $\rho'= -\rho$.
Moreover, we can assign  a real structure $\rho''$ of type $(\epsilon_1,-\epsilon_2,\epsilon_3)$ as follows:
we put  $\rho''(a)=\rho(a)$ if $a$ is even, and $\rho''(a)= i \rho(a)$ when $a$ is odd,
where $i=\sqrt{-1}$. See \cite[Section 1.11.2]{BLMS}.
Thus in order to classify real structures on $A$ of all types $(\epsilon_1,\epsilon_2,\epsilon_3)$,
it suffices to classify real structures of the types $(1,-1,1)$ and $(-1,-1,1)$.
In this paper we consider real structures of the type   $(1,-1,1)$.
Then for a commutative superalgebra $A$ and for a real structure $\rho$ on $A$ we have
\[\uprho(\uprho(a))= a, \qquad\uprho(ab)= \uprho\hm a\,\uprho b, \qquad \uprho(\lambda a)=\bar\lambda\uprho\hm a.\]
For a definition of a real structure on a complex Lie superalgebra see for instance \cite{Serganova}.

\def\Hm{{\mathcal H}}

\subsection{Real structures on supermanifolds}\label{sec complex conjugation}
Recall that a {\it real structure} on a complex-analytic manifold $M$ is an anti-holomorphic involution $\mu:M\to M$, see for instance \cite{ADima-Stefani}. This definition is equivalent to the following one. Let $(M,\mathcal F)$ be a complex-analytic manifold, that is  $\mathcal F$ is the sheaf of holomorphic functions on $M$. A homomorphism   $\rho: \mathcal F\to \mathcal F$ of sheaves of real local algebras is called a {\it real structure} on $M$ if
$$
\rho^2=id, \quad   \rho(\lambda f) = \ov{\lambda} \rho(f),
$$
where $\lambda\in \mathbb C$ and $f,g\in \mathcal F$. These definitions are equivalent. Indeed, if $\mu$ is a complex involution, we put
$$
\rho(f) =\ov{\mu^*(f)}\in \mathcal F.
$$
Conversely, if $\rho$ satisfies the second definition of a real structure, we put
$$
\mu^*(f) = \ov{\rho(f)},\quad \mu^*(\ov{f}) = \rho(f)\quad \text{for}\,\,\, f\in \mathcal F.
$$

Let us define a real structure on a supermanifold. A {\em real structure} of type $(1,-1,1)$ on a complex-analytic supermanfold $\M=(M,\mO)$ is a homomorphism   $\rho: \mcO\to \mcO$ of sheaves of real local superalgebras such that
\[ \rho^2=id, \quad  \rho(\lambda f) = \ov{\lambda} \rho(f)
\qquad\text{for}\ \ \lambda\in\C,\ f\in  \mO.\]
Since $\rho$ preserves the parity of elements in $\mcO$, it  preserves the sheaf of ideals $\mcJ\subset \mcO$ generated by odd elements.  Clearly the induced sheaves homomorphism $\rho' : \mcO/\mcJ \to \mcO/\mcJ$ is a real structure on the underlying manifold $M$.

 As in the case of complex-analytic manifolds let us give another definition of a real structure. To any complex-analytic supermanfold $\M=(M,\mO)$ of dimension $(n|m)$, we can assign a real supermanifold $\M^{\mathbb R}$ of dimension $(2n|2m)$.  This procedure is described for instance in \cite[Section. From complex to real]{Kal}.  There a definition of a complex-conjugation of a  function was given. This definition can be regarded as follows. Any super-function $f\in \mcO$  is at the same time a morphism from  $\mcM$ to $\mathbb C^{1|1}=\mathbb C^{1|0}\oplus \mathbb C^{0|1}$. More precisely, an even function $f\in \mcO_{\bar 0}$ can be regarded as a morphism $f:\mcM\to \mathbb C^{1|0}$  and  an odd element $f\in \mcO_{\bar 1}$ can be regarded as a morphism $f:\mcM\to \mathbb C^{0|1}$.
  Further, let $(z,\eta)$ be the standard complex coordinates in $\mathbb C^{1|1}$, where $z=z_1+iz_2$ and $\eta=\eta_1+i\eta_2$ and $z_i$, $\eta_j$ are standard real even and odd coordinates in $\mathbb R^{2|2}$, respectively.  In $\mathbb C^{1|1}$ we define a complex conjugation by the following formula
  $$
  \ov{z}= z_1-iz_2\quad \text{and} \quad \ov{\eta}= \eta_1-i\eta_2.
  $$
  Then for any $f:\mcM \to \mathbb C^{1|1}$ its complex conjugation $\ov{f}$ can be regarded as a composition of the morphism $f$ and the complex conjugation in $\mathbb C^{1|1}$.
  We also have
  \begin{equation*}
  \ov{f_1 \cdot f_2} = \ov{f_1} \cdot \ov{f_2},\quad f_1,\,f_2\in \mcO.
  \end{equation*}

If $f:\mcM\to \mathbb C^{1|1}$ is a morphism (or a super-function), we have for any morphism $F=(F_0,F^*):\mcM\to \mcM$ the following formula
\begin{equation}\label{eq fucntion as a morphism}
F^*(f) = f\circ F.
\end{equation}

 Now a {\it real structure $\mu$} on a supermanifold $\M$  is an anti-holomorphic involutive  automorphism $\mu=(\mu_0,\mu^*)$ of the  supermanifold $\M^{\mathbb R}$. That is $\mu^*$ maps $\mcO$ to $\ov{\mcO}$ and $\mu^2=\id$.

 \begin{proposition}
 	Two definitions of a real structure on $\mcM$ coincide.
 \end{proposition}
 \begin{proof}
 	Let $\mu=(\mu_0,\mu^*)$ be an anti-holomorphic involutive automorphism of $\mcM^{\mathbb R}$. Then we define
 	$$
 	\rho(f):= \ov {\mu^* (f)}\quad \text{for} \,\,\, f\in \mcO.
 	$$
 Let us check that $\rho\circ \rho =\id$. Using (\ref{eq fucntion as a morphism}) we have
 \begin{align*}
\rho (\rho(f)) = \rho (\ov {f\circ \mu}) = \ov {\mu^* (\ov {f\circ \mu})} = \ov { \ov {f\circ \mu \circ \mu}} = f.
 \end{align*}	
 Other properties of $\rho$ are clear.

 On the other hand, let $\rho:\mcO\to \mcO$ be a real structure according to the first definition. By definition we put
 \begin{align*}
 \mu^*(f) = \ov {\rho (f)}\quad \text{for} \,\,\, f\in \mcO;\quad \mu^*(f) = \rho (\ov{f})\quad \text{for} \,\,\, f\in \ov{\mcO}.
 \end{align*}	
 Clearly $\mu$ is anti-holomorphic. Further, we have
 \begin{align*}
 &\mu^*(\mu^*(f)) = \mu^*(\ov {\rho (f)}) = \rho( {\rho (f)}) = f\quad \text{for} \,\,\, f\in \mcO;\\
 &\mu^*(\mu^*(f)) = \mu^*({\rho (\ov {f})}) = \ov{\rho( {\rho (\ov {f})}) }= f\quad \text{for} \,\,\, f\in \ov{\mcO}.
 \end{align*}	
 The proof is complete.  	
 \end{proof}

\def\upsig{{}^\gamma\hm}

Two real structures $\mu,\mu'$ on  $\M$
are called {\em equivalent} if the pairs $(\M,\mu)$ and $(\M,\mu')$ are isomorphic,
that is, if there exists a (complex-analytic) isomorphism of supermanifolds $\beta\colon\M\to \M$ such that $\mu'\circ\beta=\beta\circ\mu$.

\subsection{Isotropic $\Pi$-symmetric Hermitian super-Grassmannians}\label{sec isotropic Pi-symmetric super-Grassmannians}

Let $A$ be a commutative superalgebra. For even matrices over $A$ an $i$-transposition is defined
\begin{equation}\label{eq i transposition general}
\left(\begin{array}{cc}
A & B\\
C & D\\
\end{array}
\right)^{t_i} = \left(\begin{array}{cc}
A^t & iC^t\\
iB^t & D^t\\
\end{array}
\right),
\end{equation}
which satisfies $(C_1\cdot C_2)^{t_i} = C_2^{t_i} \cdot C_1^{t_i}$. (Here $i B^t$ is the matrix $B^t$, transpose to the matrix $B$, multiplied by the complex number $i=\sqrt{-1}$.) The $i$-transposition maps a $\Pi$-symmetric even matrix to a $\Pi$-symmetric even matrix
\begin{equation}\label{eq i transposition}
\left(\begin{array}{cc}
A & B\\
B & A\\
\end{array}
\right)^{t_i} = \left(\begin{array}{cc}
A^t & iB^t\\
iB^t & A^t\\
\end{array}
\right).
\end{equation}

We denote an isotropic $\Pi$-symmetric Hermitian super-Grassmannian by  $\Pi\operatorname {I}\!\Gr^{\mathbf{H}}_{2k,k}$, where $\mathbf{H}$ is the super-Hermitian form given by the following matrix in the standard basis of $\mathbb C^{2k|2k}$
\begin{align*}
H= \left(\begin{array}{cc}
b_k & 0\\
0 & b_k\\
\end{array}
\right), \quad  b_k=\left(\begin{array}{cc}
0 & -iE_k\\
iE_k & 0\\
\end{array}
\right),
\end{align*}
Compare with Section \ref{append matrix b_k}. We see that the restriction of the form $\mathbf{H}$ to the vector subspace of even $\Pi$-symmetric vectors is equal to the form $\mathcal F(-\,,-)$, see \ref{append matrix b_k}.

The isotropic $\Pi$-symmetric Hermitian super-Grassmannian $\Pi\operatorname {I}\!\Gr^{\mathbf{H}}_{2k,k}$ is a real subsupermanifold in $\Pi\!\Gr_{2k,k}$ defined by the following isotropy condition
\begin{equation}\label{eq isotropy condition}
(\ov{\mathcal Z}_I)^{t_i} \cdot H\cdot \mathcal Z_I = 0,
\end{equation}
where $\mathcal Z_I = \left(
\begin{array}{cc}
X'&\Xi'\\
\Xi'&X'\\
\end{array} \right)$ is a  $\Pi$-symmetric coordinate matrix of $\Pi\!\Gr_{2k,k}$, $i$-transposition is as in (\ref{eq i transposition}) and $\ov{\mathcal Z}_I$ is complex conjugation of $\mathcal Z_I$, see Section \ref{sec complex conjugation}.

Let us write this condition (\ref{eq isotropy condition}) explicitly for the coordinate matrix
\begin{align*}
\mathcal Z_{I}=
\left(\begin{array}{cc}
X&\Xi \\
E&0\\
\Xi& X\\
0&E
\end{array}
\right).
\end{align*}
We get after a direct computation
\begin{equation}\label{eq isotropy condition easy chart}
X= (\ov{X})^t,\quad - i \Xi = (\ov{\Xi})^t.
\end{equation}
The base space of this supermanifold is described in Corollary (\ref{cor iso gr}).

\section{Real structures on a $\Pi$-symmetric super-Grassmannian}

\subsection{The action of the complex conjugation on
$	\operatorname{Aut} \PiG_{n,k}$ }\label{ss:action of bar on Aut mcM}
If $a\in \operatorname{Aut} \mathcal M$,
we denote by  $\upsig a\in \operatorname{Aut} \mathcal M$ the  element satisfying
$(\upsig a)^*(f) = \overline{a^*(\overline f)}$ for any local function $f$ on $\mathcal M$.
(Recall that $\overline f$ is defined in Section \ref{sec complex conjugation}.)
Let us compute $\upsig a$ for any $a\in \operatorname{Aut}  \PiG_{n,k}$.  	
The group $\operatorname{Aut} \PiG_{n,k}$ was computed in Theorem \ref{t:Aut}.

\begin{proposition}\label{prop action of bar on Aut mcM}
	For any $(n,k)$ we have
	$$
	\upsig g = \ov{g}, \quad g\in \PGL_n(\mathbb C),\quad \upsig\hs(\Psi^{st}_{-1}) = \Psi^{st}_{-1}.
	$$
For $(2k,k)$, $k\geq 2$, we have
$$
\upsig\hs\Theta = \Theta^{-1}= \Psi^{st}_{-1}\circ \Theta. $$ 	
For $(2,1)$ we have
$$
 \upsig\phi_{\al} = \phi_{\ov{\al}}.
$$
\end{proposition}
\begin{proof}
 First of all, $ \Psi^{st}_{-1}$ multiplies an odd local function by $-1$ and it is identical on an even function.
 Therefore, for any supermanifold we have $\upsig\hs(\Psi^{st}_{-1}) = \Psi^{st}_{-1}$,
 see Section \ref{sec complex conjugation}.

 Further, let $\mathcal Z_I$ be  a coordinate matrix of $\PiG_{n,k}$, see Section \ref{sec def of a supergrassmannian} and $g\in \PGL_n(\mathbb C)$. Then the action of $\PGL_n(\mathbb C)$ on $\PiG_{n,k}$ is defined by the matrix multiplication $g\cdot \mathcal Z_I$, see \cite{Manin}.
 Therefore, $\overline{g\cdot  \overline{\mathcal Z_I}} = \overline{g} \cdot \mathcal Z_I$.
 Hence, $\upsig g = \ov{g}$.

 The automorphism $\Phi$ of $\Gr_{n,k}$ is described in local coordinates in Section \ref{sec The automorphism group of  Gr}.
 Clearly, $\upsig\hs\Phi = \Phi$.
 Note that $\upsig\hs(\Phi, \wedge d(\Phi))$ is the lift of $\upsig\hs\Phi$
 to $ \textsf{T}^*(\Gr_{n,k})$. The result follows.

 Recall that  $\phi_{\al}$ is defined for split supermanifolds and it multiplies a local section of the corresponding bundle by $\al$.
 Hence, $\upsig \phi_{\al}$ multiplies a local section of the corresponding bundle by $\ov{\al}$.

 To compute $\upsig\hs\Theta$ we note that $\gr (\ov{f}) = \ov{\gr (f)}$ for any function on  $\PiG_{n,k}$.
 Therefore,
 $$\gr (\upsig\hs\Theta) = \upsig\hs(\gr \Theta) =  \upsig\hs(\phi_{i} \circ (\Phi, \wedge d(\Phi)))
     = \phi_{-i} \circ (\Phi, \wedge d(\Phi)) = \gr (\Theta^{-1}).$$
 Recall that $\gr \operatorname{Aut}  \PiG_{n,k} \to \operatorname{Aut}  (\gr\PiG_{n,k})$ is injective, see Proof of Theorem \ref{t:Aut}. The result follows.
 \end{proof}

\subsection{Real structures on $\PiG_{n,k}$}\label{ss:real-structures}
Let $\mu^o=(\mu^o_0,(\mu^o)^*)$ denote the standard real structure on $\M=\PiG_{n,k}$.
Namely, the anti-holomorphic  involution $\mu^o$ of $\PiG_{n,k}$ is induced by the complex conjugation in $\C^{n|n}$. More precisely, let us describe $\mu^o$ in our charts. For instance, in the chart $\mathcal V_{1}$, see (\ref{eq two standard charts}), with local coordinates $x_{ij} = x^1_{ij} + ix^2_{ij} $, $\xi_{ab} = \xi^1_{ab} + i \xi^2_{ab}$, where $x^1_{ij},\, x^2_{ij};\, \xi^1_{ab},\,\xi^2_{ab}$ are real even and odd coordinates, the real structure $\mu^o$ is given by
	$$
	(\mu^o)^*(x^1_{ij} + ix^2_{ij})  = x^1_{ij} - ix^2_{ij},\quad
	(\mu^o)^*(\xi^1_{ab} + i \xi^2_{ab})  = \xi^1_{ab} -i \xi^2_{ab}.
	$$
In other charts the idea is similar.

If $n$ is even, write
	\[ a_J={\rm diag}(J,\dots,J)\ \,\text{($n/2$ times),\ \ where}\quad J=\SmallMatrix{0 &1\\-1&0}.\]
Set $c:=\pi(a_J)\in \PGL_n(\mathbb C)$, where $\pi\colon \GL_n(\mathbb C)\to\PGL_n(\mathbb C)$ is the canonical homomorphism. Due to Theorem \ref{t:Aut}, the element $c\in \PGL_n(\mathbb C)$ can be regarded as an holomorphic automorphism of $\mcM$.

\begin{theorem}\label{c:Pi}
	The number  of the equivalence classes of real structures $\mu$ on $\mcM$, and  representatives of these classes, are given in the list below:
	\begin{enumerate}
		\item[\rm (i)] If $n$ is odd, then there are two equivalence classes with representatives
		$$
		\mu^o, \quad (1,\ps)\circ\mu^o.
		$$
		\item[\rm (ii)] If $n$ is even and $n\neq 2k$,
		 then there are four equivalence classes with representatives
		$$
		\mu^o,\quad (1,\ps)\circ\mu^o, \quad  (c_J,1)\circ\mu^o, \quad (c_J,\ps)\circ\mu^o.
		$$
	
		\item[\rm (iii)] If $n=2k\ge 4$, then there are $k+3$ equivalence classes with representatives
		$$
		\mu^o,\quad (c_J,1)\circ\mu^o, \quad  (c_r,\Theta)\circ\mu^o, \,\, r= 0,\ldots, k.
		$$

	\item[\rm (iv)] 	If $(n,k)= (2,1)$, then there are two equivalence classes with representatives
		$$
		\mu^o,\quad  (c_J,1)\circ\mu^o.
		$$
	\end{enumerate}
Here  $\mu^o$ denotes the standard real structure on $\M=\PiG_{n,k}$ as above.
Moreover, $c_J\in\PGL_n(\C)$ and $c_r\in\PGL_{2k}(\C)$ for $r= 0,\ldots, k$ are certain elements
constructed  in Proposition \ref{p:H1} and Subsection \ref{ss:cp}, respectively.
\end{theorem}

\begin{proof}
By Proposition \ref{p:Serre-reduction} in Appendix \ref{app:cohomology},
the theorem follows immediately from Theorem  \ref{t:Aut-M}.
\end{proof}

\subsection{Ringed space of real points of $\PiG_{n,k}$}
We wish to describe the ringed space of real points of $\M=(M,\mO)$
of the corresponding  real structures $\mu=(\mu_0,\mu^*)$ from Theorem \ref{c:Pi}. In more details the {\it ringed space of real points} is by definition the ringed space $\M^{\mu}:= (M^{\mu_0}, \mcO^{\mu^*})$, where $M^{\mu_0}$ is the set of fixed points of $\mu_0$ and $\mcO^{\mu^*}$ is the sheaf of fixed points of $\mu^*$ over $M^{\mu}$. Let us first describe real points corresponding to to the real structures $\mu^o$ and $\ps\circ \mu^o$.

\begin{proposition}\label{prop PiGr(R)}
	\begin{enumerate}
		\item The ringed space of real points of $\PiG_{n,k}$ corresponding to the real structure $\mu^o$ can be identified with $\PiG_{n,k}(\mathbb R)$.
		\item The ringed space of real points of $\PiG_{n,k}$ corresponding to the real structure $\ps\circ \mu^o$ can be identified with a real subsupermanifold, which we denote by $\PiG'_{n,k}(\mathbb R)$.
	\end{enumerate}
\end{proposition}

\begin{proof}
	The first statement is obvious. Let us describe the supermanifold $\PiG'_{n,k}(\mathbb R)$ using charts and local coordinates. An atlas $\mathcal A$ of $\PiG'_{n,k}(\mathbb R)$ contains charts $\mathcal U_{I}$, where $I\subset\{1,\ldots,n\}$ with $|I| = k$. To any $I$ we assign the following $2n\times 2k$-matrix
	$$
	\mathcal Z_{I} =\left(
	\begin{array}{cc}
	X' & \Xi'\\
	\Xi' & X' \end{array} \right),
	$$
	We assume that $\mathcal Z_{I}$ contains the identity submatrix $E_{2k}$ of size $2k$ in the lines with numbers $i$ and $n+i$, where $i\in I$.  We also assume that non-trivial elements $x_{ij}$ of $X'$ are real numbers, i.e. $\ov{x}_{ij} = x_{ij}$,  while non-trivial elements $\xi_{ij}$ of $\Xi'$ are pure imaginary odd variables, i.e. $\ov{\xi}_{ij} = - \xi_{ij}$.  Transition functions are defined as above.
\end{proof}

\begin{remark}
	Note that for $n= 2k\geq 4$, compare with Theorem \ref{c:Pi}, the real structure $\ps\circ \mu^o$ is equivalent to the real structure $\mu^o$. Indeed, we have $\Theta \circ(\ps\circ \mu^o )\circ \Theta^{-1} = \mu^o$.
\end{remark}

In Proposition \ref{p:real-points} in Appendix \ref{app:cohomology} a description of real points of the base space $\Gr_{2n,2k}$ corresponding to the real structure $(c_J,1)\circ \mu^o$ was given. Indeed, we have a natural embedding $\GL_{n}(\mathbb H)\hookrightarrow \GL_{2n}(\mathbb C)$ defined as follows. If $(a_{ij})\in \GL_{n}(\mathbb H)$, we replace any quaternion entry $a_{ij}$ by the complex matrix $\left(\begin{array}{cc}
a^{ij}_{11}&  a^{ij}_{12}\\
-\bar a^{ij}_{12}&  \bar a^{ij}_{11}
\end{array}
\right)$, see (\ref{e:ijk}). This inclusion leads to an inclusion of homogeneous spaces
\begin{equation}\label{eq inclusion of homog spaces}
\Gr_{n,k}(\mathbb H)\simeq \GL_{n}(\mathbb H)/P(\mathbb H)\hookrightarrow \GL_{2n}(\mathbb C)/P(\mathbb C) \simeq \Gr_{2n,2k},
\end{equation}
where $P(\mathbb H)\subset \GL_{n}(\mathbb H)$ contains all matrices in the form $\left(\begin{array}{cc}
S&  0\\
T&  U
\end{array}
\right)$ over $\mathbb H$, where $S\in \operatorname{Mat}_{n-k}(\mathbb H)$, and $P(\mathbb C)\subset \GL_{2n}(\mathbb C)$ contains all matrices in the form $\left(\begin{array}{cc}
S'&  0\\
T'&  U'
\end{array}
\right)$ over $\mathbb C$, where $S'\in \operatorname{Mat}_{2n-2k}(\mathbb C)$.

 Let us describe this embedding in terms of local charts on  $\Gr_{n,k}(\mathbb H)$ and $\Gr_{n,k}(\mathbb C)$ as above. Let $U_I$, where $I\subset \{1,\ldots, n\}$ with $|I|=k$, be a local chart on  $\Gr_{n,k}(\mathbb H)$. Denote by $I'\subset \{1,\ldots, 2n\}$ the following set
 \begin{equation}\label{eq I'}
 I':= \{ 2i,\,\,2i-1\,\,| i\in I \}.
 \end{equation}
and by $U'_{I'}$ the corresponding to $I'$ chart on $\Gr_{n,k}(\mathbb C)$. Now the embedding (\ref{eq inclusion of homog spaces}) in charts $U_{I}\to U'_{I'}$ is given by
$$
\mathbb H\ni x_{ij} \mapsto
\left(\begin{array}{cc}
x^{ij}_{11}&  x^{ij}_{12}\\
-\bar x^{ij}_{12}&  \bar x^{ij}_{11}
\end{array}
\right),
$$
where $(x_{ij})$ are standard coordinates in $U_I$. For instance we see that the image of $\Gr_{n,k}(\mathbb H)$ in $\Gr_{n,k}(\mathbb C)$ is covered by chart of the form $U'_{I'}$, where $I'$ is as in (\ref{eq I'}).

\begin{proposition}\label{prop PiGr(H)2}
	Let $\M=\PiG_{2n,k}$.	
	\begin{enumerate}
		\item If $k$ is odd, there are no real points corresponding to the real structure $(c_J,1)\circ \mu^o$.

		\item If $k=2k'$, the real points of $\M$ corresponding to the real structure $(c_J,1)\circ \mu^o$ can be identified with $\PiG_{n,k'}(\mathbb H)$.
\end{enumerate}
\end{proposition}

\begin{proof}
		The first statement follows from Proposition \ref{p:real-points} in Appendix \ref{app:cohomology}. Indeed, in this case there is no real point on the base space. Assume that $k=2k'$.  We have to compute the fixed points of $(c_J,1)\circ \mu^o$. Let us do that for $n=2$ and $k'=1$; the general case in charts $U'_{I'}$, see (\ref{eq I'}), is similar. We apply the standard real structure $\mu^o$ to the coordinates $Z_1$, see (\ref{eq two standard charts}), and now we apply $c_J$
	\begin{align*}
	{\rm diag}(J,J,J,J)
	\left(\begin{array}{cccc}
	\bar x_{11}& \bar x_{12}&\bar \xi_{11}& \bar \xi_{12}\\
	\bar x_{21}& \bar x_{22}& \bar \xi_{21} & \bar \xi_{22} \\
	1&0&0&0\\
	0&1&0&0\\
	\bar \xi_{11}& \bar \xi_{12} &\bar x_{11}& \bar x_{12} \\
	\bar \xi_{21} & \bar \xi_{22} &\bar x_{21}& \bar x_{22} \\
	0&0& 1&0\\
	0&0& 0&1\\
	\end{array}
	\right)
	=
	\left(\begin{array}{cccc}
	\bar x_{21}& \bar x_{22}&\bar \xi_{21}& \bar \xi_{22}\\
	-\bar x_{11}& -\bar x_{12}& -\bar \xi_{11} & \bar \xi_{12} \\
	0&1&0&0\\
	-1&0&0&0\\
	\bar \xi_{21}& \bar \xi_{22} &\bar x_{21}& \bar x_{22} \\
	-\bar \xi_{11} & -\bar \xi_{12} &-\bar x_{11}& -\bar x_{12} \\
	0&0& 0&1\\
	0&0& -1&0\\
	\end{array}
	\right)
	\end{align*}

Therefore, in the standard coordinates, we get
\begin{align*}
\left(\begin{array}{cccc}
\bar x_{21}& \bar x_{22}&\bar \xi_{21}& \bar \xi_{22}\\
-\bar x_{11}& -\bar x_{12}& -\bar \xi_{11} & \bar \xi_{12} \\
0&1&0&0\\
-1&0&0&0\\
\bar \xi_{21}& \bar \xi_{22} &\bar x_{21}& \bar x_{22} \\
-\bar \xi_{11} & -\bar \xi_{12} &-\bar x_{11}& -\bar x_{12} \\
0&0& 0&1\\
0&0& -1&0\\
\end{array}
\right)
\left(\begin{array}{cccc}
0&-1 & 0& 0\\
1&0 & 0& 0\\
0& 0 &0&-1\\
0& 0 & 1&0
\end{array}
\right)
=\\
\left(\begin{array}{cccc}
\bar x_{22}& -\bar x_{21}&\bar \xi_{22}& -\bar \xi_{21}\\
-\bar x_{12}& \bar x_{11}& -\bar \xi_{12} & \bar \xi_{11} \\
1&0&0&0\\
0&1&0&0\\
\bar\xi_{22}& -\bar \xi_{21} &\bar x_{22}& -\bar x_{21} \\
-\bar \xi_{12} & \bar \xi_{11}  &-\bar x_{12}& \bar x_{11} \\
0&0& 1&0\\
0&0& 0&1\\
\end{array}
\right)
 =
\left(\begin{array}{cccc}
 x_{11}&  x_{12}& \xi_{11}&  \xi_{12}\\
 x_{21}&  x_{22}&  \xi_{21} &  \xi_{22} \\
1&0&0&0\\
0&1&0&0\\
 \xi_{11}& \xi_{12} & x_{11}&  x_{12} \\
 \xi_{21} &  \xi_{22} & x_{21}&  x_{22} \\
0&0& 1&0\\
0&0& 0&1\\
\end{array}
\right)
\end{align*}
Therefore, $x_{11} = \bar x_{22}$, $ x_{12} = -\bar x_{21}$,  $\xi_{11} = \bar \xi_{22}$, $ \xi_{12} = -\bar \xi_{21}$. In the matrix form we have
$$
\left(\begin{array}{cccc}
x_{11}&  x_{12}& \xi_{11}&  \xi_{12}\\
-\bar x_{12}&  \bar x_{11}&  -\bar \xi_{12} & \bar \xi_{11} \\
1&0&0&0\\
0&1&0&0\\
\xi_{11}& \xi_{12} & x_{11}&  x_{12} \\
-\bar \xi_{12} & \bar \xi_{11}  & -\bar x_{12}&  \bar x_{11} \\
0&0& 1&0\\
0&0& 0&1\\
\end{array}
\right)
$$
Now we can regard the matrix
$
z:=\left(\begin{array}{cc}
x_{11}&  x_{12}\\
-\bar x_{12}&  \bar x_{11}
\end{array}
\right)$ as a new quaternion even variable, while $
\zeta:=\left(\begin{array}{cc}
\xi_{11}&  \xi_{12}\\
-\bar \xi_{12}&  \bar \xi_{11}
\end{array}
\right)$ is a new quaternion odd variable. In short we have
$$
\left(\begin{array}{cc}
z&\zeta\\
1&0\\
\zeta& z\\
0&1
\end{array}
\right)
$$
is a standard chart on a $\Pi$-symmetric quaternion super-Grassmannian $\PiG_{2,1}(\mathbb H)$, see above.  Clearly our computation does not depend on the  choice of a local chart $U'_{I'}$ and it is similar for other $n$ and $k'$.
\end{proof}

\begin{proposition}\label{prop PiG'(H)}
\begin{enumerate}
	\item If $n$ is even but $k$ is odd, then there are no real points corresponding to the real structure $(c_J,\ps)\circ \mu^o$.
	\item If  $n=2n'$ and $k=2k'$, then the real points of $\M$ corresponding to the real structure $(c_J,\ps)\circ  \mu^o$ is a real supermanifold, which we denote by $\PiG'_{2n',k'}(\mathbb H)$.
\end{enumerate}
\end{proposition}

\begin{proof}
The first statement follows from Proposition \ref{p:real-points} in Appendix \ref{app:cohomology}, since in this case there is no real point on the base space. If $k=2k'$ is even, we repeat the argument similar to Proposition \ref{prop PiGr(R)}. For simplicity we assume that $k'=1$, further we apply $(c_J,\ps)\circ  \mu^o$, we will get for odd coordinates (for even coordinates the equation is the same as above):
$$
\left(\begin{array}{cc}
\xi_{11}&  \xi_{12}\\
\xi_{21}&  \xi_{22}
\end{array}
\right)
=
\left(\begin{array}{cc}
-\ov{\xi}_{22}&  \ov{\xi}_{21}\\
\ov{\xi}_{12}&  -\ov{\xi}_{11}
\end{array}
\right).
$$
Therefore, $\xi_{11} = -\ov{\xi}_{22}$ and $\xi_{21} =  \ov{\xi}_{12}$. Now $
\zeta:=\left(\begin{array}{cc}
\xi_{11}&  \xi_{12}\\
\bar \xi_{12}&  -\bar \xi_{11}
\end{array}
\right)$ is a new quaternion odd variable. Now local charts on $\PiG'_{2n',k'}(\mathbb H)$ are defined.
\end{proof}

\begin{remark}
	Note that for $n= 2k\geq 4$, compare with Theorem \ref{c:Pi}, the real structure $(c_J,\ps)\circ \mu^o$ is equivalent to the real structure $(c_J,1)\circ \mu^o$. Indeed, we have $\Theta \circ((c_J,\ps)\circ \mu^o)\circ \Theta^{-1} = (c_J,1)\circ \mu^o$.
\end{remark}

Let us prove the following proposition.

\begin{proposition}\label{prop Pi Isotropic_super_Gra}
	\begin{enumerate} Let $n=2k$.
		\item There are no real points corresponding to the real structure $(c_r,\Theta)\circ\mu^o, \,\, r= 0,\ldots, k-1$.
		\item Let $k\geq 2$. The real points of $\M$ corresponding to the real structure $(c_k,\Theta)\circ\mu^o$ can be identified with $\Pi\operatorname {I}\!\Gr^{\mathbf{H}}_{2k,k}$.
		\end{enumerate}
\end{proposition}

\begin{proof}
	The first statement follows from Corollary \ref{p:X}. In this case there no real points on the base space.

	Let us prove the second statement. We take a coordinate matrix $\mathcal Z_I$ of  $\PiG_{2k,k}$. The idea is to prove that the isotropy condition (\ref{eq isotropy condition}), see also (\ref{eq isotropy condition easy chart}), is equivalent to the condition $(c_k,\Theta)\circ\mu^o (\mathcal Z_I) =\mathcal Z_I$.
	Indeed, by Section \ref{append matrix b_k}, the cocycle $c_k$ may be represented by the matrix
	\begin{align*}
	H =
	\left(\begin{array}{cc}
	b_k & 0\\
	0 & b_k\\
	\end{array}
	\right), \quad  b_k=\left(\begin{array}{cc}
	0 & -iE_k\\
	iE_k & 0\\
	\end{array}
	\right).
	\end{align*}
We have
\begin{align*}
(c_k, \Theta) \circ\mu^o (\mathcal Z_I) =   H \cdot \Theta (\ov{\mathcal Z}_I) = H \cdot (\ov{\mathcal Z}_I^{\perp})^{t_i}.
\end{align*}	
Here we denoted by $\mathcal Z^{\perp}_I$ the result of application of the map (\ref{eq Z to Z perp}).  Further, we have
\begin{align*}
	(\ov{H \cdot (\ov{\mathcal Z}_I^{\perp})^{t_i}})^{t_i}  \cdot H\cdot \mathcal Z_I= (\mathcal Z_I^{\perp})  \ov{H}^{t} \cdot H \cdot \mathcal Z_I= - \mathcal Z_I^{\perp}\cdot \mathcal Z_I =0.
\end{align*}
Therefore the fixed point condition $H \cdot (\ov{\mathcal Z}_I^{\perp})^{t_i} = \mathcal Z_I$ is equivalent to the isotropy condition (\ref{eq isotropy condition})
\end{proof}

From Propositions  \ref{prop PiGr(H)2}, \ref{prop PiGr(R)}, \ref{prop PiG'(H)}, \ref{prop Pi Isotropic_super_Gra} and Theorem \ref{c:Pi}, we obtain the following result.

\begin{theorem}\label{theor real main} Let $\mcM = \PiG_{n,k}$. The ringed space of real points $\mcM^{\mu}$ corresponding to the real structures $\mu$ are as follows.
\begin{enumerate}
	\item[\rm (i)] If $n$ is odd, then
	$$
	\mcM^{\mu^o} =\PiG_{n,k}(\mathbb R)\quad  \text{and} \quad \mcM^{\ps\circ\mu^o} = \PiG'_{n,k}(\mathbb R).
	$$	
	
	\item[\rm (ii)] For $(n,k) = (2,1)$ we have
	$$
	\mcM^{\mu^o} =\PiG_{2,1}(\mathbb R)\quad  \text{and} \quad \mcM^{(c_J,1)\circ\mu^o} = \emptyset.
	$$	
	
	\item[\rm (iii)] If $n$ is even, $n\ne 2k$ and $k$ is odd,  we have
	\begin{align*}
	\mcM^{\mu^o} =\PiG_{n,k}(\mathbb R), \quad &\mcM^{\ps\circ\mu^o} = \PiG'_{n,k}(\mathbb R), \quad \mcM^{(c_J,1)\circ\mu^o} = \emptyset,\\
	&\mcM^{(c_J,\ps)\circ\mu^o} = \emptyset.
	\end{align*}
	
	\item[\rm (iv)] If $n,k$ are even and $n\ne 2k$, we have
	\begin{align*}
	\mcM^{\mu^o} =\PiG_{n,k}(\mathbb R), \quad &\mcM^{\ps\circ\mu^o} = \PiG'_{n,k}(\mathbb R), \quad \mcM^{(c_J,1)\circ\mu^o} = \PiG_{n/2,k/2}(\mathbb H),\\ &\mcM^{(c_J,\ps)\circ\mu^o} = \PiG'_{n/2,k/2}(\mathbb H).
	\end{align*}
	
	\item[\rm (v)] If $k$ is even and $n= 2k$, we have
	\begin{align*}
		\mcM^{\mu^o} =\PiG_{2k,k}(\mathbb R), \quad \mcM^{(c_J,1)\circ\mu^o} = \PiG_{k,k/2}(\mathbb H), \quad \mcM^{(c_r,\Theta)\circ\mu^o} = \emptyset, \,\,\, r=0,\ldots, k-1,\\ \mcM^{(c_k,\Theta)\circ\mu^o} = \Pi\operatorname {I}\!\Gr^{\mathbf{H}}_{2k,k}.
	\end{align*}
	
		\item[\rm (vi)] If $k\geq 3$ is odd and $n= 2k$, we have
		\begin{align*}
		\mcM^{\mu^o} =\PiG_{2k,k}(\mathbb R), \quad \mcM^{(c_J,1)\circ\mu^o} = \emptyset,
		\quad \mcM^{(c_r,\Theta)\circ\mu^o} = \emptyset, \,\,\, r=0,\ldots, k-1,\\ \mcM^{(c_k,\Theta)\circ\mu^o} = \Pi\operatorname {I}\!\Gr^{\mathbf{H}}_{2k,k}.
		\end{align*}
		
\end{enumerate}

\end{theorem}


\newcommand{\mm}{{\mu}}
\newcommand{\Am}{{\sf A}}
\newcommand{\am}{{a}}
\newcommand{\cm}{{{c}}}

\newcommand{\Pm}{{P}}	
\newcommand{\Pmm}{{P\hskip 0.4pt}}

\newcommand{\upgam}{{\hs^\gamma}}
\newcommand{\gm}{{\gamma}}
\newcommand{\sm}{{\sigma}}
\renewcommand{\cG}{{\hs_\cm G}}

\newcommand{\U}{{\rm U}}
\newcommand{\PU}{{\rm PU}}
\newcommand{\diag}{{\rm diag}}

\def\GmR{{\mathbb G}_{{\rm m},\R}}
\def\ctil{{\breve c}}
\def\gbar{{\bar g}}
\def\cF{{\mathcal F}}
\def\db{{\breve{d}}}

\def\ii{{\boldsymbol{i}}}

\newcommand{\GrI}{{\text{\rm IGr}}}

\appendix

\section{Real Galois cohomology}
\label{app:cohomology}
\medskip

\centerline{\em by Mikhail Borovoi}
\medskip

In this appendix  we prove Theorem \ref{t:Aut-M}
computing $H^1(\R, \Aut\,\M)$ where $M=\Pi\Gr_{n,k}(\C)$.
This result  gives us Theorem \ref{c:Pi}
classifying real structures on $\Pi\Gr_{n,k}(\C)$.
After that, we state and  prove  Proposition \ref{p:real-points},
Proposition  \ref{p:X}, and Corollary \ref{cor iso gr},
which compute the set of real points
of certain twisted forms of the real Grassmannian $\Gr_{n,\hs k,\hs\R}$\hs.

\begin{subsec}
	Let $\Gamma={\rm Gal}(\C/\R)$ denote the Galois group of $\C$ over $\R$.
	Then $\Gamma=\{1,\gamma\}$, where $\gamma$ is the complex conjugation.
	
	Let $(\Am,\sm)$ be a pair where  $\Am$ is a group (not necessarily abelian) and $\sm\colon \Am\to \Am$
	is an automorphism such that $\sm^2=\id_\Am$.
	We define a left action of $\Gamma$ on $\Am$ by
	\[ (\gm, a)\mapsto \upgam\hm a\coloneqq\sm(a)\ \, \text{for}\ a\in \Am.\]
	We say that $(\Am,\sm)$ is a {\em $\Gamma$-group.}
	
	We consider the set of $1$-{\em cocycles}
	\[ Z^1(\Am,\sm)=\{\cm\in \Am\,\mid\,\cm\cdot\hm\upgam\hm\cm =1\},\]
	where $\upgam\hm c=\sm(c)$.
	The group $\Am$ acts on the left on $Z^1(\Am,\sm)$ by
	\[\am\star\cm=\am\cdot\cm\cdot\hm\upgam\hm \am^{-1}\ \ \text{for}\  \am\in \Am,\ \cm\in Z^1(\Am,\sm).\]
	As in \cite{Serre}, Section I.5.1, we define the 1-{\em cohomology set}
	\[H^1(\Am,\sm)=Z^1(\Am,\sm)/\Am,\]
	to be the set of orbits of $\Am$ in $Z^1(\Am,\sm)$.
	We shall write $Z^1\Am$ for $Z^1(\Am,\sm)$ and $H^1 \Am$ for $H^1(\Am,\sm)$.
	We denote by $[\cm]\in H^1 \Am$ the cohomology class of a $1$-cocycle $\cm\in Z^1 \Am$.
	Note that $1\in Z^1 \Am$; we denote its class in $H^1 \Am$ by $[1]$ or just by $1$.
	In general (when $\Am$ is nonabelian) the cohomology set $H^1\Am$ has no natural group structure,
	but it has a canonical neutral element $[1]$.
	
\end{subsec}

\begin{rem}\label{r:H1-obvious}
	\begin{enumerate}
		\item If $(\Am,\sm)=(\Am',\sm')\times(\Am'',\sm'')$, then
		\begin{align*}
		Z^1(\Am,\sm)=Z^1(\Am',\sm')\times Z^1(\Am'',\sm'')\ \ \text{and}\ \ H^1(\Am,\sm)=H^1(\Am',\sm')\times H^1(\Am'',\sm'').
		\end{align*}
		
		\item $H^1(\{\pm 1\}, \id)=Z^1(\{\pm 1\}, \id)=\{\pm 1\}$.
		
		\item $H^1\big(\C^*,(z\mapsto \bar z)\big)=\{1\}$ (exercise); this is a special case of Hilbert's Theorem 90.
	\end{enumerate}
\end{rem}

\subsec{}

Let $G$ be a linear algebraic group defined over $\R$ (we write ``an $\R$-group").
We denote by $G(\R)$ and $G(\C)$ the groups of $\R$-points and of $\C$-points of $G$, respectively.
The nontrivial element $\gamma\in\Gamma$ acts on $G(\C)$ by the complex conjugation $\sm\colon g\mapsto\bar g$.
We write $H^1(\R,G)=H^1(G(\C),\sm)$.
For simplicity, we write $H^1\hs G$ or $H^1\hs G(\C)$ for $H^1(\R,G)$.

\begin{subsec}
	Consider the short exact sequence of $\G$-groups
	\begin{equation}\label{e:exact-GL}
	1\to \C^*\to\GL(n,\C)\labelto\pi\PGL_n(\mathbb C)\to 1,
	\end{equation}
	where $\gamma\in\G={\rm Gal}(\C/\R)$ acts on $g\in\GL(n,\C)$
	by the complex conjugation $g\mapsto \bar g$.
	This exact sequence induces a cohomology exact sequence
	\[1=H^1\GL(n,\C)\to H^1\PGL_n(\mathbb C)\labelto\Delta H^2\, \C^*;\]
	see Serre's book \cite{Serre}, Section I.2.2 for the definition of $H^2$,
	and Proposition 43 in Section I.5.7 for the exact sequence.
	See also \cite{BTb}, Section 1.1 and Construction 4.4.
\end{subsec}

\begin{prop}[well-known]
	\label{p:H1}
	\[\#H^1\PGL_n=
	\begin{cases} &\!\!1\qquad\text{if $n$ is odd}\\&\!\!2 \qquad\text{if $n$ is even} \end{cases}
	\]
	If $n$ is even, write
	\[ a_J={\rm diag}(J,\dots,J)\ \,\text{($n/2$ times)\ \ where}\quad J=\SmallMatrix{0 &1\\-1&0}.\]
	Then $\cm:=\pi(a_J)\in \PGL_n(\mathbb C)$ is a cocycle representing the nontrivial cohomology class in $H^1\PGL_n$.
\end{prop}

\begin{proof}
	The cardinality $\#H^1 \PGL_n$ can be computed, for instance, using \cite[Corollary 13.6]{BT}.
	If $n$ is even, then we  have $a_J\hs\cdot \hm\upgam\hm a_J=a_J^2=-1$.
	It follows that $\cm\cdot\hm\upgam\hm\cm=\pi(-1)=1$.
	These formulas mean that $\cm$ is a 1-cocycle and that $\Delta[\cm]=[-1]\neq [1]$, whence $[\cm]\neq [1]$, as required.
\end{proof}

\begin{subsec}\label{ss:cp}
	Consider the short exact sequence of real algebraic groups
	\[1\to \U_1\to \U_n\labelto\pi\PU_n\to 1,\]
	where $\U_n$ is the unitary group, that is,
	$\GL(n,\C)$ with complex conjugation $\sigma$ acting
	by  $g\mapsto (\bar g^t)^{-1}$ where $\bar g^t$ denotes the transposed matrix
	to the complex conjugate matrix $\bar g$.
	This exact sequence is the exact sequence \eqref{e:exact-GL}
	of complex algebraic groups, but with another complex conjugation.
	Consider the cocycles
	$a_p=\diag(-1,\dots,-1,1,\dots,1)\in Z^1\U_n$ where  $-1$ appears $p$ times and $1$ appears $n-p$ times,  $0\le p\le n$.
	Set $c_p=\pi(a_p)\in Z^1\PU_n$.
\end{subsec}

\begin{prop}[well-known]
	\label{p:H1-U}
	\[H^1\PU_n=\{\hs[c_p]\ |\ 0\le p\le n/2\}.\]
	In particular, when  $n=2k$, we have $\# H^1 \PU_n=k+1$.
\end{prop}

\begin{proof}[Idea of proof]
	Since $\PU_n$ is a compact group, one can compute $\#H^1\PU_n$ by the method of Borel and Serre.
	See Serre \cite[Section III.4.5,  Example (a) for Theorem 6]{Serre},
	which says that $H^1\PU_n=T_2/W$, where $T$ is the diagonal maximal torus,
	$T_2$ is the set of its elements of order dividing 2,
	and $W=W(\PU_n, T)\cong S_n$ is the corresponding Weyl group
	(which is isomorphic to the symmetric group on $n$ symbols).
\end{proof}

\begin{subsec}\label{ss:semi-direct}
	Consider the cyclic group  $C_4$ of order $4$ with generator $\Theta$. Define an action of $C_4$ on
	$\GL(n,\C)$ by the formula
	\begin{equation*}
	\Theta(g)=g^{-t}\quad\ \text{for}\ \, g\in \GL(n,\C)
	\end{equation*}
	where $g^t$ denotes the transposed matrix to $g$
    and we write
    \[g^{-t}\coloneqq (g^{-1})^t=(g^t)^{-1}.\]
	Note that $\Theta^2$ acts on $\GL(n,\C)$ trivially.
	Set
	\[\tilde\AA=\GL(n,\C)\rtimes C_4.\]
	Define an action of $\Gamma=\{1,\gamma\}$ on $C_4$ by $\upgam (\Theta^r)=\Theta^{-r}$ for $r\in \Z$,
	and let $\Gamma$ act on $\GL(n,\C)$ by the usual formula $\upgam\hm g=\bar g$, the bar denoting complex conjugation.
	Then
	\[\upgam(\Theta(g))=\ov{(g^{-t})}=\bar g\,^{-t}=\Theta(\upgam\hm g)=\Theta^{-1}(\upgam\hm g)=\upgam \Theta(\upgam\hm g),\]
	which shows that the actions of $\gamma$ on $C_4$ and $\GL(n,\C)$ are compatible, and so $\Gamma$ naturally acts on $\tilde \AA$.
	The center  $Z=Z(\GL(n,\C))$ (consisting of the scalar matrices) is clearly $\Theta$-invariant and $\Gamma$-invariant.
	We set
	\[\AA=\tilde\AA/Z=\PGL(n,\C)\rtimes C_4\hs.\]
\end{subsec}

\begin{lem}\label{l:H1:n=2k}
	Consider the $\Gamma$-group $\AA=\PGL(n,\C)\rtimes C_4$ as above, where $n=2k\ge 4$.
	Then $\#H^1\AA=k+3$ with cocycles
	\begin{equation*}
	(1,1),\ (c_J,1),\ (c_0,\Theta),\ (c_1,\Theta),\ \dots,\ (c_k,\Theta)
	\end{equation*}
	where $c_J\in Z^1 \PGL_{n,\R}\subset\PGL(n,\C)$ is as in Proposition \ref{p:H1},
	and $c_p\in Z^1\PU_n\subset\PGL(n,\C)$ for $p=0,1\dots,k$ is as in Proposition \ref{p:H1-U}.
\end{lem}

\begin{proof}
	Consider the short exact sequence
	\begin{equation}\label{e:exact-A}
	1\to \AA_0\to\AA\labelto\pi C_4\to 1
	\end{equation}
	where $\AA_0=\PGL(n,\C)$,
	and consider the induced cohomology exact sequence
	\begin{equation}\label{e:cohom-A}
	H^1\AA_0\to  H^1\AA\labelto{\pi_*} H^1 C_4\hs.
	\end{equation}
	We see from the definition of $H^1$ that
	$H^1 C_4=C_4/\langle\Theta^2\rangle\cong C_2$ with cocycles $1,\Theta$.
	Since the exact sequence \eqref{e:exact-A} splits,
	we see that the map $\pi_*$ in \eqref{e:cohom-A} is surjective.
	Thus
	\[H^1\AA=\pi_*^{-1}[1]\cup\pi_*^{-1}[\Theta].\]
	
	We compute $\pi_*^{-1}[1]=\ker\pi_*$.
	By Serre \cite[Section I.5.5, Corollary 1 of Proposition 39]{Serre},
	the inclusion map
	\[\AA_0\into\AA,\quad a\mapsto (a,1)\quad\text{for}\ \, a\in \AA_0=\PGL(n,\C)\]
	induces a bijection
	\[(H^1\AA_0)/C_4^\Gamma\isoto\ker\pi_*\hs.\]
	Here $C_4^\Gamma=\{1,\Theta^2\}$.
	It acts on the right on $H^1\AA_0$ as follows.
	Let $a\in Z^1 \AA_0$; then
	\[ [a,1]*\Theta^2=\big[(1,\Theta^2)^{-1}\cdot (a,1)\cdot\upgam(1,\Theta^2)\big].\]
	Since $(1,\Theta^2)$ is of order 2, central in $\AA$, and $\Gamma$-fixed, we have
	\begin{equation*}
	(1,\Theta^2)^{-1}\cdot (a,1)\cdot\upgam(1,\Theta^2)
	=(a,1).
	\end{equation*}
	Thus $C_4^\Gamma$ acts on $H^1\AA_0$ trivially, and $\ker\pi_*\cong H^1\AA_0$.
	By Proposition \ref{p:H1} we obtain that
	\[\#\ker\pi_*=2 \quad\ \text{with cocycles}\ \,(1,1),\ (c_J,1).\]
	
	We compute $\pi_*^{-1}[\Theta]$.
	By Serre \cite[Section I.5.5, Corollary 2 of Proposition 39]{Serre},
	the inclusion map (not a homomorphism)
	\[\AA_0\into\AA,\quad\ a\mapsto (a,\Theta)\ \,\text{for}\ \,a\in\AA_0\]
	induces a bijection
	\[(H^1\,_\Theta\AA_0)\hs/\,(_\Theta C_4)^\Gamma\,\isoto\,\pi_*^{-1}[\Theta].\]
	Since $C_4$ is an abelian group, we have $_\Theta C_4=C_4$ and $(_\Theta C_4)^\Gamma= C_4^\Gamma=\{1,\Theta^2\}$.
	It acts on the right on $H^1\hs_\Theta\AA_0$ as follows.
	Let $a\in Z^1 \hs_\Theta\AA_0$; then
	\[ [a,\Theta]*\Theta^2=\big[\hs(1,\Theta^2)^{-1}\cdot (a,\Theta)\cdot\upgam(1,\Theta^2)\hs\big].\]
	We calculate:
	\begin{equation*}
	(1,\Theta^2)^{-1}\cdot (a,\Theta)\cdot\upgam(1,\Theta^2)=
	(1,\Theta^2)^{-1}\cdot (a,\Theta)\cdot(1,\Theta^2)=(\Theta^{-2}(a),\Theta)=(a,\Theta).
	\end{equation*}
	Thus $(\hs_\Theta C_4)^\Gamma$ acts on $H^1\hs_\Theta\AA_0$ trivially, and $\pi_*^{-1}[\Theta]\cong H^1\hs_\Theta A_0$.
	By Proposition \ref{p:H1-U} we obtain that
	\[\#\pi_*^{-1}[\Theta]=k+1 \quad\ \text{with cocycles}\ \,(c_0,\Theta),\ \dots,\ (c_k,\Theta).\]
	which completes the proof of the lemma.
\end{proof}

\begin{subsec}\label{ss:cohom-vs-eq}
	Let $\M$ be a complex supermanifold, and $\mm$ be a fixed real structure on $\M$ of the type $(1,-1,1)$.
	Then  $\mu^2=\id_\M$.
	Write $\Am=\Aut\, \M$, the group of even holomorphic automorphisms of $\M$.
	For $a\in\Am$, we write
	\[\sm(a)=\mu a\mu^{-1}=\mu a \mu.\]
	Then
	\[\sm(\sm(a))=a\quad \text{for all}\ a\in\Am.\]
	We obtain a $\Gamma$-group $(\Am,\sm)$, where
	$\Gamma=\{1,\gamma\}$ acts on $\Am$ by
	\[\upgam\hm a=\sm(a)=\mu a\mu^{-1}\quad \text{for}\ a\in\Am.\]

	Let  $\mm'$ be any  real structure on $\M$ of the type $(1,-1,1)$.
	Write $\mm'=\cm\circ\mm$; then $\cm\in\Aut\,\M$ is an (even) holomorphic  automorphism.
	We have
	\[1=\mm^{\prime\hs\hs2}=(\cm\mm)^2=\cm\mm\cm\mm
	=\cm\cdot\mm\cm\mm^{-1}\cdot \mm^2=\cm\cdot\hm\upgam\hm \cm.\]
	We see that the condition $\mm^{\prime\hs\hs2}=1$ implies the {\em cocycle condition}
	\begin{equation}\label{e:cocycle}
	\cm\cdot\hm\upgam\hm \cm=1.
	\end{equation}
	Conversely, if an automorphism $\cm\in\Aut\,\M$ satisfies \eqref{e:cocycle}, and $\mm'=\cm\mm$, then
	\[\mm^{\prime\hs\hs2}=\cm\mm\hs\cm\mm=\cm\cdot\mm\cm\mm^{-1}\cdot\mm^2
	=(\cm\cdot\hm \upgam \cm)\cdot\mm^2=1,\]
	and hence $\mm'$ is a real structure on $\M$.
	
	Now let
	$\mm_1=\cm_1\circ\mm$ and $\mm_2=\cm_2\circ\mm$ be two real structures on $\M$.
	Assume that the real structures $\mm_1$ and $\mm_2$ are {\em equivalent}, that is,
	$(\M,\mm_1)\simeq (\M,\mm_2)$. This means that
	there exists an automorphism $\am\colon \M\to\M$
	such that
	\begin{equation}\label{e:rho2-rho1}
	\mm_2=\am\circ\mm_1\circ\am^{-1}.
	\end{equation}
	Then
	\[\cm_2\hs\mm=\am\cdot\cm_1\hs\mm\cdot \am^{-1}
	=\am\cdot\cm_1\cdot\hm\sm(\am)^{-1}\cdot\mm,\]
	whence
	\begin{equation}\label{e:equiv}
	\cm_2=\am\cdot\cm_1\cdot\hm\sm(\am)^{-1}=\am\cdot\cm_1\cdot\hm\upgam \am^{-1}.
	\end{equation}
	Conversely, if \eqref{e:equiv} holds for some $\am\in\Aut\,\M$,
	then \eqref{e:rho2-rho1} holds.
	This means that  $\am$ is an isomorphism   $(\M,\mm_1)\isoto (\M,\mm_2)$,
	and therefore the real structures $\mm_1$ and $\mm_2$ are equivalent.
	
	We can state the  results of this subsection as follows:
\end{subsec}

\begin{prop}\label{p:Serre-reduction}
	For $\M$,  $\mm$, and the $\Gamma$-group $(\Am,\sm)$  as in Subsection \ref{ss:cohom-vs-eq},
	define an action of  the group $\Gamma=\{1,\gamma\}$ on $\Am$ by
	\[\upgam\hm \am=\sm(\am)\coloneqq  \mm\cdot\am\cdot\mm^{-1}\ \, \text{for}\  \am\in \Am.\]
	Then:
	\begin{enumerate}
		\item[(i)] The map
		\[\cm\mapsto \cm\circ \mm\quad \text{for}\ \cm\in Z^1(\Am,\sm)\]
		is a bijection between the set of $1$-cocycles $Z^1(\Am,\sm)$ and the set of real structures on $\M$.
		\item[(ii)] This map induces a bijection between $H^1(\Am,\sm)$
		and the set of equivalence classes of real structures on $\M$, sending $[1]\in H^1(\Am,\sm)$ to the equivalence class of $\mm$.
	\end{enumerate}
\end{prop}

This proposition is similar to Proposition 5 in Section III.1.3 of Serre's book \cite{Serre}.

Now let $\M$ be the $\Pi$-symmetric Grassmannian  $\PiG_{n,k}$. Theorem \ref{t:Aut} describes the automorphism group $\Aut\hs\M=\Aut(\PiG_{n,k})$.
We wish to compute $H^1\hs\Aut(\PiG_{n,k})$ and to classify the real structures on $\PiG_{n,k}$.
We use the notation of Propositions \ref{p:H1} and \ref{p:H1-U}.

\begin{thm}\label{t:Aut-M}
	The list below gives the cardinality $\#H^1\hs\Aut(\PiG_{n,k})$ and a set of representing cocycles
	for all cohomology classes in $H^1\hs\Aut(\PiG_{n,k})$:
	\begin{enumerate}

		\item[\rm (i)] Case $n$ is odd: $\#H^1=2$ with representatives:
\[(1,1), (1,\ps)\,\in\, \PGL_n(\C)\times \{\id,\ps\}.\]

		\item[\rm (ii)] Case $n$ is even, $n\neq 2k$: $\#H^1=4$ with representatives:
\[(1,1),(1,\ps),(\cm_J,1),(\cm_J,\ps)\,\in\, \PGL_n(\C)\times \{\id,\ps\}.\]

		\item[\rm (iii)] Case $n=2k\ge 4$: $\#H^1=k+3$ with representatives:
\[(1,1),(\cm_J, 1),\,(c_0,\Theta), (c_1,\Theta),\dots,(c_k,\Theta)\,\in\,\PGL_{2k}(\C)\rtimes \{\id,\Theta,\ps,\ps\circ\Theta\}.\]

		\item[\rm (iv)] Case $n=2$, $k=1$: $\#H^1=2$ with representatives:
\[(1,1),(\cm_J,1)\,\in\,\PGL_2(\C)\times\C^*.\]
	\end{enumerate}
\end{thm}

\begin{proof}
	(i) and (ii) follow   from  Theorem \ref{t:Aut}(1),  Remark \ref{r:H1-obvious}(1), and Proposition \ref{p:H1}.
	
	(iii) follows from Theorem \ref{t:Aut}(2) and Lemma \ref{l:H1:n=2k}.
	
	(iv) for $n=2$, $k=1$ by Theorem \ref{t:Aut}(3)  we have $\Aut(\PiG_{n,k})\cong\GmR\times\PGL_{2,\R}$.
	By Hilbert's Theorem 90, we have $H^1(\R,\GmR)=1$. It is well known that $\#H^1(\R,\PGL_{2,\R})=2$
	with cocycles $1,\cm_J$; see, for instance, \cite{Borovoi22-CiM}, Theorem 3.1.
\end{proof}

We denote by $\GL_{n,\R}$\hs, $\GL_{n'\!,\hs\H}$\hs, etc. algebraic $\R$-groups,
and by  $\GL(n,\R)$, $\GL(n'\!,\hs\H)$, etc. their groups of $\R$-points.
Here $\H$ is the division algebra of Hamilton's quaternions.

\begin{lem}[well-known]
	Let $G=\GL_{n,\R}$ and assume  that $n$ is even, $n=2n'$.
	Let $\cm=\pi(a_J)\in Z^1\big(\R,\PGL_n(\mathbb C)\big)$ be as in Proposition \ref{p:H1}.
	Then the twisted group $_\cm G$ is isomorphic to $\GL_{n'\!,\hs\H}$.
	In other words,
	\[\big\{g\in\GL(n,\C)\ \big|\  \cm\cdot\bar g\cdot \cm^{-1}=g\big\}\cong \GL(n',\H).\]
\end{lem}

\begin{proof}
	Let $M_n(\R)$ denote the algebra of $n\times n$-matrices over $\R$.
	Then $G(\R)=\GL(n,\R)=M_n(\R)^*$.
	In order to compute $_\cm G$, it suffices to compute
	the twisted algebra $_\cm M_n$.
	
	Let $X\in M_n(\C)$. We write $X$ as a block matrix
	\[ X=(X_{ij})_{1\le i,j\le n'}\hs,\quad\text{where}\ X_{ij}\in M_2(\C).\]
	Then $X\in \hs_\cm M_n(\R)$ if and only if
	\[a_J\cdot\ov X\cdot a_J^{-1}=X,\]
	that is
	\begin{equation}\label{e:Xij}
	J\cdot\ov X_{ij}\cdot J^{-1}=X_{ij}\quad\text{for all}\  i,j.
	\end{equation}
	
	We write $Y$ for $X_{ij}$. Then condition \eqref{e:Xij} is equivalent to
	\begin{equation}\label{e:Y}
	J\cdot \ov Y=Y\cdot J.
	\end{equation}
	An easy calculation shows that  condition \eqref{e:Y} on $Y$ means that
	\[Y=\Mat{u &v\\ -\bar v &\bar u},\quad\text{where}\ u,v\in\C.\]
	In other words,
	\[ Y=\lambda_1 +\lambda_2 \ii+\lambda_3 \jj+\lambda_4 \kk,\]
	where $\lambda_1, \dots,\lambda_4\in\R$ and
	\begin{equation}\label{e:ijk}
	\ii=\Mat{0&1\\-1&0},\quad \jj=\Mat{0&i\\i&0},\quad \kk=\Mat{i&0\\0&-i}.
	\end{equation}
	One checks that
	\[\ii^2=-1=\jj^2,\qquad \ii\jj=\kk=-\jj\ii.\]
	Thus
	\[(\hs_\cm M_2)(\R)\cong \H,\quad (\hs_\cm M_n)(\R)\cong M_{n'}(\H),
	\quad (\hs_\cm\!\GL_n)(\R)\cong\GL_{n'}(\H),\]
	and hence $_\cm\!\GL_n\cong \GL_{n'\!,\hs\H}$\hs, as required.
\end{proof}

\begin{prop}\label{p:real-points}
	Let $n=2n'$ be an even natural number. Consider the Grassmann variety
	$X=\Gr_{n,k,\R}$ over $\R$ and its twisted form
	$_\cm X\coloneqq(\Gr_{n,k}\hs,\hs\cm\circ\mm)$, where $\cm=\pi(a_J)\in Z^1\PGL_n(\mathbb C)$
	is as in  Proposition \ref{p:H1}, and $\mm$ is the standard complex conjugation in $\Gr_{n,k,\C}$\hs.
	\begin{enumerate}
		\item[\rm(i)] if $k$ is even, $k=2k'$, then the set of real points $(_\cm X)(\R)$
		is in a canonical bijection with the set $\Gr_{n'\!,\hs k'}(\H)$
		of quaternionic $k'$-dimensional subspaces in the quaternionic $n'$-dimensional space  $\H^{\hs n'}$.
		In other words,
		\[\big\{x\in\Gr_{n,k}(\C)\ \big|\  \cm\cdot \bar x=x\big\}\cong \Gr_{n'\!,\hs k'}(\H).\]
		
		\item[\em(ii)] If $k$ is odd, then the set $(_\cm X)(\R)$ is empty.
	\end{enumerate}
\end{prop}

\begin{proof}		
	\label{ss:real-points}
	Write $V=\R^n$ with canonical basis $e_1,\dots,e_n$.
	Write $G=\GL_{n,\R}$,
	$X=\Gr_{n,k,\R}$.
	To any complex point $x\in X(\C)$ we assign its stabilizer $\Pm_x\subset G_\C$.
	Let $x_0\in X(\R)$ denote the point corresponding to the $k$-dimensional subspace
	$W_0=\langle e_1,\dots,e_k\rangle\subset V$, and write $\Pm_0=\Stab_{G_\C}(x_0)$.
	Then
	\[\Pm_0=\left\{\Mat{A&B\\0&D}\ \bigg|\ A\in M_k(\C)\right\}.\]
	We have a canonical {\em $\G$-equivariant} bijection
	\[x\mapsto \Pm_x,\quad\ x\in X(\C),\ \Pm_x\subset G_\C,\ \text{$\Pm_x$ is conjugate to $\Pm_0$.}\]
	We twist $G$ and $X$ with the same cocycle $\cm=\pi(a_J)\in Z^1\PGL_n(\mathbb C)$; then
	the map $x\mapsto \Pm_x$ is $\G$-equivariant also with respect to the twisted real structures
	corresponding to the twists $_\cm X$ and $_\cm G$.
	
	If $k$ is even, $k=2k'$, we consider the subgroup
	\[ \Pmm_{\H,0}=\left\{\Mat{A&B\\0&D}\ \bigg|\ A\in M_{k'}(\H)\right\}\subset \GL_{n'\hm,\H}.\]
	We embed $\H$ into $\M_2(\C)$ using the formulas \eqref{e:ijk}.
	In this way we obtain a canonical isomorphism
	\[\GL_{n'\hm,\H}\times_\R\hs\C\isoto \GL_{n,\C}\]
	sending $\Pmm_{\H,0}\times_\R\C$ to $\Pm_0$.
	Note that $H^1\hs \Pmm_{\H,0}=1$.
	Since the subgroup $\Pmm_{\H,0}\subset\hs_\cm G$ is defined over $\R$,
	we see that the corresponding point $x_0\in X(\C)=(\hs_\cm X)(\C)$ is defined over $\R$.
	We conclude that $(\hs_\cm X)(\R)$ is nonempty.
	
	To $x_0$ we assign the subspace
	\[W_{\H,0}=\langle e_1,\dots,e_{k'}\rangle \subset \H^{\hs n'}.\]
	For any $k'$-dimensional quaternionic subspace $W_\H\subset \H^{n'}$
	there exists an invertible matrix $g\in \GL(n',\H)$ such that $W_\H=g\cdot W_{\H,0}$,
	and we obtain a subgroup $\Pm:=g\hs \Pm_0\hs g^{-1}\subset (\hs_\cm G)$
	that is defined over $\R$ and is conjugate over $\R$ to $\Pm_0$.
	To this subgroup we assign the real point $x=g\cdot x_0\in (\hs_\cm X)(\R)$ with stabilizer $\Pm$.
	
	Conversely, if $x\in(\hs_\cm X)(\R)$, we set $\Pm_x=\Stab\hs_{_\cm G}\hs(x)\subset \hs_\cm G$.
	Then $\Pm_x\subset \hs_\cm G$ is a subgroup
	defined over $\R$ and conjugate to $\Pm_0=\Pmm_{\H,0}$ over $\C$.
	Set
	\[T_x=\{g\in (\hs_\cm G)(\C)\mid g\hs \Pm_0\hs g^{-1}=\Pm_x\}.\]
	This variety is clearly defined over $\R$ and nonempty over $\C$.
	Moreover, it is a torsor (principal homogeneous space)
	under the normalizer $N$ of $\Pmm_{\H,0}$ in $\hs_\cm G$.
	Since $N=\Pmm_{\H,0}$ and $H^1\hs \Pmm_{\H,0}=1$, we conclude that $T_x$ has an $\R$-point $g_x$;
	see Serre \cite[Section I.5.2]{Serre}.
	Thus there exists $g_x\in (\hs_\cm G)(\R)$ such that $\Pm_x=g_x\cdot \Pmm_{\H,0}\cdot g^{-1}$.
	To $x$ we assign the subspace $W_x=g_x\cdot W_{\H, 0}\subset \H^{\hs n'}$
	(which does not depend on the choice of $g_x\in T_x(\R)$).
	Thus we obtain a bijection
	\[(\hs_\cm X)(\R)\to \Gr_{n'\!,\hs k'}(\H),\quad\ x\mapsto (\Pm_x,\hs g_x)\mapsto g_x\cdot W_{\H,0}\hs,\]
	which proves assertion (i) of Proposition \ref{p:real-points}.
	
	If $k$ is odd, let $x\in\Gr_{2n'\hm, k\hs}(\C)$ be a $\C$-point.
	The stabilizer $\Pm_x$ of $x$ is a parabolic subgroup of {\em odd} codimension $k(2n'-k)$.
	On the other hand, by Lemma \ref{l:even} below,
	any {\em defined over $\R$} parabolic subgroup  of $_\cm G$ has  {\em even} codimension.
	It follows the parabolic subgroup $\Pm_x$ is not defined over $\R$, and hence $x$ is not an $\R$-point.
	We conclude that when $k$ is odd, $_\cm X$ has no real points,
	which proves assertion (ii) of Proposition \ref{p:real-points}.
\end{proof}

\begin{lem}\label{l:even}
	Set $Q\subseteq \cG\cong\GL_{n'\hm,\hs\H}$ be a parabolic $\R$-subgroup.
	Then the codimension of $Q$ in $\cG$ is divisible by 4.
\end{lem}

\begin{proof}[Idea of proof]
	Let $S\subset\cG$ be the $\R$-subtorus such that
	\[S(\R)=\big\{{\rm diag}(\lambda_1,\dots,\lambda_{n'})\ |\ \lambda_i\in\R^*\big\}.\]
	Then $S$ is a maximal split $\R$-torus in $\cG$.
	Consider the relative root system $R(\cG, S)$.
	It is easy to see that the root subspace $\g_\beta$ for any root $\beta\in R(\cG,S)$ has dimension 4.
	Let $P\subset \cG$ be the $\R$-parabolic (parabolic $\R$-subgroup) such that $P(\R)$
	is the group of upper triangular quaternionic matrices in $\GL(n',\H)$;
	then $P$ is a minimal $\R$-parabolic in $\cG$.
	Let $P'\supseteq P$ be any {\em standard $\R$-parabolic} in $\cG$ with respect to $S$ and $P$
	in the sense of Borel and Tits \cite[Section 5.12]{Borel-Tits}; then
	\[\Lie(\cG)=\Lie(P')\oplus\bigoplus_{\beta\in\Xi}\g_\beta\hs,\]
	where $\Xi\subseteq R(\cG,S)$ is some subset.
	It follows that the codimension of $P'$ is divisible by 4.
	By \cite[Proposition 5.14]{Borel-Tits}, any  $\R$-parabolic $Q$ of $\cG$
	is conjugate (over $\R$) to a standard $\R$-parabolic,
	and the lemma follows.
\end{proof}

\begin{subsec}
	If $V$ is a complex vector space, $W\subset V$ a subspace, and $B$
	is a bilinear form (or a Hermitian form) on $V$,
	then we write
	\[ W^{\bot B}=\{x\in V\ |\ B(x,y)=0\ \,\text{for all}\ \,y\in W\}\]
	for the annihilator of $W$ in $V$ with respect to $B$.

	Let $n=2k$.
	The group $\AA=\PGL_{n,\R}\rtimes C_4$ of Subsection \ref{ss:semi-direct},
	where $C_4=\{1,\Theta,\Theta^2,\Theta^3\}$,
	acts by conjugation on its identity component $G=\PGL_{n,\R}$.
	Therefore, for any cocycle $c\in Z^1\AA$ we obtain a twisted group $_c G$.
	
	Moreover, the group $\AA$ naturally acts on $X\coloneqq \Gr_{n,k}(\C)$.
	Namely, for a $k$-dimensional subspace $W$ of $V=\C^n$,
	and $[g]\in\PGL(n,\C)$, the class of $g\in\GL(n,\C)$,
	we have  $[g]*W=g(W)$.
	Furthermore, the generator $\Theta$ of $C_4$
	sends $W$ to $W^{\bot B_0}$,
	the annihilator of $W$ in $V=\C^n$
	with respect to the symmetric bilinear form $B_0$
	with matrix $\diag(1,\dots,1)$.
	
	We show that
	this action of $\AA$ on $X(\C)$ is well defined.
	For $g\in \PGL(n,\C)$ we compute
	\begin{align*}
	\Theta(g(W))&=g(W)^{\bot B_0}\\
	&=\{x\in V\ |\ B_0(x,gW)=0\}\\
	&=\{x\in V\ |\ B_0(g^t x,W)=0\}\\
	&=\{g^{-t}y\ |\ B_0(y,W)=0\}
	\qquad
	y=g^t x,\ \,x=g^{-t}y\\
	&=g^{-t} \cdot W^{\bot B_0}\\
	&=\Theta(g)\cdot\Theta(W).
	\end{align*}
	Thus
	\[\Theta(g(W))=\Theta(g)\cdot\Theta(W)\]
	as required.

	For any cocycle $c\in Z^1\AA$ we obtain a twisted variety $_c X$.
\end{subsec}

\begin{lem}
	Consider an element $c=(F,\Theta)$ where $F\in \PGL(n,\C)$.
	\begin{enumerate}
		\item[\rm (i)] $c$ is a cocycle if and only if $\ov F\hs^t=F$.
		
		\item[\rm (ii)] If  $c=(F,\Theta)$ is a cocycle and $c'=(F',\Theta)$ with $F'=gF\gbar^t$, then $c'\sim c$.
	\end{enumerate}
\end{lem}

\begin{proof}
	(i) We write the cocycle condition $c\bar c=1$:
	\[(F,\Theta)\cdot(\ov F,\Theta^{-1})=(1,1),\]
	that is,
	\[F\cdot \ov F\hs^{-t}=1.\]
	Thus
	\[ \ov F\hs^t=F,\]
	as required.
	
	(ii) We consider the equivalent cocycle
	\begin{align*}
	(g,1)\cdot (F,\Theta)\cdot\ov{(g,1)}\hs^{-1}
	&= (g,1)\cdot \big(F\cdot (\gbar\hs^{-1})^{-t}, \Theta\big)\\
	&=\big(g\!\cdot\! F\!\cdot\!\gbar\hs ^t,\Theta\big)=(F',\Theta),
	\end{align*}
	as required.
\end{proof}

Let $c=(F,\Theta)$ be a cocycle.  We compute the twisted group $_c G$ for $G=\PGL_{n,\R}$
and the twisted manifold $_c X$ for $X=\Gr_{n,k,\R}$.

\begin{lem}\label{l:cG(R)}
	For a cocycle $c=(F,\Theta)$ consider the twisted group $_c G$ with $G=\PGL_{n,\R}$.
	Then
	\[(\hs_c G)(\R)=\{g\in \PGL(n,\C)\ |\ g\hs F\hs\bar g\hs^t=F\}. \]
\end{lem}

\begin{proof}
	
	First, note that $c^{-1}=(F^t,\Theta^{-1})$. Indeed,
	\[(F,\Theta)\cdot(F^t,\Theta^{-1})=(F\cdot(F^t)^{-t},1)=(F\cdot F^{-1},1)=(1,1).\]
	
	Now,
	\begin{align*}
	(\hs_c G)(\R)
	&=\big\{g\ |\ c(\gbar,1) c^{-1}=(g,1)\big\}\\
	&=\big\{g\ |\ (F,\Theta)\cdot(\gbar,1)\cdot(F^t, \Theta^{-1})= (g,1)\big\}.
	\end{align*}
	We calculate:
	\begin{align*}
	(F,\Theta)\cdot(\gbar,1)\cdot(F^t, \Theta^{-1})
	&=(F,\Theta)\cdot (\gbar F^t,\Theta^{-1})\\
	&=(F\cdot(\gbar F^t)^{-t},1)\\
	&=(F\hs\gbar\hs^{-t} F^{-1},1).
	\end{align*}
	We obtain
	\begin{align*}
	(\hs_c G)(\R)&=\{g\in\PGL(n,\C)\ |\ g=F\hs\gbar\hs^{-t} F^{-1}\}\\
	&=\{g\in \PGL(n,\C)\ |\ g F \gbar^t=F\},
	\end{align*}
	as required.
\end{proof}

\begin{lem}\label{l:cX(R)}
	Let $c=(F,\Theta)$, $\ov F\hs^t=F$ (then $F,\Theta)$ is a 1-cocycle).
	Then
	$$(\hs_c X)(\R)=\{W\subset V\ |\ \ \dim W=k\ \ \text{and}\ \  \bar x\hs^t\hs F^{-1}y=0\ \text{for all}\ \,x,y\in W\}.$$
\end{lem}

\begin{proof}
	We write
	\[(\hs_c X)(\R) = \{W\in\Gr_k(V)\ \,|\ \,(F,\Theta)\cdot\ov W=W\}.\]
	We have
	\[ (F,\Theta)\cdot\ov W=F\cdot \ov W^{\bot B_0}.\]
	Thus the condition for $W$  to lie in $(\hs_c X)(\R)$ is
	\begin{equation}\label{e:F-W-bar}
	F\cdot \ov W^{\bot B_0}=W,
	\end{equation}
	or, equivalently,
	\begin{equation}\label{e:Wbar-bot--B0}
	\ov W^{\bot B_0}=F^{-1}\cdot W.
	\end{equation}
	This implies
	\[B_0(\bar x,F^{-1}\cdot y)=0\quad\ \text{for all}\ \, x,y\in W,\]
	that is,
	\begin{equation}\label{e:xF-1y}
	\bar x \hs^t F^{-1} y=0\quad\ \text{for all}\ \, x,y\in W.
	\end{equation}
	Conversely, if \eqref{e:xF-1y} holds,
	then
	\[\ov W^{\bot B_0}\hs\supseteq F^{-1}\cdot W,\]
	and comparing the dimensions, we obtain in turn the equalities \eqref{e:Wbar-bot--B0} and \eqref{e:F-W-bar},
	which gives $W\in (\hs_cX)(\R)$, as required.
\end{proof}

\begin{cor}
	\label{l:G}
	Let $G=\PGL_{n,\R}$. For $0\le p\le n$,
	let the cocycle  $\ctil_p=(\pi(a_p),\Theta)\in Z^1\AA$ be as in Lemma \ref{l:H1:n=2k}.
	Write $_p G=\hs_{\ctil_p}G$.
	Then $_pG\cong \PU(p,n-p)$; namely,
	\[(\hs_pG)(\R)\cong \{g\in\PGL(n,\C)\ |\ \bar g^t\cdot c_p\cdot g=c_p\}.\]
\end{cor}

\begin{proof}
	The corollary follows from Lemma \ref{l:cG(R)} with $F=c_p\coloneqq\pi(a_p)$.
\end{proof}

\begin{cor}\label{p:X}
	Let $X=\Gr_{n,k,\R}$  where $n=2k\ge 4$.
	For $0\le p\le k$,
	let the cocycle  $\ctil_p=(\pi(a_p),\Theta)\in Z^1\AA$ be as in Lemma \ref{l:H1:n=2k}.
	Write $_p X=\hs_{\ctil_p}X$.
	
	\begin{enumerate}
		\item[\rm (i)] The set of $\R$-points $(\hs_p X)(\R)$
		is in a canonical $(_p G)(\R)$-equivariant bijection with the {\em isotropic Grassmannian}
		$\GrI(n,k,p)$, that is, the set of  $k$-dimensional  subspaces
		$W\subset V=\C^n$ that are totally  isotropic
		with respect to the Hermitian form
		\[\mcH_p(x,y)=\bar x^t\cdot c_p\cdot y\]
		where $x,y\in\C^n$ are $k$-dimensional column vectors.
		
		\item[\rm (ii)] This set $(\hs_p X)(\R)$ is non-empty if and only if $p=k$.
	\end{enumerate}
\end{cor}

Here we say that $W$ is totally isotropic with respect to $\mcH_p$  if  $\mcH_p(x,y)=0$ for all $x,y\in W$.

\begin{proof}
	Assertion (i) follows from Lemma \ref{l:cX(R)} with $F=c_p$.  Note that $(c_p)^{-1}=c_p$\hs.
	
	If $p<k$, then the Hermitian form $\mcH_p$ admits no $k$-dimensional isotropic subspaces.
	If $p=k$, then $\mcH_k$ admits a $k$-dimensional isotropic subspace
with basis $e^1-e_{k+1},\  e_2-e_{k+2},\ \dots,\  e_k-e_{2k}$.
This proves (ii).
\end{proof}

\begin{rem}
	By the Witt theorem for Hermitian forms, the group $(\hs_k G)(\R)=\PU(k,k)$
	acts on the isotropic Grassmannian $(\hs_k X)(\R)=\GrI(k,k)$ transitively;
	see, for instance, Dieudonn\'e \cite[Assertion (3) in Section I.11]{Dieudonne}.
\end{rem}

\begin{subsec}\label{append matrix b_k}
	Set
	\[b_k=\ii\begin{pmatrix}0& -E_k\\ E_k,&0\end{pmatrix} \in\GL(2k,\C)\quad\ \text{where}\ \,\ii^2=-1. \]
	Then
	\[\bar b_k^t=b_k\hs, \ \,b_k^{-1}=b_k\hs.\]
	Set
	\[d_k=\pi(b_k)\in\PGL(2k,\C), \ \ \db_k=(d_k,\Theta)\in Z^1(\Gamma,\AA).\]
	Consider the Hermitian form $\cF$ on $\C^{2k}$ given by
	\[ \cF(x,y)=\bar x\hs^t\hs b_k\hs y\quad\ \text{for}\ \, x,y\in \C^{2k}.\]
\end{subsec}

\begin{cor}\label{cor iso gr}
	\begin{enumerate}
		\item[\rm (i)] $(\hs_{\db_k}\hm G)(\R)=\{g\in \PGL(2k,\C)\ |\ g\hs d_k\hs \bar g^t=d_k\}\hs.$
		\item[\rm (ii)]  $(\hs_{\db_k}\hm X)(\R)=\{W\in \Gr_{2k,k}(\C)\ |\ \cF|_W=0\}.$
	\end{enumerate}
	The set $(\hs_{\db_k}\hm X)(\R)$ is nonempty: it contains the $k$-dimensional subspace $W_0$ with basis $\{e_1\hs,\dots,e_k\}$.
\end{cor}

\begin{proof}
	Assertion (i) follows from Lemma \ref{l:cG(R)},
	and assertion (ii) follows from Lemma \ref{l:cX(R)}.
\end{proof}

\noindent

\noindent
E.~V.: Departamento de Matem{\'a}tica, Instituto de Ci{\^e}ncias Exatas,
Universidade Federal de Minas Gerais,
Av. Ant{\^o}nio Carlos, 6627, CEP: 31270-901, Belo Horizonte,
Minas Gerais, BRAZIL,

\noindent email: {\tt VishnyakovaE\symbol{64}googlemail.com}
\bigskip

\noindent M.~B.: Raymond and Beverly Sackler School of Mathematical Sciences,
Tel Aviv University, 6997801 Tel Aviv, ISRAEL,

\noindent email: {\tt borovoi@tauex.tau.ac.il}

\end{document}